%% file: cohenmirebeau.tex
\documentclass[10pt]{article}
\usepackage{amsmath}
\usepackage{amsfonts}
\usepackage{amssymb}

\newtheorem{theorem}{Theorem}[section]
\newtheorem{remark}[theorem]{Remark}
\newtheorem{lemma}[theorem]{Lemma}
\newtheorem{proposition}[theorem]{Proposition}

\usepackage{mathptmx}       
\usepackage{helvet}         
\usepackage{courier}        
\usepackage{type1cm}        
\usepackage{makeidx}         
\usepackage{graphicx}        
\usepackage{multicol}        
\usepackage[bottom]{footmisc}

\def\PP{\rm \hbox{I\kern-.2em\hbox{P}}}
\def\HH{\rm \hbox{I\kern-.2em\hbox{H}}}
\def\RR{\rm \hbox{I\kern-.2em\hbox{R}}}
\def\NN{\rm \hbox{I\kern-.2em\hbox{N}}}
\def\ZZ{\rm {{\rm Z}\kern-.28em{\rm Z}}}
\def\CC{\rm \hbox{C\kern -.5em {\raise .32ex \hbox{$\scriptscriptstyle
|$}}\kern-.22em{\raise .6ex \hbox{$\scriptscriptstyle |$}}\kern .4em}}
\def\vp{\varphi}
\def\<{\langle}
\def\>{\rangle}
\def\t{\tilde}

\def\e{\varepsilon}
\def\sm{\setminus}
\def\nl{\newline}

\def\bq{{\bf q}}
\def\cT{{\cal T}}
\def\cA{{\cal A}}
\def\cM{{\cal M}}

\def\cD{{\cal D}}

\def\cP{{\cal P}}

\def\cO{{\cal O}}
\def\H{{\rm \hbox{I\kern-.2em\hbox{H}}}}
\def\Chi{\raise .3ex
\hbox{\large $\chi$}} \def\vp{\varphi}
\def\lsima{\hbox{\kern -.6em\raisebox{-1ex}{$~\stackrel{\textstyle<}{\sim}~$}}\kern -.4em}
\def\lsim{\hbox{\kern -.2em\raisebox{-1ex}{$~\stackrel{\textstyle<}{\sim}~$}}\kern -.2em}
\def\({\Bigl (}
\def\){\Bigr )}
\def\({\Bigl (}
\def\){\Bigr )}

\def\nl{\newline}

\newcommand{\be}{\begin{equation}}
\newcommand{\ee}{\end{equation}}
\newcommand{\bea}{$$ \begin{array}{lll}}
\newcommand{\eea}{\end{array} $$}
\newcommand{\bi}{\begin{itemize}}
\newcommand{\ei}{\end{itemize}}
\newcommand{\iref}[1]{(\ref{#1})}

\def\proof{{\noindent \bf Proof: }}

\def\ve{\varepsilon}

\def\cO{\mathcal O}
\def\R{\mathbb R}

\def\gsim{\hbox{\kern -.2em\raisebox{-1ex}{$~\stackrel{\textstyle>}{\sim}~$}}\kern -.2em}

\usepackage{amsmath}
\usepackage{amsfonts}
\usepackage{amssymb}

\usepackage{psfrag}
\usepackage{graphicx}

\makeindex             

\begin{document}

\title{Adaptive and anisotropic piecewise polynomial approximation}
\author{Albert Cohen and Jean-Marie Mirebeau}

%
%
\maketitle


\abstract{We survey the main results of approximation theory for
adaptive piecewise polynomial functions.
In such methods, the partition on which the
piecewise polynomial approximation is defined is
not fixed in advance, but adapted to the given function $f$
which is approximated. We focus our discussion
on (i) the properties that describe an optimal partition 
for $f$, (ii) the smoothness properties of $f$ that govern
the rate of convergence of the approximation in the $L^p$-norms,
and (iii)  fast refinement algorithms that generate near optimal partitions.
While these results constitute a fairly established theory in the univariate
case and in the multivariate case when dealing with elements of isotropic
shape, the approximation theory for adaptive and anisotropic elements
is still building up. We put a particular emphasis on some recent
results obtained in this direction.
}

\section{Introduction}
\label{CMintro}

\subsection{Piecewise polynomial approximation}

Approximation by piecewise polynomial functions
is a procedure that occurs in numerous applications. 
In some of them such as 
terrain data simplification or image compression,
the function $f$ to be approximated might be fully known,
while it might be only partially known or fully
unknown in other applications such as denoising, statistical learning or 
in the finite element discretization of PDE's.
In all these applications, one usually makes the 
distinction between {\it uniform} and {\it adaptive} approximation.
In the uniform case, the domain of interest is decomposed into
a partition where all elements have comparable shape and size,
while these attributes are allowed to vary strongly in the adaptive case.
The partition may therefore be adapted to the local properties
of $f$, with the objective of 
optimizing the trade-off between accuracy and complexity of
the approximation. This chapter is concerned with the following fundamental
questions:
\begin{itemize}
\item
Which mathematical properties describe an optimally adapted 
partition for a given function $f$ ?
\item
For such optimally adapted partitions, what smoothness properties of $f$ govern the 
convergence properties of the corresponding piecewise polynomial approximations ?
\item
Can one construct optimally adapted 
partitions for a given function $f$ by a fast algorithm ?
\end{itemize}
For a given bounded domain $\Omega\subset \RR^d$
and a fixed integer $m>0$, we associate to any partition
$\cT$ of $\Omega$ the space 
$$
V_\cT:=\{f\;\;{\rm s.t.}\;\; f_{|T}\in\PP_{m-1},\; T\in\cT\}
$$
of piecewise polynomial functions of total degree $m-1$ over $\cT$.
The dimension of this space measures the complexity
of a function $g\in V_\cT$. It is proportional to the cardinality
of the partition:
$$
{\rm dim}(V_\cT):=C_{m,d}\#(\cT), \mbox{ with }
C_{m,d}:={\rm dim}(\PP_{m-1})=\binom{m+d-1}{d}.
$$
In order to describe how accurately a given function $f$
may be described by piecewise polynomial functions
of a prescribed complexity, it is therefore 
natural to introduce the error of best approximation
in a given norm $\|\cdot\|_X$ which is defined as 
$$
\sigma_N(f)_X:=\inf_{\#(\cT)\leq N}\min_{g\in V_{\cT}}\|f-g\|_X.
$$
This object of study is too vague if we do not make some
basic assumptions that limitate the set of partitions which 
may be considered. We therefore restrict the definition of
the above infimum to a class $\cA_N$
of ``admissible partitions'' of complexity at most $N$. The
approximation to $f$ is therefore searched in the set 
$$
\Sigma_N:=\cup_{\cT\inÊ\cA_N} V_{\cT},
$$
and the error of best approximation is now defined as
$$
\sigma_N(f)_X:=\inf_{g\in \Sigma_N}\|f-g\|_X=\inf_{\cT\in\cA_N}\inf_{g\in V_\cT}\|f-g\|_{X}.
$$
The assumptions which define the class $\cA_N$ are
usually of the following type:
\begin{enumerate}
\item
The elementary {\it geometry} of the
elements of $\cT$. The typical examples that are
considered in this chapter are: intervals when $d=1$, 
triangles or rectangles when $d=2$, simplices
when $d>2$. \vspace{0.2cm}
\item
Restrictions on the {\it regularity} of the partition, in the sense of
the relative size and shape of the elements that constitute
the partition $\cT$. \vspace{0.2cm}
\item
Restrictions on the {\it conformity} of the partition, which impose
that each face of an element $T$
is common to at most one adjacent element $T'$. 
\end{enumerate}
The conformity restriction is critical when imposing global
continuity or higher smoothness 
properties in the definition of
$V_\cT$, and if one wants to measure the
error in some smooth norm. In this survey, we 
limitate our interest to the approximation error
measured in $X=L^p$. We therefore do not
impose any global smoothness property on the space $V_\cT$
and ignore the conformity requirement. 

Throughout this chapter, we use the notation
$$
e_{m,\cT}(f)_p:=\min_{g\in V_\cT}\|f-g\|_{L^p},
$$
to denote the $L^p$ approximation error in the space $V_\cT$ and
$$
\sigma_N(f)_p:=\sigma_N(f)_{L^p}=\inf_{g\in \Sigma_N}\|f-g\|_{L^p}=
\inf_{\cT\in\cA_N}e_{m,\cT}(f)_p.
$$
If $T\in \cT$ is an element and $f$ is a function
defined on $\Omega$, we denote by
$$
e_{m,T}(f)_p:=\min_{\pi\in \PP_{m-1}} \|f-\pi\|_{L^p(T)},
$$
the local approximation error. We thus have
$$
e_{m,\cT}(f)_p=\(\sum_{T\in\cT}e_{m,T}(f)_p^p\)^{1/p},
$$
when $p<\infty$ and
$$
e_{m,\cT}(f)_\infty=\max_{T\in\cT}e_{m,T}(f)_\infty.
$$
The norm $\|f\|_{L^p}$ without precision on the domain
stands for $\|f\|_{L^p(\Omega)}$ where $\Omega$ is the full
domain where $f$ is defined.

\subsection{From uniform to adaptive approximation}

Concerning the restrictions ont the regularity of the
partitions, three situations should be distinguished:
\begin{enumerate}
\item
{\it Quasi-uniform partitions}: all elements have approximately
the same size. This may be expressed by a restriction of the type
\be
C_1 N^{-1/d} \leq \rho_T \leq h_T \leq C_2 N^{-1/d},
\label{CMunifmesh}
\ee
for all $T\in \cT$ with $\cT \inÊ\cA_N$, where $0<C_1\leq C_2$ are 
constants independent of $N$, and where $h_T$ and $\rho_T$ respectively denote
the diameters of $T$ and of it largest inscribed disc. \vspace{0.2cm}
\item
{\it Adaptive isotropic partitions}: elements may have
arbitrarily different size but their aspect ratio is controlled
by a restriction of the type
\be
\frac {h_T}{\rho_T} \leq C,
\label{isotropy}
\ee
for all $T\in \cT$ with $\cT \inÊ\cA_N$,  where $C>1$ is
independent of $N$. \vspace{0.2cm}
\item
{\it Adaptive anisotropic partitions}: element may
have arbitrarily different size and aspect ratio,
i.e. no restriction is made on $h_T$ and $\rho_T$.
\end{enumerate}
A classical result states that if a function $f$ belongs to the
Sobolev space $W^{m,p}(\Omega)$ the
$L^p$ error of approximation by piecewise polynomial of degree
$m$ on a given partition satisfies the estimate
\be
e_{m,\cT}(f)_p\leq Ch^{m}|f|_{W^{m,p}},
\ee
where $h:=\max_{T\in \cT} h_T$ is the maximal mesh-size, 
$|f|_{W^{m,p}}:=\(\sum_{|\alpha|=m}\|\partial^\alpha f\|_{L^p}^p\)^{1/p}$
is the standard Sobolev semi-norm, and $C$ is a constant that only depends on $(m,d,p)$.
In the case of quasi-uniform partitions, this yields an estimate
in terms of complexity:
\be
\sigma_N(f)_p\leq CN^{-m/d}|f|_{W^{m,p}},
\label{CMsigmaNunif}
\ee
where the constant $C$ now also depends on $C_1$ and $C_2$ in \iref{CMunifmesh}.
\nl
\nl
{\it Here and throughout the chapter, ${\rm C}$ denotes a generic constant which may
vary from one equation to the other. The dependence of this constant with respect to the relevant
parameters will be mentionned when necessary.}
\nl
\nl
Note that the above estimate can be achieved by restricting the family $\cA_N$
to a single partition: for example, we start from a 
coarse partition $\cT_0$ into cubes and recursively define a nested sequence of
partition $\cT_j$ by splitting each cube of $\cT_{j-1}$ into $2^d$ cubes of half side-length.
We then set
$$
\cA_N:=\{\cT_j\}, \mbox{ if }\#(\cT_0)2^{dj}\leq N< \#(\cT_0)2^{d(j+1)}.
$$
Similar uniform refinement rules can be proposed for  more general partitions into triangles,
simplices or rectangles. With such a choice for $\cA_N$, the set $\Sigma_N$
on which one picks the approximation is thus
a standard linear space. Piecewise polynomials
on quasi-uniform partitions may therefore be considered 
as an instance of {\it linear approximation}.

The interest of adaptive partitions is that the choice of $\cT\in \cA_N$
may vary depending on $f$, so that the set $\Sigma_N$
is inherently a nonlinear space. Piecewise polynomials
on adaptive partitions are therefore an instance of {\it nonlinear approximation}.
Other instances include approximation by rational
functions, or by $N$-term linear combinations
of a basis or dictionary. We refer to \cite{CMde} for a general survey
on nonlinear approximation. 

The use of adaptive partitions allows to improve
significantly on \iref{CMsigmaNunif}. The theory
that describes these improvements is rather well established
for adaptive isotropic partitions: as explained
further, a typical result for such partitions is of the form
\be
\sigma_N(f)_p\leq CN^{-m/d}|f|_{W^{m,\tau}},
\label{CMsigmaNiso}
\ee
where $\tau$ can be chosen smaller than $p$.
Such an estimate reveals that the same rate of decay
$N^{-\frac m d}$ as in \iref{CMsigmaNunif} is achieved for $f$ 
in a smoothness space which is larger than $W^{m,p}$. It also says that 
for a smooth function, the multiplicative constant governing this
rate might be substantially smaller than when 
working with quasi-uniform partitions. 

When allowing adaptive anisotropic partitions, one
should expect for further improvements. From an
intuitive point of view, such partitions are needed
when the function $f$ itself displays locally
anisotropic features such as jump discontinuities 
or sharp transitions along smooth manifolds. 
The available approximation theory for such partitions 
is still at its infancy. Here, typical estimates are
also of the form
\be
\sigma_N(f)_p\leq CN^{-m/d}A(f),
\label{CMsigmaNaniso}
\ee
but they involve quantities $A(f)$ which are 
not norms or semi-norms associated with standard
smoothness spaces. These quantities are
highly nonlinear in $f$ in the sense that they
do not satisfy $A(f+g)\leq C(A(f)+A(g))$ even
with $C\geq 1$.

\subsection{Outline}

This chapter is organized as follows. As a starter, we study
in \S 2 the simple case of piecewise constant approximation on
an interval. This example gives a first illustration the difference
between the approximation properties of uniform and adaptive partitions.
It also illustrates the principle of {\it error equidistribution} 
which plays a crucial role in the 
construction of adaptive partitions which are optimally
adapted to $f$. This leads us to propose and
study a {\it multiresolution greedy refinement algorithm} as a design tool for such partitions. 
The distinction between isotropic and anisotropic partitions
is irrelevant in this case, since we work with one-dimensional intervals.

We discuss in \S 3 the derivation of estimates of the
form \iref{CMsigmaNiso} for adaptive isotropic partitions.
The main guiding principle for the design of the partition
is again error equidistribution. Adaptive greedy refinement
algorithms are discussed, similar to the one-dimensional case.

We study in \S 4 an elementary case of 
adaptive anisotropic partitions for which all elements
are two-dimensional rectangles with sides that are parallel
to the $x$ and $y$ axes. This type of anisotropic partitions
suffer from an intrinsic lack of directional selectivity. 
We limitate our attention to piecewise constant functions, and identify the
quantity $A(f)$ involved in \iref{CMsigmaNaniso} for this particular case.
The main guiding principles for the design of the optimal partition
are now error equidistribution combined with a local {\it shape optimization}
of each element. 

In \S 5, we present some recently available theory for
piecewise polynomials on adaptive anisotropic
partitions into triangles (and simplices in dimension $d>2$)
which offer more directional selectivity than the previous
example. We give a general formula for
the quantity $A(f)$ which can be turned into an explicit
expression in terms of the derivatives of $f$
in certain cases such as piecewise linear functions i.e. $m=2$.
Due to the fact that $A(f)$ is not a semi-norm,
the function classes defined by 
the finiteness of $A(f)$
are not standard smoothness spaces.
We show that these classes include piecewise smooth objects
separated by discontinuities or sharp transitions
along smooth edges.

We present in \S 6 several greedy refinement algorithms which
may be used to derive anisotropic partitions. The convergence
analysis of these algorithms is more delicate than for their
isotropic counterpart, yet some first results indicate that
they tend to generate optimally adapted partitions which
satisfy convergence estimates in accordance with \iref{CMsigmaNaniso}.
This behaviour is illustrated by numerical tests on two-dimensional functions.

\section{Piecewise constant one-dimensional approximation}

We consider here the very simple problem of approximating a
continuous function by piecewise constants on the unit interval
$[0,1]$, when we measure the error in the uniform norm.
If $f\in C([0,1])$ and $I\subset [0,1]$ is an arbitrary interval
we have
$$
e_{1,I}(f)_\infty:=\min_{c\in \RR}\|f-c\|_{L^\infty(I)} = \frac 1 2 \max_{x,y\in I}|f(x)-f(y)|.
$$
The constant $c$ that achieves the minimum is the median of $f$ on $I$. Remark that
we multiply this estimate at most by a factor $2$ if we take $c=f(z)$ for any $z\in I$.
In particular, we may choose for $c$ the average of $f$ on $I$ which is still defined
when $f$ is not continuous but simply integrable.

If $\cT_N=\{I_1,\cdots,I_N\}$ is a partition of $[0,1]$ into $N$ sub-intervals
and $V_{\cT_N}$ the corresponding space of piecewise constant functions,
we thus find hat 
\be
e_{1,\cT_N}(f)_\infty:=\min_{g\in V_{\cT_N}} \|f-g\|_{L^\infty}=\frac 1 2\max_{k=1,\cdots,N} \max_{x,y\in I_k}|f(x)-f(y)|.
\label{CMbasicconst}
\ee
\subsection{Uniform partitions}

We first study the error of approximation 
when the $\cT_N$ are uniform partitions
consisting of the intervals $I_k=[\frac k N,\frac {(k+1)} N]$.
Assume first that $f$ is a Lipschitz function i.e. $f'\in L^\infty$.
We then have
$$
\max_{x,y\in I_k}|f(x)-f(y)|\leq |I_k| \|f'\|_{L^\infty(I_k)}=N^{-1} \|f'\|_{L^\infty}.
$$
Combining this estimate with \iref{CMbasicconst}, we find that for uniform partitions,
\be
f\in {\rm Lip}([0,1]) \Rightarrow \sigma_N(f)_\infty \leq  CN^{-1},
\label{CMlipunif}
\ee
with $C=\frac 1 2\|f'\|_{L^\infty}$. For less smooth functions, we may obtain
lower convergence rates: if $f$ is 
H\"older continuous of exponent $0<\alpha<1$, we have by definition
$$
|f(x)-f(y)|\leq |f|_{C^\alpha} |x-y|^\alpha,
$$
which yields
$$
\max_{x,y\in I_k}|f(x)-f(y)|\leq N^{-\alpha}|f|_{C^\alpha}.
$$
We thus find that
\be
f\in C^\alpha([0,1])  \Rightarrow \sigma_N(f)_\infty \leq  CN^{-\alpha},
\label{CMholdunif}
\ee
with $C=\frac 1 2|f|_{C^\alpha}$.

The estimates \iref{CMlipunif} and \iref{CMholdunif} are sharp in the 
sense that they admit a converse: it is easily checked that if $f$ is a continuous
function such that  $\sigma_N(f)_\infty \leq  CN^{-1}$ for some $C>0$, it is
necessarily Lipschitz. Indeed, for any $x$ and $y$ in $[0,1]$, consider
an integer $N$ such that $\frac 1 2 N^{-1} \leq |x-y| \leq N^{-1}$. For 
such an integer, there exists a $f_N\in V_{\cT_N}$ such that
$\|f-f_N\|_{L^\infty}\leq CN^{-1}$. We thus have
$$
|f(x)-f(y)| \leq 2CN^{-1}+|f_N(x)-f_N(y)|.
$$
Since $x$ and $y$ are either contained in one interval or two
adjacent intervals of the partition $\cT_N$ and since $f$ is continuous, we find
that $|f_N(x)-f_N(y)|$ is either zero or less than $2CN^{-1}$. 
We therefore have
$$
|f(x)-f(y)| \leq 4CN^{-1}\leq 8C|x-y|,
$$
which shows that $f\in {\rm Lip}([0,1])$. In summary, we have the following result.
\begin{theorem}
If $f$ is a continuous function defined on $[0,1]$ and if $\sigma_N(f)_\infty$
denotes the $L^\infty$ error of piecewise constant approximation on 
uniform partitions, we have
\be
f\in {\rm Lip}([0,1]) \Leftrightarrow \sigma_N(f)_\infty \leq  CN^{-1}.
\label{CMlipunifeq}
\ee
\end{theorem}
In an exactly similar way, is can be proved that
\be
f\in C^\alpha([0,1])  \Leftrightarrow \sigma_N(f)_\infty \leq  CN^{-\alpha},
\label{CMholdunifeq}
\ee
These equivalences reveal that Lipschitz and Holder smoothness
are the properties that do govern the rate of approximation
by piecewise constant functions in the uniform norm.

The estimate \iref{CMlipunif} is also optimal in the
sense that it describes the {\it saturation rate} of
piecewise constant approximation: 
a higher convergence rate cannot be obtained, even
for smoother functions, and the constant
$C=\frac 1 2\|f'\|_{L^\infty}$ cannot be improved. In order to see this, consider
an arbitrary function $f\in C^1([0,1])$, so that for all 
$\e>0$, there exists $\eta>0$ such that
$$
|x-y|\leq \eta \Rightarrow |f'(x)-f'(y)|\leq \e.
$$
Therefore if $N$ is such that $N^{-1}\leq \eta$, we can introduce
on each interval $I_k$ an affine function
$p_k(x)=f(x_k)+(x-x_k)f'(x_k)$ where $x_k$ is an arbitrary point
in $I_k$, and we then have
$$
\|f-p_k\|_{L^\infty(I_k)}\leq N^{-1}\e.
$$
It follows that
$$
\begin{array}{ll}
e_{1,I_k}(f)_\infty &\geq e_{1,I_k}(p_k)_\infty-e_{1,I_k}(f-p_k)_\infty\\
& \geq e_{1,I_k}(p_k)_\infty-\frac 1 2 N^{-1} \e \\
& =\frac 1 2 N^{-1}(|f'(x_k)|-\e),
\end{array}
$$
where we have used the triangle inequality
\be
e_{m,T}(f+g)_p \leq e_{m,T}(f)_p+e_{m,T}(g)_p,
\label{trierror}
\ee
Choosing for $x_k$ the point that maximize $|f'|$ on $I_k$ and taking the supremum 
of the above estimate over all $k$, we obtain
$$
e_{1,\cT_N}(f)_\infty \geq \frac 1 2 N^{-1}(\|f'\|_{L^\infty}-\e).
$$
Since $\e>0$ is arbitrary, this implies the lower estimate
\be
\liminf_{N\to +\infty} \;N\sigma_N(f)_{\infty} \geq \frac 1 2  \|f'\|_{L^\infty}.
\label{CMsaturationpwconstunif}
\ee
Combining with the upper estimate \iref{CMlipunif}, we thus obtain the equality
\be
\lim_{N\to +\infty} N\sigma_N(f)_{\infty} = \frac 1 2  \|f'\|_{L^\infty},
\ee
for any function $f\in C^1$. This identity shows that for smooth enough
functions, the numerical quantity that governs the rate of convergence $N^{-1}$
of uniform piecewise constant approximations is exactly $\frac 1 2  \|f'\|_{L^\infty}$.

\subsection{Adaptive partitions}

We now consider an adaptive partition $\cT_N$ for which the
intervals $I_k$ may depend on $f$. In order to understand the 
gain in comparison to uniform partitions, let us consider a 
function $f$ such that $f'\in L^1$, i.e. 
$f\in W^{1,1}([0,1])$. Remarking that
$$
\max_{x,y\in I}|f(x)-f(y)| \leq \int_I |f'(t)|dt,
$$
we see that a natural choice fo the $I_k$ can be done by 
imposing that
$$
\int_{I_k} |f'(t)|dt= N^{-1}\int_{0}^1 |f'(t)|dt,
$$
which means that the $L^1$ norm of $f'$ is equidistributed over all intervals. 
Combining this estimate with \iref{CMbasicconst}, we find that for adaptive partitions,
\be
f\in W^{1,1}([0,1]) \Rightarrow \sigma_N(f)_\infty \leq  CN^{-1},
\label{CMconstadapt}
\ee
with $C:=\frac 1 2 \|f'\|_{L^1}$. 
This improvement
upon uniform partitions in terms of approximation properties 
was firstly established in \cite{CMka}. The above argument
may be extended to the case where $f$ belongs to 
the slightly larger space $BV([0,1])$ which may
include discontinuous functions in contrast to $W^{1,1}([0,1])$, 
by asking that the $I_k$ are such that
$$
|f|_{BV(I_k)}\leq N^{-1}|f|_{BV}.
$$
We thus have
\be
f\in BV([0,1]) \Rightarrow \sigma_N(f)_\infty \leq  CN^{-1},
\label{CMconstadaptbv}
\ee
Similar to the case of uniform partitions, the estimate \iref{CMconstadaptbv}
is sharp in the sense that a converse result holds: if
$f$ is a continuous function such that $\sigma_N(f)_\infty \leq  CN^{-1}$ for some $C>0$,
then it is necessarily in $BV([0,1])$. To see this, consider $N>0$ and 
any set of points $0\leq x_1<x_2<\cdots<x_N\leq 1$. We know
that there exists a partition $\cT_N$ of $N$
intervals and $f_N\in V_{\cT_N}$ such that $\|f-f_N\|_{L^\infty}\leq CN^{-1}$.
We define a set of points $0\leq y_1<y_2 \cdots <y_M\leq 1$ by unioning the 
set of the $x_k$ with the nodes that define the partition $\cT_N$, excluding $0$ and $1$,
so that $M<2N$. We can write
$$
\sum_{k=0}^{N-1} |f(x_{k+1})-f(x_k)| \leq 2C+ \sum_{k=0}^{N-1}|f_N(x_{k+1})-f_N(x_k)|
\leq 2C+ \sum_{k=0}^{M-1}|f_N(y_{k+1})-f_N(y_k)|.
$$
Since $y_k$ and $y_{k+1}$ are either contained in one interval or two
adjacent intervals of the partition $\cT_N$ and since $f$ is continuous, we find
that $|f_N(y_{k+1})-f_N(y_k)|$ is either zero or less than $2CN^{-1}$, from which it
follows that
$$
\sum_{k=0}^{N-1} |f(x_{k+1})-f(x_k)|\leq 6C,
$$
which shows that $f$ has bounded variation.
We have thus proved the following result.
\begin{theorem}
If $f$ is a continuous function defined on $[0,1]$ and if $\sigma_N(f)_\infty$ denotes
the $L^\infty$ error of piecewise constant approximation
on adaptive partitions, we have
\be
f\in BV([0,1]) \Leftrightarrow \sigma_N(f)_\infty \leq  CN^{-1}.
\label{CMconstadapteq}
\ee
\end{theorem}
In comparison with \iref{CMlipunif} we thus find that
same rate $N^{-1}$ is governed by a {\it weaker}
smoothness condition since $f'$ is not assumed
to be bounded but only a finite measure. In turn, adaptive partitions
may significantly outperform
uniform partition for a given function $f$: consider for instance the
function $f(x)=x^\alpha$ for some $0<\alpha<1$. According to 
\iref{CMholdunifeq}, the convergence rate of uniform approximation
for this function is $N^{-\alpha}$. On the other hand, since $f'(x)=\alpha x^{\alpha-1}$ 
is integrable, we find that the convergence rate of adaptive approximation
is $N^{-1}$.

The above construction of an adaptive partition is based on equidistributing the
$L^1$ norm of $f'$ or the total variation of $f$ on each interval $I_k$. 
An alternative is to build $\cT_N$ in such a way that all local errors
are equal, i.e.
\be
\e_{1,I_k}(f)_\infty=\eta,
\label{CMequidistrunif}
\ee
for some $\eta=\eta (N)\geq 0$ independent of $k$. This new
construction of $\cT_N$ does not require that $f$ belongs to $BV([0,1])$.
In the particular case
where $f\in BV([0,1])$, we obtain that
$$
N\eta \leq \sum_{k=1}^N\\e_{1,I_k}(f)_\infty \leq \frac 1 2\sum_{k=1}^N|f|_{BV(I_k)}\leq \frac 1 2|f|_{BV},
$$
from which it immediately follows that
$$
e_{1,\cT_N}(f)_\infty=\eta \leq CN^{-1},
$$
with $C=\frac 1 2|f|_{BV}$.
We thus have obtained the same
error estimate as with the previous construction of $\cT_N$.
\nl
\nl
{\it The basic principle of error equidistribution, which is expressed by
{\rm \iref{CMequidistrunif}} in the case of piecewise constant approximation
in the uniform norm, plays a central role in the derivation of 
adaptive partitions for piecewise polynomial approximation.}
\nl
\nl
Similar to the case of uniform partitions
we can express the optimality of \iref{CMconstadapt} by
a lower estimate when $f$ is smooth enough. For this
purpose, we make a slight restriction on the set 
$\cA_N$ of admissible partitions, assuming that
the diameter of all intervals decreases 
as $N\to +\infty$, according to 
$$
\max_{I_k\in \cT_N} |I_k| \leq AN^{-1},
$$
for some $A>0$ which may be arbitrarily large.
Assume that $f\in C^1([0,1])$, so that for all 
$\e>0$, there exists $\eta>0$ such that
\be
|x-y|\leq \eta \Rightarrow |f'(x)-f'(y)|\leq \frac \e A.
\label{CMrestrict}
\ee
If $N$ is such that $AN^{-1}\leq \eta$, we can introduce
on each interval $I_k$ an affine function
$p_k(x)=f(x_k)+(x-x_k)f'(x_k)$ where $x_k$ is an arbitrary point
in $I_k$, and we then have
$$
\|f-p_k\|_{L^\infty(I_k)}\leq N^{-1} \e.
$$
It follows that
$$
\begin{array}{ll}
e_{1,I_k}(f)_\infty &\geq e_{1,I_k}(p_k)_\infty-e_{1,I_k}(f-p_k)_\infty\\
& \geq e_{1,I_k}(p_k)_\infty-\frac 1 2 N^{-1} \e \\
& =\frac 1 2 (\int_{I_k}|p'_k(t)|dt-N^{-1}\e)Ê\\
& \geq \frac 1 2 (\int_{I_k}|f'(t)|dt-2N^{-1}\e).
\end{array}
$$
Since there exists at least one interval $I_k$ such that 
$\int_{I_k}|f'(t)|dt\geq N^{-1} \|f'\|_{L^1}$, it follows that
$$
e_{1,\cT_N}(f)_\infty\geq \frac 1 2 N^{-1}(\|f'\|_{L^1}-2\e).
$$
This inequality becomes an equality only when all 
quantities $\int_{I_k}|f'(t)|dt$ are equal, which justifies the
equidistribution principle for the design of an optimal partition.
Since $\e>0$ is arbitrary, we have thus obtained
the lower estimate
\be
\liminf_{N\to +\infty} \;N\sigma_N(f) \geq \frac 1 2  \|f'\|_{L^1}.
\label{CMsatadaptpwconst}
\ee
The restriction on the family of adaptive partitions $\cA_N$ is not
so severe since $A$ maybe chosen arbitrarily large. In particular,
it is easy to prove that the upper estimate is almost preserved in
the following sense: for a given $f\in C^1$ and 
any $\e>0$, there exists $A>0$ depending on $\e$ such that
$$
\limsup_{N\to +\infty} \;N\sigma_N(f) \leq \frac 1 2  \|f'\|_{L^1}+\e,
\label{CMsaturationpwconst}
$$
These results show that for smooth enough
functions, the numerical quantity that governs the rate of convergence $N^{-1}$
of adaptive piecewise constant approximations is exactly $\frac 1 2  \|f'\|_{L^1}$.
Note that $\|f'\|_{L^\infty}$ may be substantially larger than $\|f'\|_{L^1}$
even for very smooth functions, in which case adaptive partitions
performs at a similar rate as uniform partitions, but with
a much more favorable multiplicative constant.

\subsection{A greedy refinement algorithm}

The principle of error distribution suggests a simple algorithm for the generation of 
adaptive partitions, based on a greedy refinement algorithm:
\begin{enumerate}
\item
Initialization: $\cT_1=\{[0,1]\}$.\vspace{0.2cm}
\item
Given $\cT_N$ select $I_m\in\cT_N$ that maximizes the local error $e_{1,I_k}(f)_\infty$.\vspace{0.2cm}
\item
Split $I_m$ into two sub-intervals of equal size to obtain $\cT_{N+1}$ and return to step 2.
\end{enumerate}
The family $\cA_N$ of adaptive partitions that are generated by this algorithm
is characterized by the restriction that all intervals are of the dyadic type
$2^{-j}[n,n+1]$ for some $j\geq 0$ and $n\in \{0,\cdots,2^j-1\}$.
We also note that all such partitions $\cT_N$ may be identified
to a finite subtree with $N$ leaves, picked within an infinite
dyadic {\it master tree} $\cM$ in which each node represents a dyadic interval.
The root of $\cM$ corresponds to $[0,1]$ and each node $I$ of generation
$j$ corresponds to an interval of length $2^{-j}$ 
which has two {\it children} nodes of generation $j+1$ corresponding 
to the two halves of $I$. This identification, which is
illustrated on Figure \ref{CMtreepart}, is useful 
for coding purposes since any such subtree can 
be encoded by $2N$ bits.

\begin{figure}
\centerline{
\includegraphics[width=10cm,height=2cm]{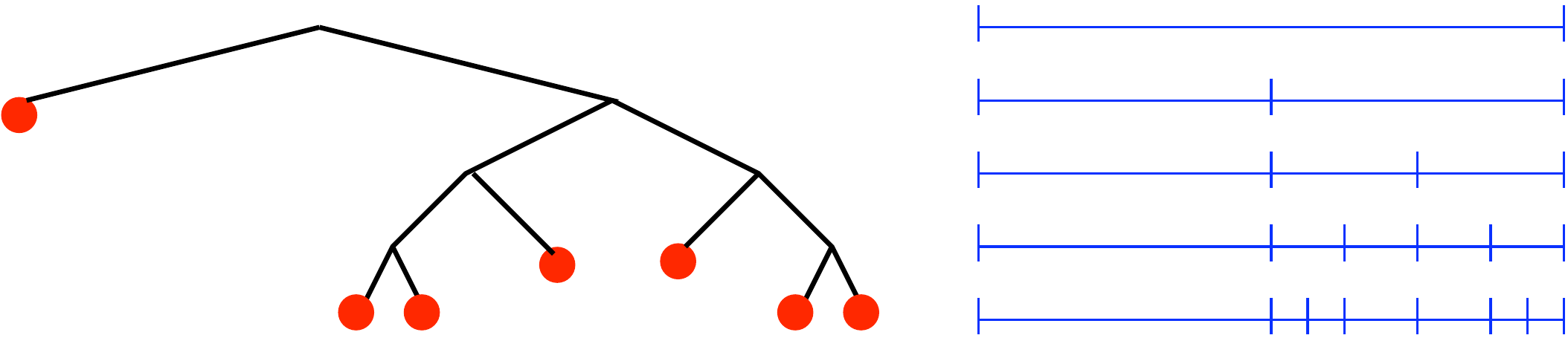}
}
\caption{Adaptive dyadic partitions identify to dyadic trees}
\label{CMtreepart}
\end{figure}

We now want to understand how the approximations generated by adaptive refinement algorithm behave
in comparison to those associated with the optimal partition. In particular, do we also have
that $e_{1,\cT_N}(f)_\infty \leq C N^{-1}$ when $f'\in L^1$ ? The 
answer to this question turns out to be negative, but it was
proved in \cite{CMdy} that a
slight strengthening of the smoothness assumption is
sufficient to ensure this convergence rate : we instead assume
that the {\it maximal function} of $f'$ is in $L^1$. We recall that
the maximal function of a locally integrable function $g$ is defined by
$$
M_g(x):=\sup_{r>0} |B(x,r)|^{-1}\int_{B(x,r)}|g(t)|dt,
$$
It is known that $M_g\in L^p$ if and only if $g\in L^p$ for $1<p<\infty$
and that $M_g\in L^1$ if and only if $g\in L\log L$, i.e. 
$\int_0^1 |g(t)|\log(1+|g(t)|)dt <\infty$, see \cite{CMstein}. In this sense, the assumption
that $M_{f'}$ is integrable is only slightly stronger than $f\in W^{1,1}$.
\nl
\nl
If $\cT_N:=(I_1,\cdots,I_N)$, define the 
accuracy
$$
\eta:=\max_{1\leq k\leq N}e_{1,I_k}(f)_\infty.
$$
For each $k$, we denote by $J_k$ the interval which is the {\it parent} of $I_k$ in the
refinement process. From the definition of the algorithm, we necessarily have
$$
\eta \leq \|f-a_{J_k}(f)\|_{L^\infty}Ê\leq \int_{J_k}|f'(t)|dt.
$$
For all $x\in I_k$, the ball $B(x,2|I_k|)$ contains $J_k$ and
it follows therefore that 
$$
M_{f'}(x)\geq  |B(x,2|I_k|)|^{-1}\int_{B(x,2|I_k|)}|f'(t)|dt
\geq [4|I_k|]^{-1}\eta,
$$
which implies in turn
$$
\int_{I_k}M_{f'}(t)dt \geq \eta/4.
$$
If $M_{f'}$ is integrable, this yields the estimate
$$
N\eta \leq  4\int_0^1M_{f'}(t)dt.
$$
It follows that 
$$
e_{1,\cT_N}(f)_\infty=\eta \leq CN^{-1}
$$
with $C=4\|M_{f'}\|_{L^1}$.  We
have thus established the following result.
\begin{theorem}
If $f$ is a continuous function defined on $[0,1]$
and if $\sigma_N(f)_\infty$ denotes the $L^\infty$
error of piecewise constant approximation
on adaptive partitions of dyadic type, we have
\be
M_{f'}\in L^1([0,1]) \Rightarrow \sigma_N(f)_\infty \leq  CN^{-1},
\label{CMconstgreedy}
\ee
and that this rate may be achieved by the above described greedy algorithm.
\end{theorem}
Note however that a converse to \iref{CMconstgreedy} does not hold and
that we do not so far know of a simple smoothness property that 
would be exactly equivalent to the rate of approximation $N^{-1}$ by
dyadic adaptive partitions. A by-product of \iref{CMconstgreedy} is that 
\be
f\in W^{1,p}([0,1]) \Rightarrow \sigma_N(f)_\infty \leq  CN^{-1},
\label{CMconstgreedyp}
\ee
for any $p>1$.

\section{Adaptive and isotropic approximation}

We now consider the problem of piecewise polynomial
approximation on a domain $\Omega\subset \RR^d$,
using adaptive and {\it isotropic} partitions. We therefore
consider a sequence $(\cA_N)_{N\geq 0}$
of families of partitions that satisfies the restriction
\iref{isotropy}. We use piecewise polynomials of
degree $m-1$ for some fixed but arbitrary $m$.

Here and in all the rest of the chapter,
we restrict our attention to partitions
into geometrically simple elements which are either 
cubes, rectangles or 
simplices. These simple elements satisfy a property
of {\it affine invariance}: there exist
a {\it reference element} $R$ such that
any $T\in\cT\in \cA_N$ is the
image of $R$ by an invertible
affine transformation $A_T$. We can
choose $R$ to be the unit
cube $[0,1]^d$
or the unit simplex $\{0\leq x_1\leq \cdots\leq x_d\leq 1\}$ in the case
of partitions by cubes and rectangles
or simplices, respectively.

\subsection{Local estimates}

If $T\in \cT$ is an element and $f$ is a function
defined on $\Omega$, we study the local approximation error
\be
e_{m,T}(f)_p:=\min_{\pi\in \PP_{m-1}} \|f-\pi\|_{L^p(T)}.
\label{CMlocalerror}
\ee
When $p=2$ the minimizing
polynomial is given by
$$
\pi:=P_{m,T}f,
$$
where $P_{m,T}$ is the $L^2$-orthogonal projection, and can
therefore be computed by solving a least square system. 
When $p\neq 2$, the minimizing polynomial
is generally not easy to determine. However
it is easily seen that the $L^2$-orthogonal projection
remains an acceptable choice: indeed, it can
easily be checked that the operator norm of $P_{m,T}$
in $L^p(T)$ is bounded by a
constant $C$ that only depends on $(m,d)$ but not
on the cube or simplex $T$. From this we infer that for all $f$
and $T$ one has
\be
e_{m,T}(f)_p\leq \|f-P_{m,T}f\|_{L^p(T)}\leq (1+C)e_{m,T}(f)_p.
\label{CMequivlocal}
\ee
Local estimates for $e_{m,T}(f)_p$ can be obtained 
from local estimates on the reference element $R$,
remarking that
\be
e_{m,T}(f)_p=\(\frac {|T|}{|R|}\)^{1/p} e_{m,R}(g)_p,
\label{CMcdvlp}
\ee 
where $g=f\circ A_T$. Assume that $p,\tau\geq 1$ are such that
$\frac 1 \tau=\frac 1 p + \frac m d$, and let $g\in W^{m,\tau}(R)$.
We know from Sobolev embedding that 
$$
\|g\|_{L^p(R)} \leq C\|g\|_{W^{m,\tau}(R)},
$$
where the constant $C$ depends on $p,\tau$ and $R$. 
Accordingly, we obtain
\be
e_{m,R}(g)_p\leq C\min_{\pi\in\PP_{m-1}}\|g-\pi\|_{W^{m,\tau}(R)}.
\label{CMlocalref}
\ee
We then
invoke Deny-Lions theorem which states that if $R$ is 
a connected domain, there exists
a constant $C$ that only depends on $m$ and $R$ such that
\be
\min_{\pi\in\PP_{m-1}}\|g-\pi\|_{W^{m,\tau}(R)}\leq C|g|_{W^{m,\tau}(R)}.
\label{CMdenylions}
\ee
If $g=f\circ A_T$, we obtain by this change of variable that
\be
|g|_{W^{m,\tau}(R)} \leq C\(\frac {|R|}{|T|}\)^{1/\tau} \|B_T\|^m |f|_{W^{m,\tau}(T)},
\label{CMcdvm}
\ee
where $B_T$ is the linear part of $A_T$ and $C$ is a constant
that only depends on $m$ and $d$.
A well known and easy to derive bound for $\|B_T\|$ is 
\be
\|B_T\|\leq \frac {h_T}{\rho_R},
\label{CMboundbt}
\ee
Combining \iref{CMcdvlp}, \iref{CMlocalref}, \iref{CMdenylions}, \iref{CMcdvm} and \iref{CMboundbt}, we thus
obtain a local estimate of the form
$$
e_{m,T}(f)_p\leq C|T|^{1/p-1/ \tau}\, h_T^m |f|_{W^{m,\tau}(T)}=C|T|^{-m/d} h_T^m |f|_{W^{m,\tau}(T)}.
$$
where we have used the relation $\frac 1 \tau=\frac 1 p+\frac m d$. 
From the isotropy restriction \iref{isotropy}, there 
exists a constant $C>0$ independent of $T$ such that 
$h_T^d\leq C|T|$. We have thus established the following 
local error estimate.

\begin{theorem}
\label{CMlocalestheo}
If $f\in W^{m,\tau}(\Omega)$, we have for all element $T$
\be
e_{m,T}(f)_p\leq C |f|_{W^{m,\tau}(T)},
\label{CMlocalinv}
\ee
where the constant $C$ only depends on $m$, $R$
and the constants in \iref{isotropy}.
\end{theorem}

Let us mention several useful generalizations
of the local estimate \iref{CMlocalinv} that can be obtained
by a similar approach based on a change of variable
on the reference element. First, if $f\in W^{s,\tau}(\Omega)$
for some $0<s\leq m$ and $\tau\geq 1$ such that $\frac 1 \tau=\frac 1 p+\frac s d$, 
we have
\be
e_{m,T}(f)_p\leq C |f|_{W^{s,\tau}(T)}.
\label{CMlocalinvs}
\ee
Recall that when $s$ is not an integer, the $W^{s,\tau}$ semi-norm is defined
by 
$$
|f|_{W^{s,\tau}(\Omega)^q}:=\sum_{|\alpha|=n} \int_{\Omega\times \Omega}
\frac {|\partial^\alpha f(x)-\partial^\alpha f(y)|^\tau}
{|x-y|^{(s-n)\tau+d}} dxdy,
$$
where $n$ is the largest integer below $s$. In the more
general case where $\frac 1 \tau\leq \frac 1 p + \frac s d$, 
we obtain an estimate that depends on the diameter of $T$:
\be
e_{m,T}(f)_p\leq Ch_T^{r} |f|_{W^{s,\tau}(T)},\;\; r:=\frac d p-\frac d \tau + s\geq 0.
\label{CMlocalinvsr}
\ee
Finally, remark that for a fixed $p\geq 1$ and $s$, the index $\tau$ defined by 
$\frac 1 \tau=\frac 1 p+\frac s d$ may be smaller than $1$, in which case
the Sobolev space $W^{s,\tau}(\Omega)$ is not well defined. The local
estimate remain valid if $W^{s,\tau}(\Omega)$ is replaced by the {\it Besov space} $B^s_{\tau,\tau}(\Omega)$.
This space consists of all $f\in L^\tau(\Omega)$ functions such that
$$
|f|_{B^s_{\tau,\tau}}:=\|\omega_{k}(f,\cdot)_\tau\|_{L^\tau([0,\infty[,\frac {dt} t)},
$$
is finite. Here $k$ is the smallest integer above $s$ and
$\omega_k(f,t)_\tau$ denotes the $L^\tau$-modulus of smoothness of order $k$
defined by
$$
\omega_k(f,t)_\tau:=\sup_{|h|\leq t}\|\Delta_h^k f\|_{L^\tau},
$$
where $\Delta_h f:=f(\cdot+h)-f(\cdot)$ is the usual difference operator.
The space $B^s_{\tau,\tau}$ describes functions which have 
``$s$ derivatives in $L^\tau$'' in a very similar way as $W^{s,\tau}$.
In particular it is known that these two spaces coincide when
$\tau\geq 1$ and $s$ is not an integer.
We refer to \cite{CMdel} and \cite{CMco} for more details on Besov spaces
and their characterization by approximation procedures. For 
all $p,\tau>0$ and $0\leq s\leq m$ such that $\frac 1 \tau\leq \frac 1 p + \frac s d$, 
a local estimate generalizing \iref{CMlocalinvsr} has the form
\be
e_{m,T}(f)_p\leq Ch_T^{r} |f|_{B^{s}_{\tau,\tau}(T)},\;\; r:=\frac d p-\frac d \tau + s\geq 0.
\ee

\subsection{Global estimates}

We now turn our local estimates into global
estimates, recalling that 
$$
e_{m,\cT}(f)_p:=\min_{g\in V_{\cT}}\|f-g\|_{L^p} =\(\sum_{T\in\cT}e_{m,T}(f)_p^p\)^{1/p};
$$
with the usual modification when $p=\infty$.
We apply the principle of error equidistribution
assuming that the partition
$\cT_N$ is built in such way
that 
\be
e_{m,T}(f)_p=\eta,
\label{strictequid}
\ee
for all $T\in\cT_N$ where $N=N(\eta)$. A first immediate estimate for the
global error is therefore
\be
e_{m,\cT_N}(f)_p\leq N^{1/p}\eta.
\label{CMtrivialglobal}
\ee
Assume now that $f\in W^{m,\tau}(\Omega)$ with $\tau\geq 1$ such
that $\frac 1 \tau=\frac 1 p +\frac m d$. It then follows from
Theorem \ref{CMlocalestheo} that
$$
N\eta^\tau\leq \sum_{T\in\cT_N}e_{m,T}(f)_p^\tau \leq C \sum_{T\in\cT_N} |f|_{W^{m,\tau}(T)}^\tau
=C|f|_{W^{m,\tau}}^\tau,
$$
Combining with \iref{CMtrivialglobal} 
and using the relation $\frac 1 \tau=\frac 1 p+Ê\frac m d$, we have thus obtained
that for adaptive partitions $\cT_N$ built according to the error 
equidistribution, we have
\be
e_{m,\cT_N}(f)_p\leq CN^{-m/d}|f|_{W^{m,\tau}}.
\label{CMadaptisoequid}
\ee
By using \iref{CMlocalinvs}, we obtain in a similar manner
that if $0\leq s\leq m$ and $\tau\geq 1$ 
are such that $\frac 1 \tau=\frac 1 p+Ê\frac s d$, then
\be
e_{m,\cT_N}(f)_p\leq CN^{-s/d}|f|_{W^{s,\tau}}.
\label{CMerrorisoequid}
\ee
Similar results hold when $\tau<1$ with $W^{s,\tau}$ replaced by $B^s_{\tau,\tau}$
but their proof requires a bit more work due to the fact that $|f|_{B^s_{\tau,\tau}}^\tau$
is not sub-additive with respect to the union of sets. We also reach similar
estimate in the case $p=\infty$ by a standard modification of the argument.

The estimate \iref{CMadaptisoequid} suggests that
for piecewise polynomial approximation on adaptive and isotropic
partitions, we have
\be
\sigma_N(f)_p\leq CN^{-m/d}|f|_{W^{m,\tau}}, \;\; \frac 1 \tau=\frac 1 p+\frac m d.
\label{CMsigmaNtisoequid} 
\ee
Such an estimate should be compared to \iref{CMsigmaNunif},
in a similar way as we compared \iref{CMconstadapteq}
with \iref{CMlipunif} in the one dimensional case: the same
same rate $N^{-m/d}$ is governed by a weaker smoothness condition.

In contrast to the one dimensional case, however, we cannot
easily prove the validity of \iref{CMsigmaNtisoequid} since it
is not obvious that there exists a partition $\cT_N\in \cA_N$
which equidistributes the error in the sense of \iref{strictequid}.
It should be remarked that the derivation of estimates such as 
\iref{CMadaptisoequid} does not require a strict equidistribution
of the error. It is
for instance sufficient to assume
that $e_{m,T}(f)_p\leq \eta$
for all $T\in \cT_N$, and that
$$
c_1\eta\leq e_{m,T}(f)_p,
$$
for at least $c_2N$ elements of $\cT_N$, where $c_1$ and $c_2$
are fixed constants. Nevertheless, the construction of 
a partition $\cT_N$ satisfying such prescriptions still appears
as a difficult task both from a theoretical and algorithmical point of view.

\subsection{An isotropic greedy refinement algorithm}
\label{CMsecgreedyiso}

We now discuss a simple adaptive refinement algorithm
which emulates error equidistribution, similar to the
algorithm which was discussed in the one dimensional case.
For this purpose, we first build a hierarchy of nested 
quasi-uniform partitions
$(\cD_j)_{j\geq 0}$, where $\cD_0$ is a
coarse triangulation and where $\cD_{j+1}$ is obtained
from $\cD_j$ by splitting each of its elements into
a fixed number $K$ of children. We therefore have
$$
\#(\cD_j)=K^j \#(\cD_0),
$$
and since the partitions $\cD_j$ are assumed to be
quasi-uniform, there exists two constants $0<c_1\leq c_2$ such that
\be
c_1 K^{-j/d}\leq h_T \leq c_2 K^{-j/d},
\label{quasiunifhier}
\ee
for all $T\in\cD_j$ and $j\geq 0$. For example, in the case of
two dimensional triangulations, we may choose
$K=4$ by splitting each triangle into $4$ similar triangles
by the midpoint rule, or $K=2$ by bisecting each triangle
from one vertex to the midpoint of the opposite edge
according to a prescribed rule in order to preserve isotropy.
Specific rules which have been extensively studied
are bisection from the most recently 
generated vertex \cite{CMbdd} or towards the longest edge \cite{CMri}.
In the case of partitions by rectangles, we may preserve isotropy by 
splitting each rectangle into $4$ similar rectangles by the midpoint rule.

The refinement algorithm reads as follows:
\begin{enumerate}
\item
Initialization: $\cT_{N_0}=\cD_0$ with $N_0:=\#(\cD_0)$.\vspace{0.2cm}
\item
Given $\cT_N$ select $T\in\cT_N$ that maximizes $e_{m,T}(f)_T$.\vspace{0.2cm}
\item
Split $T$ into its $K$ childrens to obtain $\cT_{N+K-1}$ and return to step 2.
\end{enumerate}
Similar to the one dimensional case, the
adaptive partitions that are generated by this algorithm
are restricted to a particular family where 
each element $T$ is picked within an infinite
dyadic {\it master tree} $\cM=\cup_{j\geq 0} \cD_j$ which 
roots are given by the elements $\cD_0$.
The partition $\cT_N$ may be identified
to a finite subtree of $\cM$ with $N$ leaves.
Figure \ref{CMisopart}
displays an example of adaptively refined partitions
either based on longest edge bisection for triangles,
or by quad-split for squares.

\begin{figure}
\centerline{
\includegraphics[width=4cm,height=4cm]{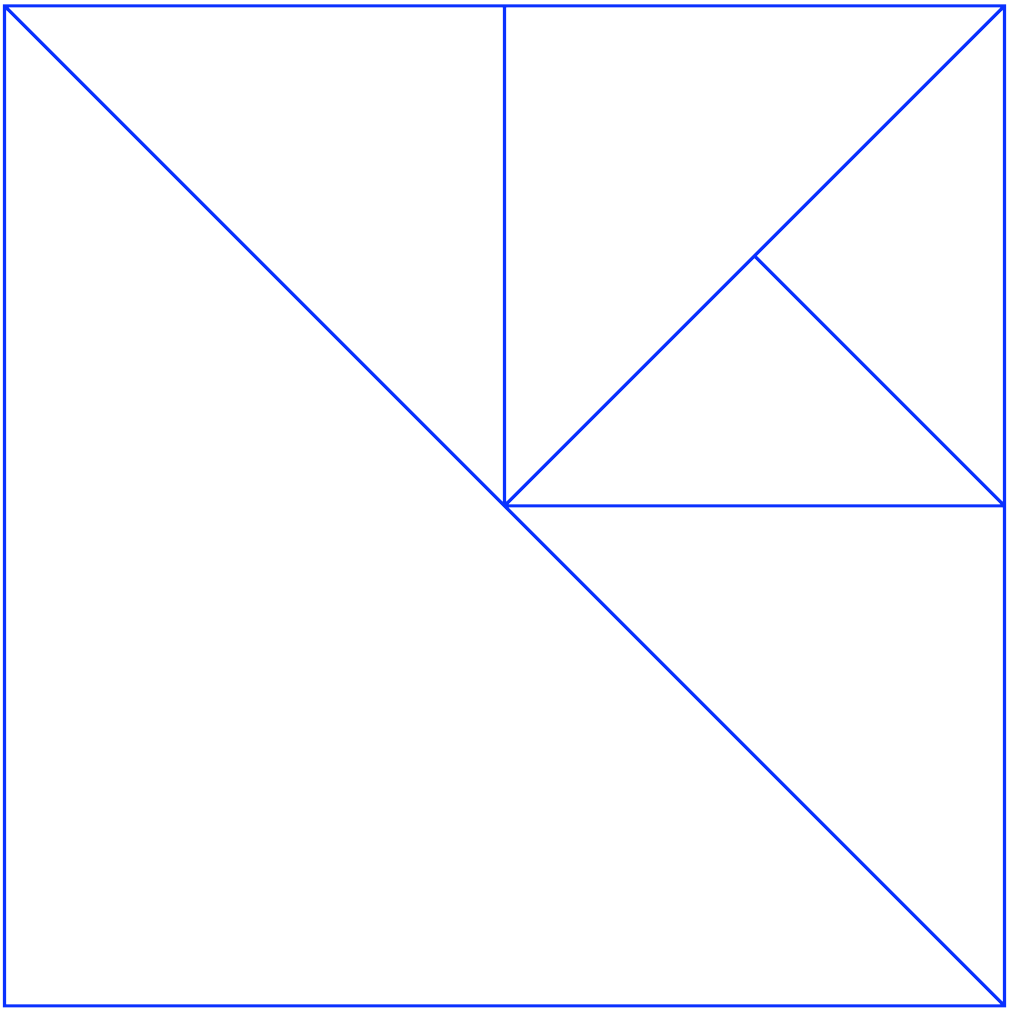}
\hspace{1cm}
\includegraphics[width=4cm,height=4cm]{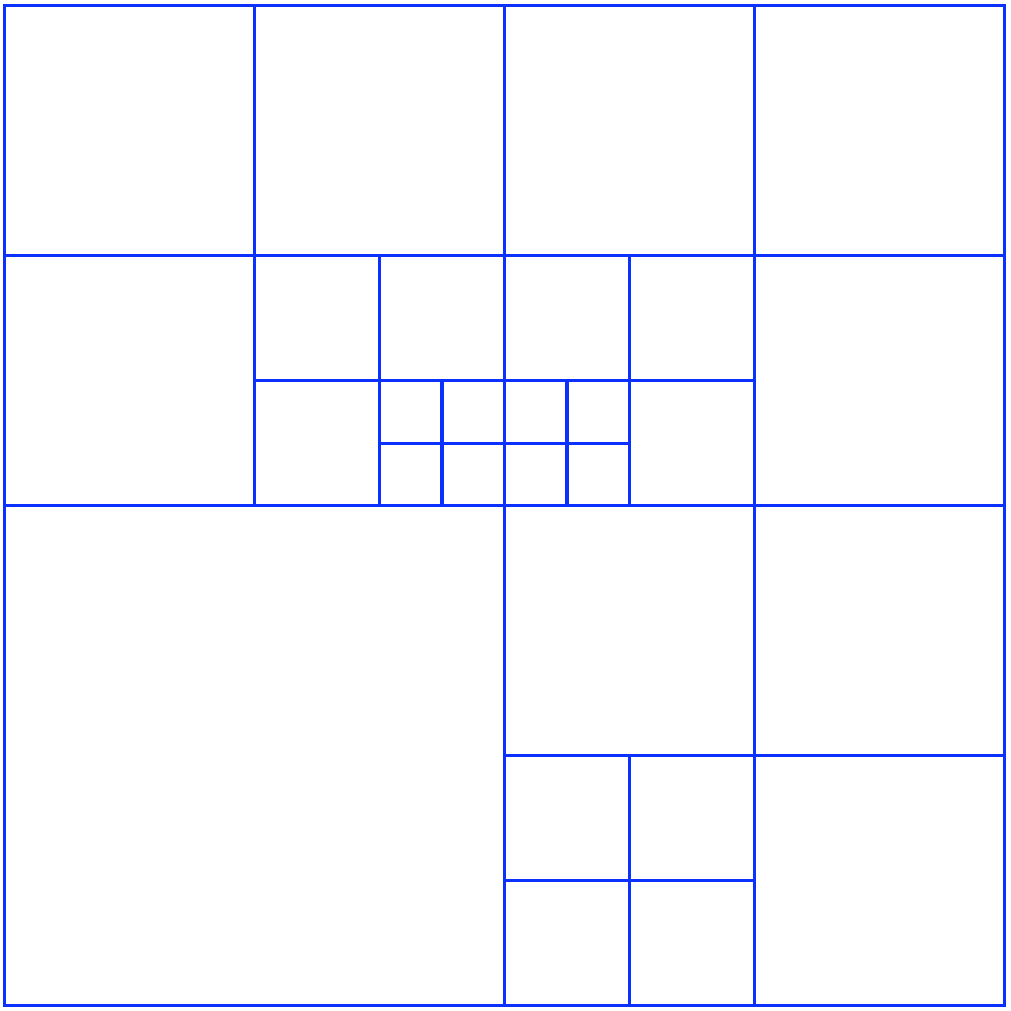}
}
\caption{Adaptively refined partitions based on longest edge bisection (left)
or quad-split (right)}
\label{CMisopart}
\end{figure}

This algorithm cannot exactly achieve error equidistribution,
but our next result reveals that it generates partitions
that yield error estimates almost similar to \iref{CMadaptisoequid}.

\begin{theorem}
\label{CMtheoisogreedy}
If $f\in W^{m,\tau}(\Omega)$ for some $\tau\geq 1$ such that
$\frac 1 \tau < \frac 1 p+\frac m d$, we then have for all $N\geq 2 N_0=2\#(\cD_0)$,
\be
e_{m,\cT_N}(f)_p\leq CN^{-m/d}|f|_{W^{m,\tau}},
\label{CMerrorisogreedy}
\ee
where $C$ depends on $\tau$, $m$, $K$, $R$ and the choice of $\cD_0$.
We therefore have for piecewise polynomial
approximation on adaptively refined partitions
\be
\sigma_N(f)_p\leq CN^{-m/d}|f|_{W^{m,\tau}}, \;\; \frac 1 \tau>\frac 1 p+\frac m d.
\label{CMsigmaNtisogreedy} 
\ee
\end{theorem}
\begin{proof}
The technique used for proving this result is adapted
from the proof of a similar result for tree-structured 
wavelet approximation in \cite{CMcddd}.
We define
\be
\eta:=\max_{T\in\cT_N}e_{m,T}(f)_p, 
\label{CMmaxlocalerr}
\ee
so that we obviously have when $p<\infty$,
\be
e_{m,\cT_N}(f)_p\leq N^{1/p}\eta.
\label{CMtrivialglobalgreedy}
\ee
For $T\in \cT_N\sm \cD_0$, we denote by $P(T)$ its parent
in the refinement process. From the definition of the algorithm, we necessarily have
$$
\eta \leq e_{m,P(T)}(f)_p,
$$
and therefore, using \iref{CMlocalinvsr} with $s=m$, we obtain 
\be
\eta \leq  Ch_{P(T)}^{r} |f|_{W^{s,\tau}(P(T))},
\label{estimbeloweta}
\ee
with $r:=\frac d p-\frac d \tau + m> 0$. We next denote by $\cT_{N,j}:=\cT_N\cap \cD_j$ the
elements of generation $j$ in $\cT_N$ and define $N_j:=\#(\cT_{N,j})$.
We estimate $N_j$ by taking the $\tau$ power of \iref{estimbeloweta}
and summing over $\cT_{N,j}$ which gives
$$
\begin{array}{ll}
N_j\eta^\tau & \leq C^\tau
\sum_{T\in \cT_{N,j}} h_{P(T)}^{r\tau} |f|_{W^{s,\tau}(P(T))}^\tau\\
& \leq C^\tau (\sup_{T\in \cT_{N,j}} h_{P(T)}^{r\tau}) \sum_{T\in \cT_{N,j}}|f|_{W^{s,\tau}(P(T))}^\tau\\
& \leq KC^\tau (\sup_{T\in \cD_{j-1}} h_T^{r\tau}) |f|_{W^{s,\tau}}^\tau.
\end{array}
$$
Using \iref{quasiunifhier} and the fact that $\#(\cD_j)=N_0K^j$, we thus obtain 
$$
N_j\leq \min\{C\eta^{-\tau}K^{-jr\tau/d} |f|_{W^{s,\tau}}^\tau\, ,\, N_0 K^j\}.
$$
We now evaluate
$$
N-N_0=\sum_{j\geq 1}N_j \leq \sum_{j\geq 1}\min\{C\eta^{-\tau} K^{-jr\tau/d} |f|_{W^{s,\tau}}^\tau\, ,\, N_0 K^j\}.
$$
By introducing $j_0$ the smallest integer such that $C\eta^{-\tau} K^{-jr\tau/d} |f|_{W^{s,\tau}}^\tau\leq N_0K^j$, 
we find that
$$
N-N_0 \leq N_0\sum_{j\leq j_0} K^{j}+C\eta^{-\tau} |f|_{W^{s,\tau}}^\tau\sum_{j> j_0}K^{-jr\tau/d},
$$
which after evaluation of $j_0$ yields
$$
N-N_0 \leq C\eta^{-\frac{d\tau}{d+r\tau}} |f|_{W^{s,\tau}}^{\frac{d\tau}{d+r\tau}}=
C\eta^{-\frac {dp}{d+mp}}|f|_{W^{s,\tau}}^{\frac {dp}{d+mp}},
$$
and therefore, assuming that $N\geq 2 N_0$,
$$
\eta\leq CN^{-1/p-m/d}|f|_{W^{s,\tau}}.
$$
Combining this estimate with \iref{CMtrivialglobalgreedy} gives the announced result.
In the case $p=\infty$, a standard modification of the argument leads to a similar conclusion.  
\hfill $\Box$
\end{proof}

\begin{remark}
By similar arguments, we obtain that if $f\in W^{s,\tau}(\Omega)$ for some $\tau\geq 1$ 
and $0\leq s \leq m$ such that $\frac 1 \tau < \frac 1 p+\frac s d$, we have
$$
e_{m,\cT_N}(f)_p\leq CN^{-s/d}|f|_{W^{s,\tau}}.
$$
The restriction $\tau \geq 1$ may be dropped if we replace $W^{s,\tau}$
by the Besov space $B^s_{\tau,\tau}$, at the price of a more technical proof.
\end{remark}

\begin{remark}
The same approximation results can be obtained if we 
replace $e_{m,T}(f)_p$ in the refinement algorithm by the more computable
quantity $\|f-P_{m,T}f\|_{L^p(T)}$, due to the equivalence \iref{CMequivlocal}.
\end{remark}

\begin{remark}
The greedy refinement algorithm defines a particular sequence of subtrees $\cT_N$
of the master tree $\cM$, but $\cT_N$ is not ensured to be the best
choice in the sense of minimizing the approximation error among all 
subtrees of cardinality at most $N$. The selection of an optimal tree
can be performed by an additional pruning strategy after enough refinement
has been performed. This approach
was developped in the context of statistical estimation under
the acronyme CART (classification and regression tree), see \cite{CMbfos,CMdo}.
Another approach that builds a near optimal subtree only based
on refinement was proposed in \cite{CMbd}.
\end{remark}

\begin{remark}
The partitions which are built by the greedy refinement
algorithm are non-conforming. Additional refinement
steps are needed when the users insists on conformity,
for instance when solving PDE's. For specific refinement
procedures, it is possible to bound the total number
of elements that are due to additional
conforming refinement by the total number of triangles $T$
which have been refined due to the fact that 
$e_{m,T}(f)_T$ was the largest at some stage of the algorithm,
up to a fixed multiplicative constant. In turn, the convergence
rate is left unchanged compared to the original non-conforming
algorithm. This fact was proved in \cite{CMbdd} for adaptive
triangulations built by the rule of newest vertex bisection. 
A closely related concept is the amount of
additional elements which are needed in order to impose that
the partition satisfies a {\it grading property}, in the sense
that two adjacent elements may only differ by one 
refinement level. For specific partitions, it was proved
in \cite{CMdahpart} that this amount is bounded up to
a fixed multiplicative constant the number of elements contained in the 
non-graded partitions.
Figure \ref{CMisopartref}
displays the conforming and graded partitions obtained by the minimal amount of 
additional refinement from the partitions of Figure \ref{CMisopart}.
\end{remark}

\begin{figure}
\centerline{
\includegraphics[width=4cm,height=4cm]{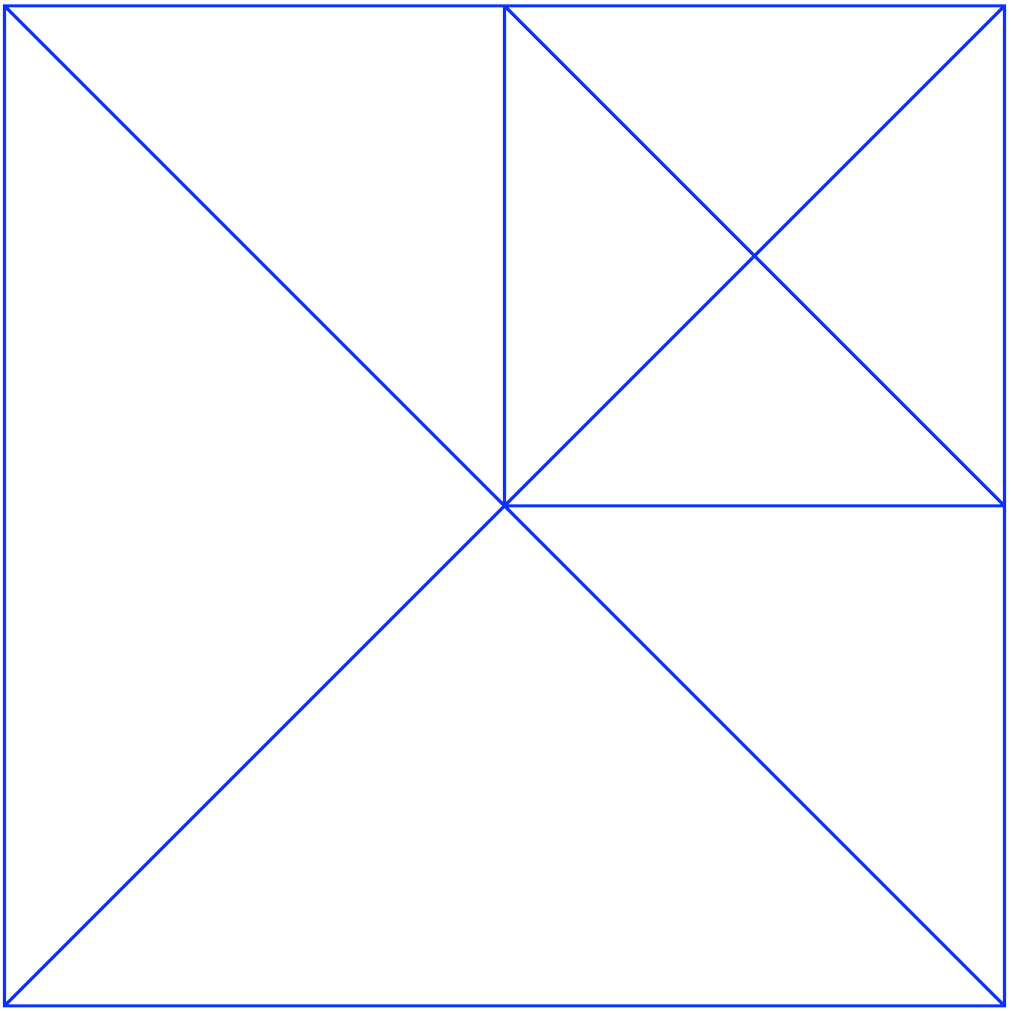}
\hspace{1cm}
\includegraphics[width=4cm,height=4cm]{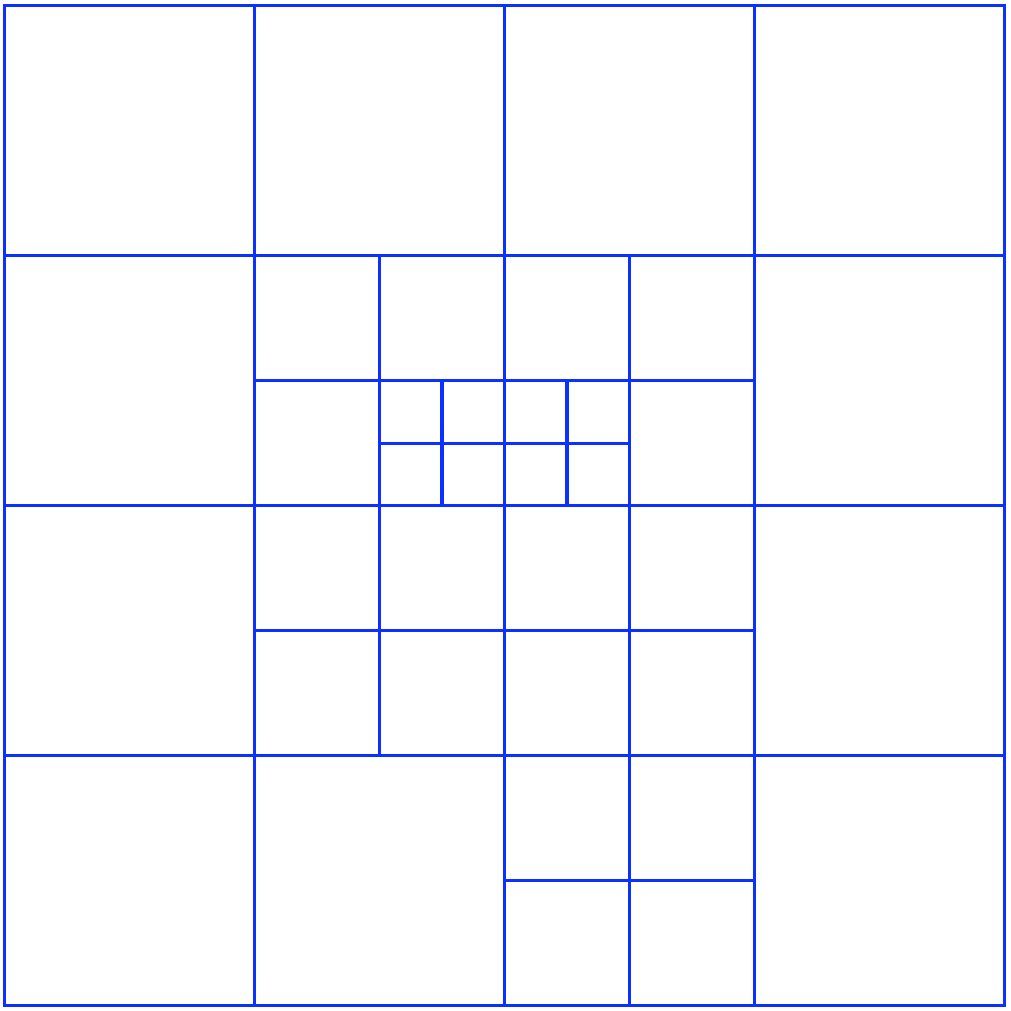}
}
\caption{Conforming refinement (left)
and graded refinement (right)}
\label{CMisopartref}
\end{figure}

The refinement algorithm may also be applied to
discretized data, such as numerical images.
The approximated $512\times 512$ image is displayed
on Figure \ref{CMfigimage} together with its approximation
obtained by the refinement algorithm
based on newest vertex bisection and the
error measured in $L^2$, using $N=2000$ triangles.
In this case, $f$ has the form of a discrete array of pixels,
and the $L^2(T)$-orthogonal projection is replaced
by the $\ell^2(S_T)$-orthogonal projection,
where $S_T$ is the set of pixels with centers contained in $T$.
The use of adaptive isotropic partitions has strong
similarity with wavelet thresholding \cite{CMde,CMco}.
In particular, it results in ringing artifacts near the edges.

\begin{figure}[htbp]
\centerline{
\includegraphics[width=5cm,height=5cm]{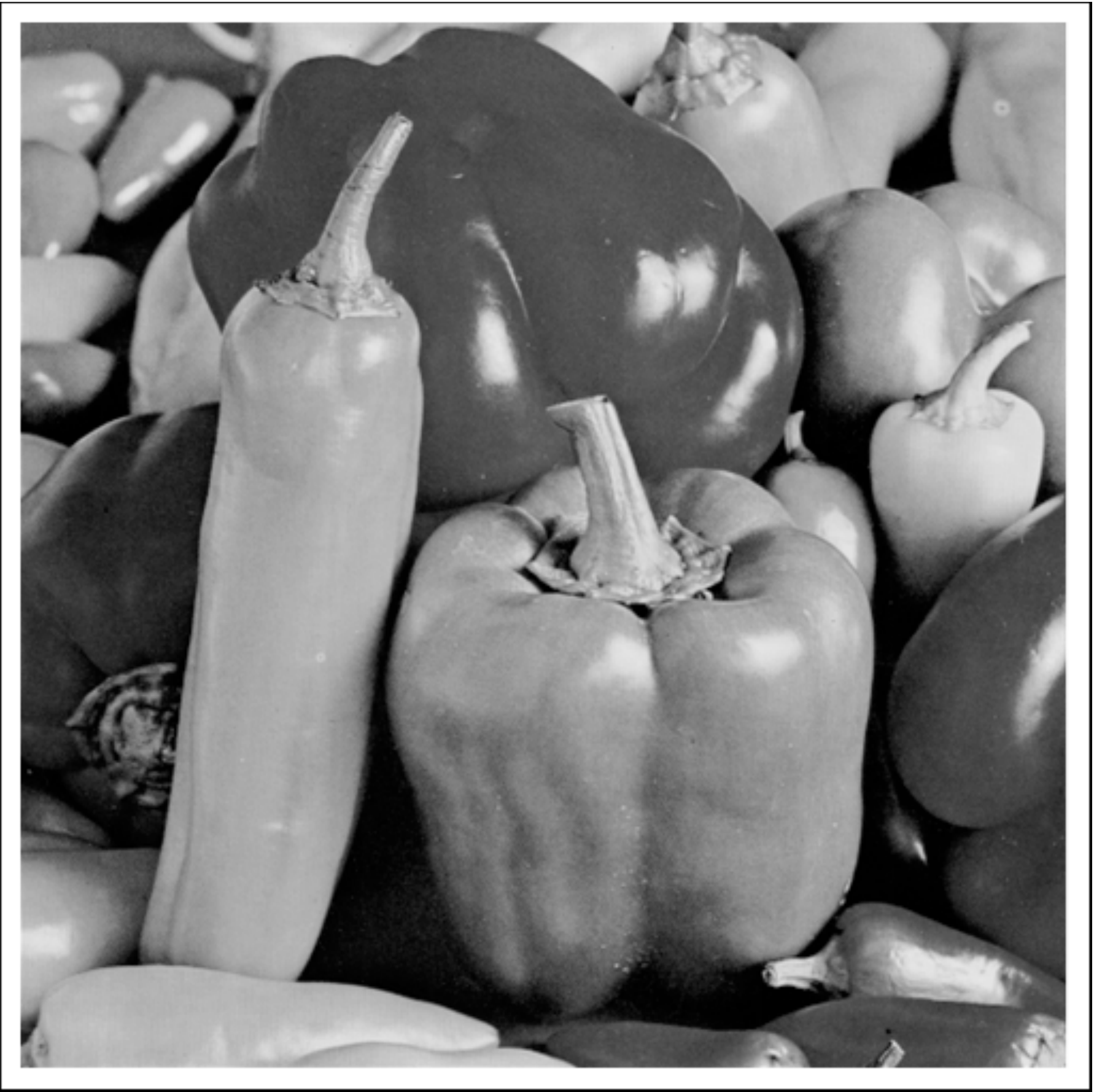}
\hspace{1cm}
\includegraphics[width=5cm,height=5cm]{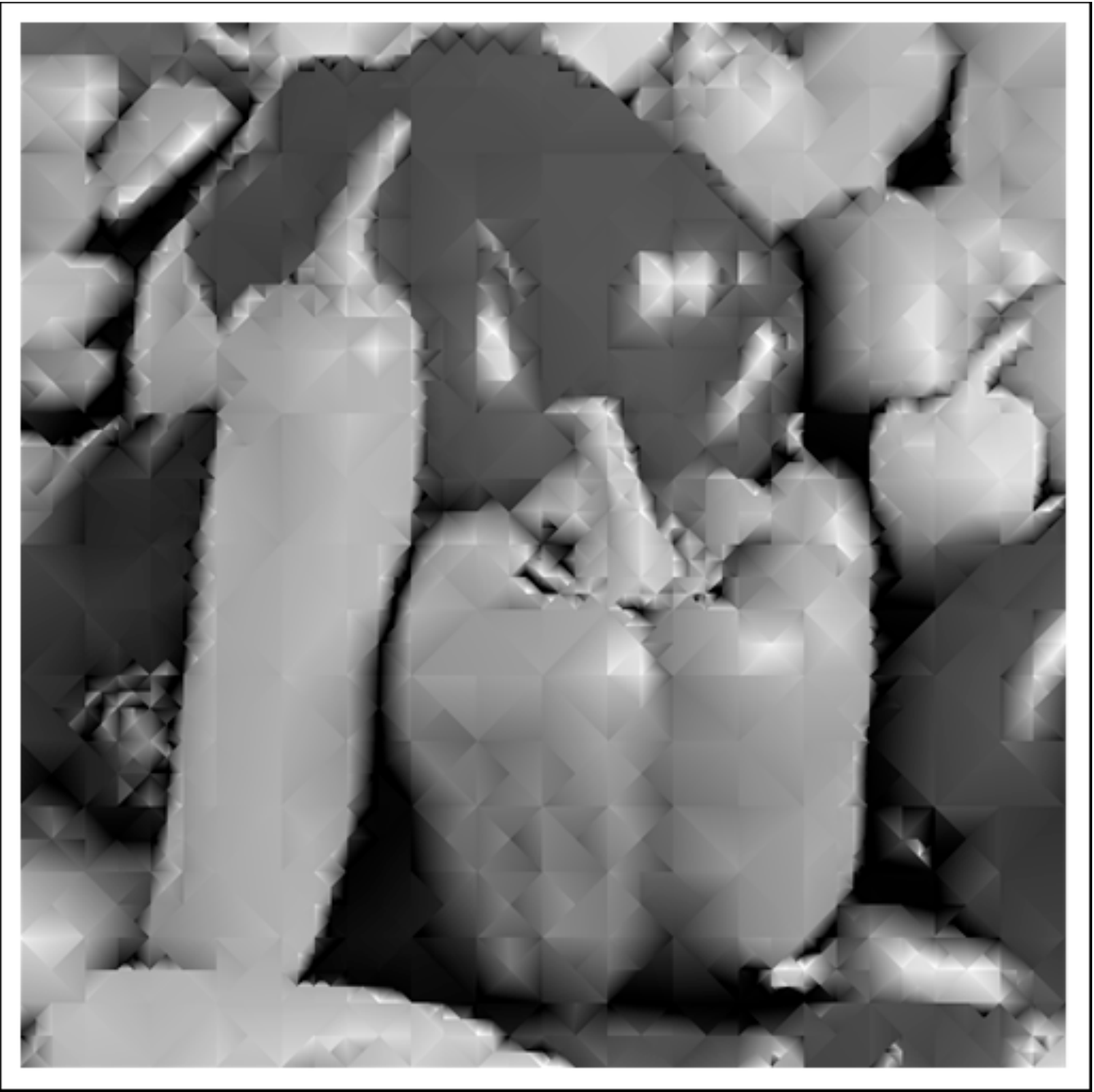}
}
\caption{The image `peppers'' (left) and its approximation by $2000$ isotropic triangles 
obtained by the greedy algorithm (right).}
\label{CMfigimage}
\end{figure}

\subsection{The case of smooth functions.}

Although the estimate \iref{CMsigmaNtisoequid} 
might not be achievable for a general $f\in W^{m,\tau}(\Omega)$, 
we can show that for smooth enough $f$, the numerical
quantity that governs the rate of convergence $N^{-\frac n d}$
is exactly $|f|_{W^{m,\tau}}:=\(\sum_{|\alpha|=m} \|\partial^\alpha f\|_{L^\tau}^\tau\)^{1/\tau}$
that we may define as so even for $\tau<1$.
For this purpose, we assume that $f\in C^m(\Omega)$.
Our analysis is based on the fact that such
a function can be locally approximated by a polynomial
of degree $m$. 

We first study in more detail the
approximation error on a function $q\in \PP_m$.
We denote by $\H_m$ the space of {\it homogeneous
polynomials} of degree $m$. To $q\in \PP_m$, we
associate its homogeneous part $\bq\in \H_m$,
which is such that
$$
q-\bq\in \PP_{m-1}.
$$
We denote by $\bq_\alpha$ the coefficient of $\bq$ associated
to the multi-index $\alpha=(\alpha_1,\cdots,\alpha_d)$ with $|\alpha|=m$.
We thus have
$$
e_{m,T}(q)_p=e_{m,T}(\bq)_p.
$$
Using the affine transformation $A_T$ which maps the reference element
$R$ onto $T$, and denoting by $B_T$ its linear part, we can write 
$$
e_{m,T}(\bq)_p=\(\frac {|T|}{|R|}\)^{1/p}e_{R,m}(\bq\circ A_T)_p=\(\frac {|T|}{|R|}\)^{1/p}
e_{m,R}(\t \bq)_p, \;\; \t\bq:=\bq\circ B_T\in\H_m
$$
where we have used the fact that $\t\bq-\bq\circ A_TÊ\in \PP_{m-1}$.
Introducing for any $r>0$ the quasi-norm on $\H_m$
$$
|\bq|_r:=\(\sum_{|\alpha|=m} |\bq_\alpha|^r\)^{1/r},
$$
one easily checks that
$$
C^{-1}\|B_T^{-1}\|^{-m} |\bq|_r\leq |\t \bq|_r\leq C\|B_T\|^m |\bq|_r,
$$
for some constant $C>0$ that only depends on $m$, $r$ and $R$.
We then remark that $e_{R,m}(\bq)_p$ is a norm on $\H_m$,
which is equivalent to $|\bq|_r$ since $\H_m$ is finite dimensional.
It follows that there exists constants $0<C_1\leq C_2$ such that
for all $q$ and $T$
$$
C_1 |T|^{1/p}\|B_T^{-1}\|^{-m} |\bq|_r   \leq e_{m,T}(q)_p\leq C_2|T|^{1/p}\|B_T\|^m |\bq|_r.
$$
Finally, using the bound \iref{CMboundbt} for $\|B_T\|$ 
and its symmetrical counterpart 
$$
\|B_T^{-1}\| \leq \frac {h_R}{\rho_T},
$$
together with the isotropy restriction \iref{isotropy}, we obtain with $\frac 1 \tau:=\frac 1 p+\frac m d$ 
the equivalence
$$
C_1 |T|^{\tau} |\bq|_r   \leq e_{m,T}(q)_p\leq C_2|T|^{\tau} |\bq|_r, 
$$
where $C_1$ and $C_2$ only depend on $m$, $R$ and the
constant $C$ in \iref{isotropy}. Choosing $r=\tau$ this equivalence 
can be rewritten as
$$
C_1 \(\sum_{|\alpha|=m} \|\bq_\alpha\|_{L^\tau(T)}^\tau \)^{1/\tau }   \leq e_{m,T}(q)_p\leq C_2
\(\sum_{|\alpha|=m} \|\bq_\alpha\|_{L^\tau(T)}^\tau \)^{1/\tau }.
$$
Using shorter notations, this is summarized by the following result.
\begin{lemma}
Let $p\geq 1$ and $\frac 1 \tau:=\frac 1 p+\frac m d$.
There exists constant $C_1$ and $C_2$ that only depends
on $m$, $R$ and the constant $C$ in \iref{isotropy} such that
\be
C_1 |q|_{W^{m,\tau}(T)}   \leq e_{m,T}(q)_p\leq C_2
|q|_{W^{m,\tau}(T)},
\label{CMequierrorpoly}
\ee
for all $q\in\PP_m$.
\end{lemma}
In what follows, we shall frequently identify the $m$-th order derivatives
of a function $f$ at some point $x$ with an homogeneous
polynomial of degree $m$. In particular we write
$$
|d^mf(x)|_r:=\(\sum_{|\alpha|=m}|\partial^\alpha f(x)|^r\)^{1/r}.
$$
We first establish a lower estimate on $\sigma_N(f)$,
which reflects the saturation rate $N^{-m/d}$ of the method,
under a slight restriction on the set
$\cA_N$ of admissible partitions, assuming that
the diameter of all elements decreases 
as $N\to +\infty$, according to 
\be
\max_{T\in \cT_N} h_T \leq AN^{-1/d},
\label{CMunidecaydiam}
\ee
for some $A>0$ which may be arbitrarily large.
\begin{theorem}
\label{CMlowerisotheo}
Under the restriction {\rm \iref{CMunidecaydiam}},
there exists a constant $c>0$ that only depends on $m$, $R$ and the
constant $C$ in \iref{isotropy} such that 
\be
\liminf_{N\to +\infty} \; N^{m/d}\sigma_N(f)_p \geq  c|f|_{W^{m,\tau}}
\label{CMliminfiso}
\ee
for all $f\inÊC^m(\Omega)$, where  $\frac 1 \tau:=\frac 1 p+\frac m d$.
\end{theorem}

\begin{proof}
If $f\in C^m(\Omega)$ and $x\in \Omega$, we denote by $q_x$ the 
Taylor polynomial of order $m$ at the point $x=(x_1,\cdots,x_d)$:
\be
q_x(y)=q_x(y_1,\cdots,y_d):=\sum_{|\alpha|\leq m} \frac 1 {|\alpha| !} \partial^\alpha f(x)(y_1-x_1)^{\alpha_1}\cdots (y_d-x_d)^{\alpha_d}.
\label{CMtaylorx}
\ee
If $\cT_N$ is a partition in $\cA_N$, we may write for each 
element $T\in \cT_N$ and $x\in T$
$$
\begin{array}{ll}
e_{m,T}(f)_p& \geq e_{m,T}(q_{x}) _p-\|f-q_{x}\|_{L^p(T)} \\
& \geq C_1 |q_{x}|_{W^{m,\tau}(T)}-\|f-q_{x}\|_{L^p(T)}\\
& \geq c |f|_{W^{m,\tau}(T)}-C_1|f-q_{x}|_{W^{m,\tau}(T)}-\|f-q_{x}\|_{L^p(T)},
\end{array}
$$
with $c:= C_1\min\{1,\tau\}$,
where we have used the lower bound in \iref{CMequierrorpoly}
and the quasi-triangle inequality 
$$
\|u+v\|_{L^\tau}\leq \max\{1,\tau^{-1}\}(\|u\|_{L^\tau}+\|v\|_{L^\tau}).
$$
By the continuity of the $m$-th order derivative of $f$,
we are ensured that for all $\e>0$ there exists $\delta>0$ such that 
\be
|x-y|\leq \delta \Rightarrow |f(y)-q_x(y)|\leq \e |x-y|^m \;\;{\rm and}\;\;
|d^mf(y)-d^mq_x|_\tau \leq \e.
\label{CMcontmder}
\ee
Therefore if $N\geq N_0$ such that $AN_0^{-1/d}\leq \delta$, we 
have
$$
\begin{array}{ll}
e_{m,T}(f)_p & \geq c |f|_{W^{m,\tau}(T)}- (C_1\e |T|^{1/\tau}+\e h_T^m |T|^{1/p})\\
& \geq c |f|_{W^{m,\tau}(T)}-(1+C_1)\e h_T^{m+d/p}\\
& \geq c |f|_{W^{m,\tau}(T)}- C\e N^{-1/\tau},
\end{array}
$$
where the constant $C$ depends on $C_1$ in \iref{CMequierrorpoly} and $A$
in \iref{CMunidecaydiam}. Using triangle inequality, 
it follows that
$$
e_{m,\cT_N}(f)_p =\(\sum_{T\in\cT}e_{m,T}(f)_p^p\)^{1/p}Ê
\geq c\(\sum_{T\in\cT}|f|_{W^{m,\tau}(T)}^p\)^{1/p}-C\e N^{-m/d}.
$$
Using H\"older's inequality, we find that
\be
|f|_{W^{m,\tau}}=\(\sum_{T\in\cT}|f|_{W^{m,\tau}(T)}^\tau \)^{1/\tau}\leq N^{m/d}
\(\sum_{T\in\cT}|f|_{W^{m,\tau}(T)}^p\)^{1/p},
\label{CMholder}
\ee
which combined with the previous estimates shows that
$$
N^{m/d} e_{m,\cT_N}(f)_p \geq c|f|_{W^{m,\tau}}-C\e.
$$
Since $\e>0$ is arbitrary this concludes the proof. \hfill $\Box$
\end{proof}

\begin{remark}
The H\"older's inequality \iref{CMholder} becomes an equality if and only if all quantities
in the sum are equal,  which justifies the error equidistribution principle
since these quantities are approximations of 
$e_{m,T}(f)_p$.
\end{remark}

We next show that if $f\in C^m(\Omega)$, the
adaptive approximations obtained by
the greedy refinement algorithm introduced in \S\ref{CMsecgreedyiso}
satisfy an upper estimate which closely 
matches the lower estimate \iref{CMliminfiso}. 

\begin{theorem}
\label{CMupperisotheo}
There exists a constant $C$ that only depends on 
$m$, $R$ and on the choice of the
hierarchy $(\cD_j)_{j\geq 0}$ such that
for all $f\in C^m(\Omega)$, the partitions $\cT_N$ obtained
by the greedy algorithm satisfy.
\be
\limsup_{N\to +\infty} \; N^{m/d}e_{m,\cT_N}(f)_p \leq  C|f|_{W^{m,\tau}},
\label{CMlimsupisogreedy}
\ee
where  $\frac 1 \tau:=\frac 1 p+\frac m d$. In turn, for
adaptively refined partitions, we have
\be
\limsup_{N\to +\infty} \; N^{m/d}
\sigma_N(f)_p \leq  C|f|_{W^{m,\tau}},
\label{CMlimsupsigmaNisogreedy}
\ee
for all $f\in C^m(\Omega)$.
\end{theorem}

\begin{proof}
For any $\e>0$, we choose $\delta>0$ such that 
\iref{CMcontmder} holds.
We first remark that there exists $N(\delta)$ sufficiently large
such that for any $N\geq N(\delta)$ at least $N/2$ elements $T\in\cT_N$
have parents with diameter $h_{P(T)} \leq \delta$. Indeed, the 
uniform isotropy of the elements
ensures that
$$
|T|\geq ch_{P(T)}^d,
$$
for some fixed constant $c>0$. We thus have
$$
\#\{T\in\cT_N\; ; \; h_{P(T)} \geq \delta\}  \leqÊ\frac {|\Omega|}{c \delta^d},
$$
and the right-hand side is less than $N/2$ for large enough $N$.
We denote by $\t \cT _N$ the subset of $T\in \cT_N$
such that $h_{P(T)} \leq \delta$. Defining $\eta$
as previously by \iref{CMmaxlocalerr},
we observe that for all $T\in \t \cT_N \sm \cD_0$, we have
\be
\eta \leq e_{m,P(T)}(f)_p.
\label{CMlocalerroreta}
\ee
If $x$ is any point contained in $T$
and $q_{x}$ the Taylor polynomial of $f$
at this point defined by \iref{CMtaylorx}, we have
$$
\begin{array}{ll}
e_{m,P(T)}(f)_p & \leq e_{m,P(T)}(q_{x})_p+\|f-q_{x}\|_{L^p(P(T))} \\
& \leq C_2|q_{x}|_{W^{m,\tau}(P(T))} +\e h_{P(T)}^m|P(T)|^{1/p}\\
& \leq C_2\(\frac {|P(T)|}{|T|}\)^{1/\tau}|q_{x}|_{W^{m,\tau}(T)}+\e h_{P(T)}^m|P(T)|^{1/p}\\
& \leq C_2\(\frac {|P(T)|}{|T|}\)^{1/\tau}|f|_{W^{m,\tau}(T)}+
\e D_2\(\frac {|P(T)|}{|T|}\)^{1/\tau}|T|^{1/\tau}
+\e h_{P(T)}^m|P(T)|^{1/p},\\
\end{array}
$$
where $C_2$ is the constant appearing in \iref{CMequierrorpoly}
and $D_2:=C_2\max\{1,1/\tau\}$.
Combining this with \iref{CMlocalerroreta}, we obtain 
that for all $T\inÊ\t\cT_N$,
$$
\etaÊ\leq D (|f|_{W^{m,\tau}(T)}+\e |T|^{1/\tau})
$$
where the constant $D$ depends on $C_2$, $m$
and on the refinement rule defining the hierarchy $(\cD_j)_{j\geq 0}$.
Elevating to the power $\tau$ and summing on all $T\in\t \cT_N$, we
thus obtain 
$$
(N/2-N_0)\eta^\tau \leq \max\{1,\tau\} D^\tau( |f|_{W^{m,\tau}}^\tau +\e^\tau |\Omega|),
$$
where $N_0:=\#(\cD_0)$.
Combining with \iref{CMtrivialglobalgreedy}, we therefore
obtain
$$
e_{m,\cT_N}(f)_p\leq D\max\{\tau^{\frac 1 \tau},1/\tau\}
N^{1/p}(N/2-N_0)^{-1/\tau}(|f|_{W^{m,\tau}} +\e |\Omega|^{1/\tau}).
$$
Taking $N>4N_0$ and remarking that 
$\e>0$ is arbitrary, we conclude that
\iref{CMlimsupsigmaNisogreedy} holds
with $C=4^{1/\tau}D\max\{\tau^{\frac 1 \tau},1/\tau\}$. \hfill $\Box$
\end{proof}
\noindent
\nl
Theorems \ref{CMlowerisotheo} and \ref{CMupperisotheo} reveal that for smooth enough
functions, the numerical quantity that governs the rate of convergence $N^{-m/d}$
in the $L^p$ norm of piecewise polynomial approximations on adaptive isotropic partitions
is exactly $|f|_{W^{m,\tau}}$. In a similar way one would 
obtain that the same rate for quasi-uniform partitions is governed by
the quantity $|f|_{W^{m,p}}$. Note however that these results are 
of asymptotic nature since they involve $\limsup$ and $\liminf$ as $N\to +\infty$,
in contrast to Theorem \ref{CMtheoisogreedy}. The results dealing
with piecewise polynomial 
approximation on anisotropic adaptive partitions that we present in the
next sections are of a similar asymptotic nature.

\section{Anisotropic piecewise constant approximation on rectangles}

We first explore a simple case of 
adaptive approximation on anisotropic partitions
in two space dimensions. More precisely, we
consider piecewise constant approximation
in the $L^p$ norm
on adaptive partitions by rectangles
with sides parallel to the $x$ and $y$ axes.
In order to build such partitions, 
$\Omega$ cannot be any polygonal domain,
and for the sake of simplicity we fix it to be the unit square:
$$
\Omega=[0,1]^2.
$$
The family $\cA_N$ consists therefore of all partitions of $\Omega$
of at most $N$ rectangles of the form
$$
T=I\times J,
$$
where $I$ and $J$ are intervals contained in $[0,1]$. This type
of adaptive anisotropic partitions suffers from a strong coordinate
bias due to the special role of the $x$ and $y$ direction:
functions with sharp transitions on line edges are better
approximated when these eges are parallel to the $x$ and $y$ axes.
We shall remedy this defect in \S 5 by considering adaptive
piecewise polynomial approximation on anisotropic partitions 
consisting of triangles, or simplices in higher dimension. Nevertheless,
this first simple example is already instructive. In particular,
it reveals that the numerical quantity governing the
rate of approximation has an inherent non-linear structure.
Throughout this section, we assume that $f$ belongs to $C^1([0,1]^2)$.

\subsection{A heuristic estimate}

We first establish an error estimate which is based on the heuristic
assumption that the partition is sufficiently fine so that we
may consider that $\nabla f$ is constant on each $T$,
or equivalently $f$ coincides with an affine function $q_T\in \PP_1$
on each $T$. We thus first study the local $L^p$ approximation error 
on $T=I\times J$ for an affine function of the form
$$
q(x,y)=q_0+q_x x+ q_y y.
$$
Denoting by $\bq(x,y):=q_x x+ q_y y$ the homogeneous linear part of $q$,
we first remark that 
\be
e_{1,T}(q)_p=e_{1,T}(\bq)_p,
\label{CMqbq}
\ee
since $q$ and $\bq$ differ by a constant. We thus concentrate on
$e_{1,T}(\bq)_p$ and discuss the shape of $T$ that minimizes this 
error when the area $|T|=1$ is prescribed. We associate to
this optimization problem a function $K_{p}$ that acts on the space of linear functions 
according to
\be
K_{p}(\bq)=\inf_{|T|=1}e_{1,T}(\bq)_p.
\label{CMshapoptrect1}
\ee
As we shall explain further, the above infimum may or may not be attained.

We start by some observations that can be derived
by elementary change of variable. If $a+T$ is a translation of $T$, then
\be
e_{1,a+T}(\bq)_p=e_{1,T}(\bq)_p
\label{CMinvartransrect}
\ee 
since $\bq$ and $\bq(\cdot-a)$ differ by a constant.
Therefore, if $T$ is a minimizing rectangle in
\iref{CMshapoptrect1}, then $a+T$ is also one. If $hT$ is a dilation of $T$, then 
\be
e_{1,hT}(\bq)_p= h^{2/p+1}e_{1,T}(\bq)_p
\label{CMinvarhomrect}
\ee
Therefore, if we are interested in minimizing the error 
for an area $|T|=A$, we find that
\be
\inf_{|T|=A}e_{1,T}(q)_p=A^{1/\tau}K_{p}(\bq),\;\;    \frac 1 \tau:=\frac 1  p+\frac 1 2
\label{CMshapoptrect}
\ee
and the minimizing rectangles for \iref{CMshapoptrect}
are obtained by rescaling the minimizing rectangles for \iref{CMshapoptrect1}.

In order to compute $K_{p}(\bq)$, we thus consider a rectangle $T=I\times J$
of unit area which barycenter is the origin.
In the case $p=\infty$, using the notation $X:=|q_x|\,|I|/2$ and $Y:=|q_y|\,|J|/2$,
we obtain
$$
e_{1,T}(\bq)_\infty=X+Y.
$$
We are thus interested in the minimization of
the function $X+Y$ 
under the constraint $XY=|q_x q_y|/4$.
Elementary computations show that when $q_xq_y\neq 0$, the infimum
is attained when $X=Y=\frac 1 2\sqrt {|q_y q_x|}$ which yields
$$
|I|=\sqrt {\frac {|q_y|}  {|q_x|}}\;\; {\rm and }\;\; |J|=\sqrt {\frac {|q_x|}  {|q_y|}}.
$$
Note that the optimal aspect ratio is given by the simple relation
\be
\frac {|I|}{|J|}= \frac {|q_y|}{|q_x|},
\label{CMoptiaspectrect}
\ee
which expresses the intuitive fact that the refinement should be
more pronounced in the direction where the function
varies the most. Computing $e_{1,T}(q)_\infty$ for such an optimized rectangle, we find that 
\be
K_{\infty}(\bq)= \sqrt {|q_y q_x|}.
\label{CMlocalrectinf1}
\ee
In the case $p=2$, we find that
$$
\begin{array}{ll}
e_{1,T}(\bq)_2^2&=\int_{-|I|/2}^{|I|/2}\int_{-|J|/2}^{|J|/2} |q_x x+q_y y|^2 dy \, dx \\
&=\int_{-|I|/2}^{|I|/2}\int_{-|J|/2}^{|J|/2} (q_x^2 x^2+q_y^2 y^2+ 2q_xq_y xy) dy \, dx \\
&=4\int_{0}^{|I|/2}\int_{0}^{|J|/2} (q_x^2 x^2+q_y^2 y^2) dy \, dx \\
&= \frac 4 3 (q_x^2   (|I|/2)^3|J|/2+q_y^2 (|J|/2)^3 |I|/2 )\\
& = \frac 1 {3}(X^2+Y^2),
\end{array}
$$
where we have used the fact that $|I|\, |J|=1$. We now want to
minimize 
the function $X^2+Y^2$ 
under the constraint $XY=|q_x q_y|/4$.
Elementary computations again show that when $q_xq_y\neq 0$, the infimum
is again attained when $X=Y=\frac 1 2\sqrt {|q_y q_x|}$,
and therefore leads to the same aspect ratio given by
\iref{CMoptiaspectrect}, and the value
\be
K_{2}(\bq)=\frac 1{\sqrt 6} \sqrt {|q_x q_y|}.
\label{CMlocalrectinf2}
\ee
For other values of $p$ the computation of
$e_{1,T}(\bq)_p$
is more tedious, but leads to a same conclusion: the optimal
aspect ratio is given by \iref{CMoptiaspectrect} and the function $K_p$ has the general form
\be
K_{p}(\bq)=C_p \sqrt {|q_x q_y|},
\label{CMlocalrect1}
\ee
with $C_p:=\(\frac {2}{(p+1)(p+2)}\)^{1/p}$.
Note that the optimal shape of $T$ does
not depend on the $L^p$ metric in which we measure the error.

By \iref{CMqbq}, \iref{CMinvartransrect} and \iref{CMinvarhomrect}, we find
that for shape-optimized triangles of arbitrary area, the error
is given by
\be
e_{1,T}(q)_p=|T|^{1/\tau}K_{p}(\bq)_p=C_p \sqrt {|q_y q_x|}|T|^{1/\tau},
\label{CMlocalrect}
\ee
Note that $C_p$ is uniformly bounded for all $p\geq 1$. 

In the case where $q\neq 0$ but $q_xq_y=0$, the infimum
in \iref{CMshapoptrect1} is not attained, and the rectangles
of a minimizing 
sequence tend to become infinitely long
in the direction where $q$ is constant. 
We ignore at the moment this degenerate case.

Since we have assumed that $f$ coincides with an affine function on $T$, 
the estimate \iref{CMlocalrect} yields
\be
e_{1,T}(f)_p= C_p\left \| \sqrt {|\partial_x f \partial_y f|}\right \|_{L^\tau(T)}
=\|K_{p}(\nabla f)\|_{L^\tau} ,\;\; \frac 1 \tau:=\frac 1 p +\frac 1 2.
\label{CMlocalrectaniso}
\ee
where we have identifed $\nabla f$ to the linear function 
$(x,y)\mapsto x \partial_x f + y \partial_y f $.
This local estimate should be compared to those
which were discussed in \S 3.1 for isotropic elements: 
in the bidimensional case, the estimate \iref{CMlocalinv}
of Theorem \ref{CMlocalestheo} can be restated as
$$
e_{1,T}(f)_p\leq C \|\nabla f \|_{L^\tau (T)},\; \; \frac 1 \tau:=\frac 1 p +\frac 1 2.
$$
The improvement in \iref{CMlocalrectaniso} comes the fact that
$\sqrt {|\partial_x f \partial_y f|}$ may be substantially smaller than
$|\nabla f|$ when $|\partial_x f|$ and $|\partial_y f|$ have
different order of magnitude which reflects an anisotropic behaviour
for the $x$ and $y$ directions. However, let us keep in mind
that the validity of \iref{CMlocalrectaniso} is only when
$f$ is identified to an affine function on $T$. 

Assume now that the partition $\cT_N$ is built
in such a way that all rectangles have optimal
shape in the above described sense, and obeys
in addition the error equidistribution principle, which 
by \iref{CMlocalrectaniso} means that
$$
\|K_{p}(\nabla f)\|_{L^\tau(T)}=\eta,\;\; T\in\cT_N.
$$
Then, we have on the one hand that
$$
e_{1,\cT_N}(f)_p\leq \eta N^{1/p},
$$
and on the other hand, that
$$
N\eta^\tau \leq  \| K_{p}(\nabla f)\|_{L^\tau}^\tau.
$$
Combining the two above, and using the relation $\frac 1 \tau:=\frac 1 p +\frac 1 2$,
we thus obtain the error estimate
\be
\sigma_N(f)_p \leq N^{-1/2}\| K_{p}(\nabla f)\|_{L^\tau}.
\label{CMidealanirecterr}
\ee
This estimate should be compared  with those which were
discussed in \S 3.2 for adaptive partition with isotropic
elements: for piecewise constant functions on adaptive
isotropic partitions in the two dimensional case, the estimate
\iref{CMsigmaNtisoequid} can be restated as
$$
\sigma_N(f)_p\leq CN^{-1/2}\|\nabla f\|_{L^{\tau}}, \;\; \frac 1 \tau=\frac 1 p+\frac 1 2.
$$
As already observed for local estimates, the improvement
in \iref{CMlocalrectaniso} comes from the fact that 
$|\nabla f|$ is replaced by the possibly much smaller $\sqrt {|\partial_x f \partial_y f|}$.
It is interesting to note that the quantity
$$
A_p(f):=\| K_{p}(\nabla f)\|_{L^\tau}=C_p\left \| \sqrt {|\partial_x f \partial_y f|}\right \|_{L^\tau},
$$ 
is strongly nonlinear in the sense that it does not satisfy for any $f$ and $g$ an
inequality of the type $A_p(f+g)\leq C(A_p(f)+A_p(g))$, even with $C>1$. This
reflects the fact that two functions $f$ and $g$ may be well approximated
by piecewise constants on anisotropic rectangular partitions while their
sum $f+g$ may not be. 

\subsection{A rigourous estimate}

We have used heuristic arguments to derive the estimate
\iref{CMidealanirecterr}, and a simple example shows that 
this estimate cannot hold as such: if $f$ is a non-constant function that
only depends on the variable $x$ or $y$, the quantity $A_p(f)$ vanishes
while the error $\sigma_N(f)_p$ may be non-zero. 
In this section, we prove a valid estimate by a rigourous derivation. The price
to pay is in the asymptotic nature of the new estimate, which has
a form similar to those obtained in \S 3.4.

We first introduce a ``tamed'' variant of the function $K_{p}$, in which
we restrict the search of the infimum to rectangles of limited diameter.
For $M>0$, we define
\be
K_{p,M}(\bq)=\min_{|T|=1,h_T\leq M}e_{1,T}(\bq)_p.
\label{CMshapoptrect1M}
\ee
In contrast to the definition of $K_p$, the above minimum
is always attained, due to the compactness in the Hausdorff
distance of the set of 
rectangles of area $1$, diameter less or equal to $M$, and centered
at the origin. It is also not difficult to check that
the functions $\bq\mapsto e_{1,T}(\bq)_p$ are uniformly
Lipschitz
continuous for all $T$ of area $1$ and diameter less than $M$: 
there exists a constant $C_M$ such that
\be
|e_{1,T}(\bq)_p- e_{1,T}(\t\bq)_p| \leq C_M|\bq-\t \bq|, 
\label{CMkmplip}
\ee
where $|\bq|:=(q_x^2+q_y^2)^{1/2}$. In turn $K_{p,M}$
is also Lipschitz continuous with constant $C_M$.
Finally, it is obvious that
$K_{p,M}(\bq)\to K_p(\bq)$ as $M\to +\infty$.

If $f$ is a $C^1$ function, we denote by
$$
\omega(\delta):=\max_{|z-z'|\leq \delta}|\nabla f(z)-\nabla f(z')|,
$$
the modulus of continuity of $\nabla f$, which satisfies $\lim_{\delta\to 0}\omega(\delta)=0$.
We also define for all $z\in \Omega$
$$
q_z(z')=f(z)+\nabla f\cdot (z'-z),
$$
the Taylor polynomial of order $1$ at $z$. We identify
its linear part to the gradient of $f$ at $z$:
$$
\bq_z=\nabla f(z).
$$
We thus have
$$
|f(z')-q_z(z')|\leq |z-z'|\omega(|z-z'|).
$$
At each point $z$, we denote by $T_M(z)$ a rectangle of area $1$ which is shape-optimized
with respect to the gradient of $f$ at $z$
in the sense that it solves \iref{CMshapoptrect1M} with $\bq=\bq_z$.
The following results gives an estimate of the local
error for $f$ for such optimized triangles.

\begin{lemma}
\label{CMlemmalocalrect}
Let $T=a+hT_M(z)$ be a rescaled and shifted version of $T_M(z)$. We then have
for any $z'\in T$
$$
e_{1,T}(f)_p\leq (K_{p,M}(\bq_{z'})+B_M\omega(\max\{|z-z'|,h_T\}))|T|^{1/\tau},
$$
with $B_M:=2C_M+M$.
\end{lemma}

\begin{proof}
For all $z,z'\in \Omega$, we have
$$
\begin{array}{ll} 
e_{1,T_M}(\bq_{z'}) &\leq e_{1,T_M}(\bq_{z})+C_M|\bq_z-\bq_{z'}| \\
& =K_{p,M}(\bq_z)+C_M|\bq_z-\bq_{z'}| \\
&\leq K_{p,M}(\bq_{z'})+2C_M|\bq_z-\bq_{z'}| \\
&\leq K_{p,M}(\bq_{z'})+2C_M\omega(|z-z'|).
\end{array}
$$
We then observe that if $z'\in T$
$$
\begin{array}{ll}
e_{1,T}(f)_p &\leq e_{1,T}(\bq_{z'})+\|f-q_{z'}\|_{L^p(T)}\\
& \leq e_{1,T_M}(\bq_{z'})|T|^{1/\tau}+\|f-q_{z'}\|_{L^\infty(T)}|T|^{1/p}\\
& \leq (K_{p,M}(\bq_{z'})+2C_M\omega(|z-z'|))|T|^{1/\tau}
+h_T\omega(h_T)|T|^{1/p} \\
&\leq 
(K_{p,M}(\bq_{z'})+2C_M\omega(|z-z'|)+M\omega(h_T))|T|^{1/\tau},
\end{array}
$$
which concludes the proof. \hfill $\Box$
\end{proof}
$\;$
\nl
We are now ready to state our main convergence theorem.

\begin{theorem}
\label{CMupperecttheo}
For piecewise constant approximation on adaptive anisotropic
partitions on rectangles, we have 
\be
\limsup_{N\to +\infty} \; N^{1/2} \sigma_N(f)_p \leq  \|K_p(\nabla f)\|_{L^\tau}.
\label{CMlimsupsigmaNanisorect}
\ee
for all $f\in C^1([0,1]^2)$.
\end{theorem}

\begin{proof} We first fix some number $\delta>0$ and $M>0$
that are later pushed towards $0$ and $+\infty$ respectively.
We define a uniform partition $\cT_\delta$ of $[0,1]$ into squares $S$
of diameter $h_S\leq \delta$, for example by 
$j_0$ iterations of uniform dyadic refinement, where $j_0$ 
is chosen large enough such that $2^{-j_0+1/2}\leq \delta$.
We then build partitions $\cT_N$ by further decomposing the 
square elements of $\cT_\delta$ in an anisotropic
way. For each $S\in\cT_\delta$, we pick an
arbitrary point $z_S\in S$ (for example the barycenter of $S$)
and consider the Taylor polynomial $q_{z_S}$
of degree $1$ of $f$ at this point. We denote
by $T_S=T_M(\bq_{z_S})$ the rectangle of area $1$ such that,
$$
e_{1,T_S}(\bq_{z_S})_p=\min_{|T|=1,h_T\leq M} e_{1,T}(\bq_{z_S})_p
=K_{p,M}(\bq_{z_S}).
$$
For $h>0$, we rescale this rectangle according to 
$$
T_{h,S}=h(K_{p,M}(\bq_{z_S})+(B_M+C_M)\omega(\delta)+\delta)^{-\tau/2} T_S.
$$
and we define $\cT_{h,S}$ as the tiling of the plane by 
$T_{h,S}$ and its translates. We assume that $hC_A\leq \delta$ 
so that $h_T\leq \delta$ for all
$T\in\cT_{h,S}$ and all $S$. Finally, we define the partition 
$$
\cT_N=\{T\cap S\, ; \; T\in \cT_{h,S}\;{\rm and}\; S\in \cT_\delta\}.
$$
We first estimate the local approximation error. By lemma
\iref{CMlemmalocalrect}, we obtain that
for all $T\in\cT_{h,S}$ and $z'\in T\cap S$
$$
\begin{array}{ll}
e_{1,T\cap S}(f)_p & \leq e_{1,T}(f)_p\\
& \leq (K_{p,M}(\bq_{z'})+B_M\omega(\delta))|T|^{1/\tau} \\
& \leq h^{2/\tau}
(K_{p,M}(\bq_{z_S})+(B_M+C_M)\omega(\delta))
(K_{p,M}(\bq_{z_S})+(B_M+C_M)\omega(\delta)+\delta)^{-1}\\
&\leq h^{2/\tau}
\end{array}
$$
The rescaling has therefore the effect of equidistributing
the error on all rectangles of $\cT_N$, and the global approximation error is bounded by
\be
e_{1,\cT_N}(f)_p \leq N^{1/p}h^{2/\tau}
\label{CMestimglobrect}
\ee
We next estimate the number of rectangles
$N=\#(\cT_N)$, which behaves like 
$$
\begin{array}{ll}
N&=(1+\eta(h))\sum_{S\in \cT_{\delta}} \frac {|S|}{|T_{h,S}|}\\
&=(1+\eta(h)) h^{-2}\sum_{S\in \cT_{\delta}}|S| (K_{p,M}(\bq_{z_S})+(B_M+C_M)\omega(\delta)+\delta)^\tau\\
&=(1+\eta(h))  h^{-2}\sum_{S\in \cT_{\delta}}\int_{S}(K_{p,M}(\bq_{z_S})+(B_M+C_M)\omega(\delta)+\delta)^{\tau},
\end{array}
$$
where $\eta(h)\to 0$ as $h\to 0$. Recalling that $K_{p,M}(\bq_{z_S})$ is Lipschitz continuous
with constant $C_M$, it follows that
\be
N\leq (1+\eta(h))  h^{-2}\int_{\Omega}(K_{p,M}(\bq_{z})+(B_M+2C_M)\omega(\delta)+\delta)^{\tau}.
\label{CMestimNrect}
\ee
Combining \iref{CMestimglobrect} and \iref{CMestimNrect},
we have thus obtained
$$
N^{1/2} e_{1,\cT_N}(f)_p \leq  (1+\eta(h)) ^{1/\tau}\|K_{p,M}(\bq_z)+(B_M+2C_M)\omega(\delta)+\delta\|_{L^\tau}.
$$
Observing that for all $\e>0$, we can choose $M$ large enough and $\delta$
and $h$ small enough
so that
$$
(1+\eta(h)) ^{1/\tau}\|K_{p,M}(\bq_z)+(B_M+2C_M)\omega(\delta)+\delta\|_{L^\tau}\leq
\|K_{p,M}(\bq_z)\|_{L^\tau}+\e,
$$
this concludes the proof.
\hfill $\Box$
\end{proof}
$\;$
\nl
In a similar way as in Theorem \ref{CMlowerisotheo},
we can establish a lower estimate on $\sigma_N(f)$,
which reflects the saturation rate $N^{-1/2}$ of the method,
and shows that the numerical quantity that governs this rate 
is exactly equal to $\|K_p(\nabla f)\|_{L^\tau}$. We again
impose a slight restriction on the set
$\cA_N$ of admissible partitions, assuming that
the diameter of all elements decreases 
as $N\to +\infty$, according to 
\be
\max_{T\in \cT_N} h_T \leq AN^{-1/2},
\label{CMunidecaydiamrect}
\ee
for some $A>0$ which may be arbitrarily large.

\begin{theorem}
\label{CMlowerisotheorect}
Under the restriction {\rm \iref{CMunidecaydiamrect}}, we
have
\be
\liminf_{N\to +\infty} \; N^{1/2}\sigma_N(f)_p \geq \|K_p(\nabla f)\|_{L^\tau}
\label{CMliminfaniso}
\ee
for all $f\inÊC^1(\Omega)$, where  $\frac 1 \tau:=\frac 1 p+\frac 1 2$.
\end{theorem}

\begin{proof} We assume here $p<\infty$. The case $p=\infty$
can be treated by a simple modification of the argument.
Here, we need a lower estimate for the local approximation error,
which is a counterpart to Lemma \ref{CMlemmalocalrect}.
We start by remarking that for all rectangle $T\in \Omega$ and $z\in T$, we have
$$
|e_{1,T}(f)_p-e_{1,T}(q_z)_p|\leq \|f-q_z\|_{L^p(T)}\leq  |T|^{1/p}h_T\omega(h_T),
$$
and therefore
$$
e_{1,T}(f)_p\geq e_{1,T}(q_z)_p-|T|^{1/p}h_T\omega(h_T) \geq K_p(\bq_z)|T|^{1/\tau}-|T|^{1/p}h_T\omega(h_T)
$$
Then, using the fact that if $(a,b,c)$ are positive  numbers such that
$a\geq b-c$ one has $a^p\geq b^p-pcb^{p-1}$, we find that
$$
\begin{array}{ll}
e_{1,T}(f)_p^p & \geq K_p(\bq_z)^p|T|^{p/\tau}- pK_p(\bq_z)^{p-1}|T|^{(p-1)/\tau}|T|^{1/p}h_T\omega(h_T)\\
&=K_p(\bq_z)^p|T|^{1+p/2}-pK_p(\bq_z)^{p-1}|T|^{1+(p-1)/2} h_T\omega(h_T),
\end{array}
$$
Defining $C:=p\max_{z\in \Omega}K_p(\bq_z)^{p-1}$ and remarking that $|T|^{(p-1)/2}\leq h^{p-1}$, 
this leads
to the estimate
$$
e_{1,T}(f)_p^p \geq K_p(\bq_z)^p|T|^{1+p/2}- Ch_T^p|T|\omega(h_T).
$$
Since we work under the assumption
\iref{CMunidecaydiamrect}, we can rewrite this estimate as
\be
e_{1,T}(f)_p^p \geq K_p(\bq_z)^p|T|^{1+p/2}- C|T|N^{-p/2}\e(N),
\label{CMloclowerrect}
\ee
where $\e(N)\to 0$ as $N\to \infty$. Integrating \iref{CMloclowerrect} over $T$,
gives
$$
e_{1,T}(f)_p^p \geq \int_T (K_p(\bq_z)^p|T|^{p/2}- CN^{-p/2}\e(N))dz.
$$
Summing over all rectangles $T\in\cT_N$ and denoting by $T_z$ the triangle that contains $z$, we
thus obtain
\be
e_{1,\cT_N}(f)_p^p\geq \int_\Omega K_p(\nabla f(z))^p|T_z|^{p/2}dz - C|\Omega| N^{-p/2}\e(N).
\label{CMgloblowerrect}
\ee
Using H\"older inequality, we find that
\be
\int_\Omega K_p(\nabla f(z))^\tau dz \leq 
\(\int_\Omega K_p(\nabla f(z))^p|T_z|^{p/2}dz\)^{\tau/p}
\(\int_\Omega |T_z|^{-1}dz\)^{1-\tau/p}.
\label{CMholderrect}
\ee
Since $\int_\Omega |T_z|^{-1}dz=\#(\cT_N)=N$, it follows that
$$
e_{1,\cT_N}(f)_p^p\geq \|K_p(\nabla f)\|_{L^\tau}^pN^{-p/2} - C|\Omega| N^{-p/2}\e(N),
$$
which concludes the proof. \hfill $\Box$
\end{proof}

\begin{remark}
The H\"older inequality \iref{CMholderrect} which is used in the above proof
becomes an equality when the quantity $K_p(\nabla f(z))^p|T_z|^{p/2}$
and $|T_z|^{-1}$ are proportional, i.e. $K_p(\nabla f(z))|T|^{1/\tau}$ is constant,
which again reflects the principle of error equidistribution. In summary,
the optimal partitions should combine this principe with locally optimized
shapes for each element.
\end{remark}

\section{Anisotropic piecewise polynomial approximation}

We turn to adaptive
piecewise polynomial approximation on anisotropic partitions 
consisting of triangles, or simplices in higher dimension.
Here $\Omega\subset \R^d$ is a domain that
can be decomposed into such partitions, therefore
a polygon when $d=2$, a polyhedron when $d=3$, etc.
The family $\cA_N$ consists therefore of all partitions of $\Omega$
of at most $N$ simplices. The first estimates of the form
\iref{CMsigmaNaniso} were rigorously established in
\cite{CMcsx} and \cite{CMbbls} in the case of 
piecewise linear element for bidimensional triangulations.
Generalization to higher polynomial degree
as well as higher dimensions were recently proposed
in \cite{CMcao1,CMcao2,CMcao3}
as well as in \cite{CMm}. Here we follow the general
approach of \cite{CMm} to the characterization
of optimal partitions.

\subsection{The shape function}

If $f$ belongs to $C^m(\Omega)$, 
where $m-1$ is the degree of the piecewise polynomials that we use for
approximation, we mimic the heuristic approach
proposed for piecewise constants on rectangles in \S 4.1
by assuming that on each triangle $T$
the relative variation of $d^mf$ is small so that it can be considered
as a constant over $T$. This means that
$f$ is locally identified with its Taylor polynomial
of degree $m$ at $z$, which is defined as
$$
q_z(z'):=f(z)+\nabla f(z)\cdot (z'-z)+\sum_{k=2}^m\frac 1 {k!} d^kf(z)[z'-z,\cdots,z'-z].
$$
If $q\in \PP_m$ is a polynomial of degree $m$, we denote
by $\bq\in\HH^m$ its homogeneous part of degree $m$.
For $q=q_z$ we can identify $\bq_z\in \HH_m$ with $\frac 1 {m!} d^mf(z)$.
Since $\bq-q\in \PP_{m-1}$ we have
$$
e_{m,T}(q)_p=e_{m,T}(\bq)_p.
$$
We optimize the shape of the simplex $T$ with respect to $\bq$
by introducing the function $K_{m,p}$ defined on the space $\HH_m$
\be
K_{m,p}(\bq):=\inf_{|T|=1}e_{m,T}(\bq)_p,
\label{CMshapopttri1}
\ee
where the infimum is taken among all triangles of area $1$.
This infimum may or may not be attained. We refer to
$K_{m,p}$ as the {\it shape function}.
It is obviously a generalization of the
function $K_p$ introduced for piecewise 
constant on rectangles in \S 4.1.

As in the case of rectangles, some elementary
properties of $K_{m,p}$
are obtained by change of variable: if $a+T$ is a shifted version of $T$, then
\be
e_{m,a+T}(\bq)_p=e_{m,T}(\bq)_p
\label{CMinvartranstri}
\ee 
since $\bq$ and $\bq(\cdot-a)$ differ by a polynomial of degree $m-1$,
and that if $hT$ is a dilation of $T$, then 
\be
e_{m,hT}(\bq)_p= h^{d/p+m}e_{m,T}(\bq)_p
\label{CMinvarhomtri.}
\ee
Therefore, if $T$ is a minimizing simplex in
\iref{CMshapopttri1}, then $a+T$ is also one,
and if we are interested in minimizing the error 
for a given area $|T|=A$, we find that
\be
\inf_{|T|=A}e_{m,T}(q)_p=A^{1/\tau}K_{m,p}(\bq),\;\;  \frac 1 \tau:=\frac 1  p+\frac m d
\label{CMshapopttri}
\ee
and the minimizing simplex for \iref{CMshapoptrect}
are obtained by rescaling the minimizing simplex for \iref{CMshapoptrect1}.

Remarking in addition that if $\vp$ is an invertible
linear transform, we then have for all $f$
$$
|{\rm det} (\vp)|^{1/p} e_{m,T}(f\circ\vp)_p=e_{m,\vp(T)}(f)_p,
$$
and using \iref{CMshapopttri}, we also obtain that
\be
K_{m,p}(\bq \circ \vp)=|{\rm det} (\vp)|^m K_{m,p}(\bq)
\label{CMinvarphik}
\ee
The minimizing simplex of area $1$ for $\bq \circ \vp$
is obtained by application of $\vp^{-1}$ followed by
a rescaling by $|{\rm det} (\vp)|^{1/d}$ to the minimizing simplex of 
area $1$ for $\bq$ if it exists. 

\subsection{Algebraic expressions of the shape function}

The identity \iref{CMinvarphik} can be used 
to derive the explicit expression
of $K_{m,p}$ for particular values of $(m,p,d)$,
as well as the exact shape of the minimizing triangle $T$ 
in \iref{CMshapopttri1}. 

We first consider the 
case of piecewise affine elements on two dimensional
triangulations, which corresponds to $d=m=2$. Here
$\bq$ is a quadratic form and we denote by
$\det(\bq)$ its determinant. We also denote by
$|\bq|$ the positive quadratic form associated
with the absolute value of the symmetric matrix 
associated to $\bq$.

If ${\rm det} (\bq)\neq 0$, there exists 
a $\vp$ such that $\bq \circ \vp$ is either $x^2+y^2$ or $x^2-y^2$, up to
a sign change, and we have 
$|{\rm det} (\bq)|=|{\rm det} (\vp)|^{-2}$.
It follows from \iref{CMinvarphik} that $K_{2,p}(\bq)$ has the simple form
\be
K_{2,p}(\bq)=\kappa_p  |{\rm det} (\bq)|^{1/2},
\label{CMk2pexp}
\ee
where $\kappa_p:=K_{2,p}(x^2+y^2)$ 
if ${\rm det} (\bq)>0$ and  $\kappa_p=K_{2,p}(x^2-y^2)$ if ${\rm det} (\bq)<0$. 

The triangle of area $1$ that minimizes the $L^p$ error when $\bq=x^2+y^2$
is the equilateral triangle, which is unique up to rotations.
For $\bq=x^2-y^2$, the triangle that minimizes the $L^p$ error
is unique up to an hyperbolic transformation with eigenvalues $t$ and $1/t$
and eigenvectors $(1,1)$ and $(1,-1)$ for any $t\neq 0$. Therefore,
such triangles may be highly anisotropic, but at least one of them
is isotropic. For example, it can be checked that a triangle of area
$1$ that minimizes the $L^\infty$ error is given
by the half square with vertices $((0,0),(\sqrt 2,0),(0,\sqrt 2))$.
It can also be checked that an equilateral triangle $T$ of area $1$
is a ``near-minimizer'' in the sense that
$$
e_{2,T}(\bq)_p\leq CK_{2,p}(\bq),
$$
where $C$ is a constant independent of $p$.  It follows that
when ${\rm det} (\bq)\neq 0$, the triangles
which are isotropic with respect to the distorted metric induced by $|\bq|$
are ``optimally adapted'' to $\bq$ in the sense
that they nearly minimize the $L^p$ error 
among all triangles of similar area.

In the case when ${\rm det} (\bq)= 0$, which corresponds to
one-dimensional quadratic forms $\bq=(ax+by)^2$, the
minimum in \iref{CMshapopttri1} is not attained and the
minimizing triangles become infinitely long along the
null cone of $\bq$. In that case one has $K_{2,p}(\bq)=0$
and the equality \iref{CMk2pexp} remains therefore valid.

These results easily generalize to piecewise affine functions
on simplicial partitions in higher dimension $d>1$:
one obtains
\be
K_{2,p}(\bq)=\kappa_p  |{\rm det} (\bq)|^{1/d},
\label{CMk2dpexp}
\ee
where $\kappa_p$ only takes a finite number of possible values.
When ${\rm det} (\bq)\neq 0$, the simplices
which are isotropic with respect to the distorted metric induced by $|\bq|$
are ``optimally adapted'' to $\bq$ in the sense
that they nearly minimize the $L^p$ error 
among all simplices of similar volume.

The analysis becomes more delicate 
for higher polynomial degree $m\geq 3$.
For piecewise quadratic elements in dimension two, 
which corresponds to $m=3$ and $d=2$, 
it is proved in \cite{CMm} that
$$
K_{3,p}(\bq) =\kappa_p  |{\rm disc} (\bq)|^{1/4}.
$$
for any {\it homogeneous} polynomial $\bq\in \HH_{3}$, where
$$
{\rm disc}(a x^3+ b x^2 y+ c x y^2+ d y^3):= b^2 c^2 - 4 a c^3 - 4 b^3 d + 18 a b c d - 27 a^2 d^2,
$$
is the usual discriminant and $\kappa_p$ only takes two values depending
on the sign of ${\rm disc} (\bq)$. The analysis that leads
to this result also describes the shape of the
triangles which are optimally adapted to $\bq$.

For other values of $m$ and $d$, the exact expression
of $K_{m,p}(\bq)$ is unknown, but it is possible to give
equivalent versions in terms of polynomials $Q_{m,d}$ in the
coefficients of $\bq$, in the following sense: for all $\bq\in\HH_m$
$$
c_1(Q_{m,d}(\bq))^{\frac 1 r}\leq K_{3,p}(\bq)\leq c_2 (Q_{m,d}(\bq))^{\frac 1 r},
$$
where $r:={\rm deg}(Q_{m,d})$, see \cite{CMm}.

\begin{remark}
It is easily checked that the shape functions $\bq\mapsto K_{m,p}(\bq)$ are equivalent for all $p$
$p$ in the sense that there exist constant $0<C_1\leq C_2$ that
only depend on the dimension $d$ such that 
$$
C_1 K_{m,\infty }(\bq) \leq K_{m,p}(\bq)\leq C_2 K_{m,\infty}(\bq),
$$
for all $\bq\in \HH_m$ and $p\geq 1$. In particular a minimizing triangle
for $K_{m,\infty}$ is a near-minimizing triangle for $K_{m,p}$. In that
sense, the optimal shape of the element does not strongly depend on $p$.
\end{remark}

\subsection{Error estimates}

Following at first a similar heuristics as in \S 4.1
for piecewise constants on rectangles, we 
assume that the triangulation $\cT_N$ is
such that all its triangles $T$ have optimized shape
with respect to the polynomial $q$ that coincides
with $f$ on $T$.

According to \iref{CMshapopttri}, we thus have
for any triangle $T\in\cT$,
$$
e_{m,T}(f)_p=|T|^{\frac 1 \tau} K_{m,p}(\bq)=\left \|K_{m,p}\(\frac {d^mf}{m!}\)\right \|_{L^\tau(T)}.
$$
We then apply the principle of {\it error equidistribution},
assuming that
$$
e_{m,T}(f)_p=\eta,
$$
From which it follows that $e_{m,\cT_N}(f)_p\leq N^{1/p}\eta$
and
$$
N\eta^\tau\leq \left \|K_{m,p}\(\frac {d^mf}{m!}\) \right \|_{L^\tau}^\tau,
$$
and therefore
\be
\sigma_N(f)_p  \leq N^{-m/d}  \left \|K_{m,p}\(\frac {d^mf}{m!}\)\right \|_{L^\tau}.
\label{CMfalsesigman}
\ee
This estimate should be compared to
\iref{CMsigmaNtisoequid} which was obtained for 
adaptive partitions with elements of isotropic shape.
The essential difference is in the quantity $K_{m,p}\(\frac {d^mf}{m!}\)$
which replaces $d^m f$ in the $L^\tau$ norm, and which may
be significantly smaller. Consider for example
the case of piecewise affine elements, for which 
we can combine \iref{CMfalsesigman} with \iref{CMk2dpexp} to obtain
\be
\sigma_N(f)_p  \leq CN^{-2/d} \left \||{\rm det}(d^2f)|^{1/d}\right \|_{L^\tau}.
\label{CMfalseaffsigman}
\ee
In comparison to \iref{CMsigmaNtisoequid}, 
the norm of the hessian $|d^2f|$ is replaced by the quantity $|{\rm det}(d^2f)|^{1/d}$ 
which is geometric mean of its eigenvalues, a quantity
which is significantly smaller when two eigenvalues have
different orders of magnitude which reflects an anisotropic behaviour
in $f$.

As in the case of piecewise constants on rectangles, the example
of a function $f$ depending on only one variable 
shows that the estimate
\iref{CMfalseaffsigman} cannot hold as such. 
We may obtain some valid estimates by 
following the same approach as
in Theorem \ref{CMupperecttheo}.
This leads to the following result
which is established in \cite{CMm}.

\begin{theorem}
\label{CMuppersimplextheo}
For piecewise polynomial approximation on adaptive anisotropic
partitions into simplices, we have 
\be
\limsup_{N\to +\infty} \; N^{m/d} \sigma_N(f)_p \leq  
C\left \|K_{m,p}\(\frac {d^m f}{m!}\)\right \|_{L^\tau},\;\; \frac 1 \tau:=\frac 1 p+\frac m d,
\label{CMlimsupsigmaNanisosimp}
\ee
for all $f\in C^m(\Omega)$. The constant $C$ can be chosen equal to $1$ in the case
of two-dimensional triangulations $d=2$.
\end{theorem}

The proof of this theorem follows exactly the same line as the one of 
Theorem \ref{CMupperecttheo}: we build a sequence of partitions
$\cT_N$ by refining the triangles $S$ of a 
sufficiently fine quasi-uniform partition $\cT_\delta$,
intersecting each $S$ with a partition $\cT_{h,S}$
by elements with shape optimally adapted to the
local value of $d^m f$ on each $S$. The constant $C$ can 
be chosen equal to $1$ in the two-dimensional case, due
to the fact that it is then possible to build $\cT_{h,S}$
as a tiling of triangles which are all optimally adapted.
This is no longer possible in higher dimension,
which explains the presence of a constant $C=C(m,d)$ larger than $1$.

We may also obtain lower estimates, following
the same approach as in Theorem
\ref{CMlowerisotheorect}: we first
impose a slight restriction on the set
$\cA_N$ of admissible partitions, assuming that
the diameter of the elements decreases 
as $N\to +\infty$, according to 
\be
\max_{T\in \cT_N} h_T \leq AN^{-1/d},
\label{CMunidecaydiamtri}
\ee
for some $A>0$ which may be arbitrarily large.
We then obtain the following result, which proof is
similar to the one of Theorem \ref{CMlowerisotheorect}.

\begin{theorem}
\label{CMlowerisotheotri}
Under the restriction {\rm \iref{CMunidecaydiamtri}}, we
have
\be
\liminf_{N\to +\infty} \; N^{m/d}\sigma_N(f)_p \geq \left \|K_{m,p}\(\frac {d^m f}{m!}\) \right \|_{L^\tau}
\label{CMliminftri}
\ee
for all $f\inÊC^m(\Omega)$, where  $\frac 1 \tau:=\frac 1 p+\frac m d$.
\end{theorem}

\subsection{Anisotropic smoothness and cartoon functions}

Theorem \ref{CMuppersimplextheo}
reveals an improvement over the
approximation results based on adaptive isotropic partitions
in the sense that $\|K_{m,p}\(\frac {d^mf}{m!}\)\|_{L^\tau}$
may be significantly smaller than $\|d^m f\|_{L^\tau}$,
for functions which have an anisotropic behaviour.
However, this result suffers
from two major defects:
\begin{enumerate}
\item
The estimate \iref{CMlimsupsigmaNanisosimp} is asymptotic: it says that for all $\e>0$, there
exists $N_0$ depending on $f$ and $\e$ such that
$$ 
 \sigma_N(f)_p \leq  CN^{-m/d}
\( \left \|K_{m,p}\(\frac {d^m f}{m!}\) \right \|_{L^\tau}+\e \),
$$
for all $N\geq N_0$. However, it does not ensure a uniform bound on $N_0$ which may be
very large for certain $f$.\vspace{0.2cm}
\item
Theorem \ref{CMuppersimplextheo} is based on the assumption 
$f\inÊC^m(\Omega)$, and therefore the estimate \iref{CMlimsupsigmaNanisosimp} only
seems to apply to sufficiently smooth functions.
This is in contrast to the estimates
that we have obtained for adaptive isotropic partitions, which are based
on the assumption that $f\in W^{m,\tau}(\Omega)$ or $f\in B^m_{\tau,\tau}(\Omega)$.
\end{enumerate}
The first defect is due to the fact that a certain amount of refinement
should be performed before the relative variation of $d^m f$ 
is sufficiently small so that there is no ambiguity in defining the optimal
shape of the simplices. It is in that sense unavoidable. 

The second defect raises a legitimate question concerning the validity
of the convergence estimate \iref{CMlimsupsigmaNanisosimp}
for functions which are not in $C^m(\Omega)$. It suggests in particular
to introduce a class of
distributions such that
$$
\left \|K_{m,p}\(\frac {d^m f}{m!}\)\right \|_{L^\tau} <+\infty,
$$
and to try to understand if the estimate remains valid inside this class
which describe in some sense functions which have a certain amount
anisotropic smoothness. The main difficulty is that that this class is not well
defined due to the nonlinear nature of $K_{m,p}\(\frac {d^m f}{m!}\)$.
As an example consider the case of piecewise linear elements
on two dimensional triangulation, that corresponds to $m=d=2$.
In this case, we have seen that $K_{2,p}(\bq)=\kappa_p \sqrt{|{\rm det}(\bq)|}$.
The numerical quantity that governs the approximation rate $N^{-1}$
is thus
$$
A_p(f):= \left \|\sqrt {|{\rm det}(d^2f)|}\right \|_{L^\tau}, \;\; \frac 1 \tau=\frac 1 p +1.
$$
However, this quantity cannot be defined in the distribution
sense since the product of two distributions is generally ill-defined.
On the other hand, it is known that the rate $N^{-1}$
can be achieved for functions which do not have
$C^2$ smoothness, and which may even be discontinuous
along curved edges. Specifically, we say that $f$ is a
{\it cartoon function} on $\Omega$ if it is almost everywhere of the form 
$$
f=\sum_{1\leq i\leq k} f_i \Chi_{\Omega_i},
$$
where the $\Omega_i$ are disjoint open sets with piecewise $C^2$ boundary, no cusps (i.e. satisfying an interior and exterior cone condition), and such that $\overline \Omega = \cup_{i=1}^k \overline \Omega_i$, 
and where for each $1\leq i\leq k$, the function $f_i$ is $C^2$ on a neighbourhood of $\overline \Omega_i$.
Such functions are a natural candidates to represent images with sharp edges or 
solutions of PDE's with shock profiles.

Let us consider a fixed cartoon function $f$ on a polygonal domain $\Omega$
associated with a partition $(\Omega_i)_{1\leq i\leq k}$. We define 
$$
\Gamma :=\bigcup_{1\leq i\leq k} \partial \Omega_i,
$$
the union of the boundaries of the $\Omega_i$. The above definition
implies that $\Gamma$ is the disjoint union of a finite set of points $\cP$ and a finite number of open curves $(\Gamma_i)_{1\leq i\leq l}$.
$$
\Gamma = \(\bigcup_{1\leq i\leq l} \Gamma_i\) \cup \cP.
$$
If we consider the approximation of $f$ by piecewise affine
function on a triangulation $\cT_N$ of cardinality $N$, we may 
distinguish two types of elements of $\cT_N$. A 
triangle $T\in \cT_N$ is called ``regular'' if $T\cap \Gamma=\emptyset$, 
and we denote the set of such triangles by $\cT_N^r$. 
Other triangles are called ``edgy'' and their set is denoted by $\cT_N^e$.
We can thus split $\Omega$ according to
$$
\Omega:=(\cup_{T\in \cT_N^r}T) \cup (\cup_{T\in \cT_N^e}T)=\Omega_N^r \cup \Omega_N^e.
$$
We split accordingly the $L^p$ approximation error into
$$
e_{2,\cT_N}(f)_p^p = \sum_{T\in\cT_N^r}e_{2,T}(f)_p^p +\sum_{T\in\cT_N^e}e_{2,T}(f)_p^p.
$$
We may use $\cO(N)$ triangles in $\cT_N^e$ and $\cT_N^r$ (for example $N/2$
in each set). Since $f$ has discontinuities along $\Gamma$, the approximation error on 
the edgy triangles does not tend to zero in $L^\infty$ and $\cT_N^e$ should be chosen so 
that $\Omega_N^e$ has the aspect of a thin layer around $\Gamma$. 
Since $\Gamma$  is a finite union of $C^2$ curves, we can build this layer
of width $\cO(N^{-2})$ and therefore of global area $|\Omega_N^e|\leq C N^{-2}$,
by choosing long and thin triangles in $\cT_N^e$. 
On the other hand, since $f$ is uniformly $C^2$ on $\Omega_N^r$, we may
choose all triangles in $\cT_N^r$ of regular shape 
and diameter $h_T\leq C N^{-1/2}$.
Hence we obtain the following heuristic error estimate, 
for a well designed anisotropic triangulation:
$$
\begin{array}{ll}
e_{2,\cT_N}(f)_p& \leq
 \sum_{T\in\cT_N^r}|T|e_{2,T}(f)_\infty^p
 + \sum_{T\in\cT_N^e}|T|e_{2,T}(f)_\infty^p\\
 &\leq C|\Omega_N^r| (\sup_{T\in\cT_N^r}h_T^2)\|d^2f\|_{L^\infty(\Omega_N^r)}^p
 +C|\Omega_N^e| \|f\|_{L^\infty(\Omega_N^e)}^p,
 \end{array}
 $$
and therefore
\be
e_{2,\cT_N}(f)_p\leq C N^{-\min\{1,2/p\}},
 \label{CMheuristcartoon}
 \ee
where the constant $C$ depends on $\|d^2f\|_{L^\infty(\Omega\sm \Gamma)}$, $\|f\|_{L^\infty(\Omega)}$
and on the number, length and maximal curvature of the $C^2$ curves which constitute $\Gamma$.

These heuristic estimates have been discussed in
\cite{CMmallat} and rigorously proved 
in \cite{CMdeckel}.
Observe in particular that the error
is dominated by the edge contribution 
when $p>2$ and by the smooth contribution
when $p<2$. For the critical value $p=2$ the two contributions have the same order. 

For $p\geq 2$, we obtain the approximation rate $N^{-1}$ 
which suggests that approximation results such as Theorem \ref{CMuppersimplextheo} 
should also apply to cartoon functions and
that the quantity $A_p(f)$ should be
finite for such functions. In some sense, we want to ``bridge the gap''
between results of anisotropic piecewise polynomial
approximation for cartoon functions and for smooth functions. 
For this purpose, we first need to give a proper meaning
to $A_p(f)$ when $f$ is a cartoon function. As already explained,
this is not straightforward, due to the fact that
the product of two distributions has no meaning in general.
Therefore, we cannot define $\det (d^2 f)$ in the distribution sense,
when the coefficients of $d^2 f$ are distributions without sufficient smoothness.

We describe a solution to this problem proposed in \cite{CMcmcart}
which is based on a regularization process. 
In the following, we consider a fixed
radial nonnegative function $\vp$ of unit integral and supported in the unit ball,
and define for all $\delta>0$ and $f$ defined on $\Omega$,
\be
\vp_\delta(z) := \frac 1 {\delta^2} \vp\left(\frac z \delta\right) \text{ and } f_\delta = f * \vp_\delta.
\label{CMdefphih}
\ee
It is then possible to gives a meaning to $A_p(f)$ based on this regularization.
This approach is additionally justified by the fact that sharp curves of discontinuity 
are a mathematical idealisation. In real world applications, 
such as photography, several physical limitations 
(depth of field, optical blurring) impose a certain level of blur on the edges.

If $f$ is a cartoon function on a set $\Omega$, 
and if $x\in \Gamma\sm\cP$, we denote by $[f](x)$ the jump of $f$
at this point.
We also denote by $|\kappa(x)|$ the absolute value of the curvature
at $x$.
For $p\in [1,\infty]$ and $\tau$ defined by $\frac 1 \tau := 1+\frac 1 p$,
we introduce the two quantities
\begin{eqnarray*}
S_p(f) &:= &\left \|\sqrt{|\det (d^2 f) |} \right \|_{L^\tau(\Omega\sm \Gamma)}=A_p(f_{|\Omega\sm \Gamma}),\\
E_p(f) &:= & \|\sqrt{|\kappa|} [f]\|_{L^\tau (\Gamma)},\\
\end{eqnarray*}
which respectively measure the ``smooth part'' and the ``edge part'' of $f$. We
also introduce the constant
\be
C_{p,\vp}:=\|\sqrt {|\Phi\Phi'|}\|_{L^\tau (\R)},\;\; \Phi(x) := \int_{y\in\R} \vp(x,y)dy.
\label{CMdefchi}
\ee
Note that $f_\delta$ is only properly defined on the set
$$
\Omega^\delta:=\{z\in \Omega\; ; \; B(z,\delta)\subset \Omega\},
$$
and therefore, we define
$A_p(f_\delta)$ as the $L^\tau$ norm of $\sqrt{|\det(d^2f_\delta)|}$ on this set.
The following result is proved in \cite{CMcmcart}.

\begin{theorem}
\label{CMthreg}
For all cartoon functions $f$, the quantity $A_p(f_\delta)$ behaves as follows:
\begin{itemize}
\item If $p<2$, then
$$
\lim_{\delta\to 0} A_p(f_\delta) = S_p(f).
$$
\item If $p=2$, then $\tau=\frac 2 3$ and 
$$
\lim_{\delta\to 0} A_2(f_\delta) = (S_2(f)^{\tau}+ E_2(f)^{\tau} C_{2,\vp}^{\tau})^{1/\tau}.
$$
\item If $p>2$, then $A_p(f_\delta) \to \infty$ according to
$$
\lim_{\delta\to 0}  \delta^{\frac 1 2 - \frac 1 p} A_p(f_\delta) = E_p(f) C_{p,\vp}.
\label{equivAphih}
$$
\end{itemize}
\end{theorem}

\begin{remark}
This theorem reveals that as $\delta\to 0$, the 
contribution of the neighbourhood of $\Gamma$ to $A_p(f_\delta)$
is neglectible when $p<2$ and 
dominant when $p>2$, which was already remarked in
the heuristic computation leading to {\rm \iref{CMheuristcartoon}}.
\end{remark}

\begin{remark}
In the case $p=2$, it is interesting to compare the limit expression 
$(S_2(f)^{\tau}+ E_2(f)^{\tau} C_{2,\vp}^{\tau})^{1/\tau}$
with the {\it total variation} $TV(f)=|f|_{BV}$.
For a cartoon function, the total variation also
can be split into a contribution of the smooth part and
a contribution of the edge, according to
$$
TV(f):=\int_{\Omega\sm\Gamma}|\nabla f| +\int_\Gamma |[f]|.
$$
Functions of bounded variation are thus allowed to have jump discontinuities
along edges of finite length. For this reason, $BV$
is frequently used as a natural smoothness space 
to describe the mathematical properties of images.
It is also well known that $BV$ is a regularity space for certain 
hyperbolic conservation law, in the sense that the total variation
of their solutions remains finite for all time $t>0$. 
In recent years, it has been observed that the space $BV$ (and more generally
classical smoothness spaces) do not provide a fully satisfactory description of
piecewise smooth functions arising in the above mentionned applications,
in the sense that the total variation only takes into account
the size of the sets of discontinuities and not their geometric smoothness.
In contrast, we observe that the term
$E_2(f)$ 
incorporates an information on the smoothness of $\Gamma$ through
the presence of the curvature $|\kappa|$. The quantity
$A_2(f)$ appears therefore as a potential substitute
to $TV(f)$ in order to take into account the geometric smoothness
of the edges in cartoon function and images.
\end{remark}

\section{Anisotropic greedy refinement algorithms}

In the two previous sections, we have established
error estimates in $L^p$ norms for the approximation
of a function $f$ by piecewise polynomials on
optimally adapted anisotropic partitions. 
Our analysis reveals that the optimal
partition needs to 
satisfy two intuitively desirable features:
\begin{enumerate}
\item
Equidistribution of the local error.
\item
Optimal shape adaptation of each element based on the local properties of $f$.
\end{enumerate}
For instance, in the case of piecewise affine approximation
on triangulations, these items mean that each triangle
$T$ should be close to equilateral with respect to
a distorted metric induced by the local value of
the hessian $d^2f$. 

From the computational viewpoint, a commonly used strategy
for designing an optimal triangulation consists therefore in
evaluating the hessian $d^2f$ and imposing that 
each triangle is isotropic with respect
to a metric which is properly related to its local value.
We refer in particular to \cite{CMbfgls}
and to \cite{CMbois} where this program is executed 
by different approaches, both based on
Delaunay mesh generation techniques (see also
the software package \cite{CMbamg} which
includes this type of mesh generator).
While these algorithms produce 
anisotropic meshes which are naturally
adapted to the approximated function,
they suffer from two intrinsic limitations:
\begin{enumerate}
\item
They are based on the data of $d^2f$,
and therefore do not apply well
to non-smooth or noisy functions. 
\item
They are non-hierarchical: for $N>M$, the triangulation
$\cT_N$ is not a refinement of $\cT_M$. 
\end{enumerate}
Similar remark apply to anisotropic mesh generation techniques
in higher dimensions or for finite elements of higher degree. 

The need for hierarchical partitions is critical
in the construction of wavelet bases, which play
an important role in applications
to image and terrain data processing,
in particular data compression \cite{CMcddd}.
In such applications, the multilevel structure
is also of key use for the fast encoding
of the information. Hierarchy is also useful in the 
design of optimally converging adaptive
methods for PDE's \cite{CMbdd,CMmns,CMste}.
However, all these developments are so 
far mostly limited to isotropic refinement methods,
in the spirit of the refinement procedures discussed in
\S 3. Let us mention that hierarchical {\it and} anisotropic
triangulations have been investigated in \cite{CMkp},
yet in this work the triangulations are {\it fixed in advance}
and therefore generally not adapted to the approximated function.
\nl
\nl
{\it A natural objective is therefore to design
adaptive algorithmic techniques that combine
hierarchy and anisotropy, that apply
to any function $f\in L^p(\Omega)$, and that
lead to optimally adapted partitions.}
\nl
\nl
In this section, we discuss anisotropic
refinement algorithms which fullfill this objective.
These algorithms have been introduced and studied
in \cite{CMcdhm} for piecewise polynomial approximation
on two-dimensional triangulations. In the
particular case of piecewise affine elements,  it was
proved in \cite{CMcm} that they lead to optimal
error estimates.  The main idea is again to
refine the element $T$ that maximizes the
local error $e_{m,T}(f)_p$, but to allow
several scenarios of refinement for this element.
Here are two typical instances in two dimensions:
\begin{enumerate}
\item
For rectangular partitions, we allow to split each rectangle 
into two rectangles of equal size
by either a vertical or horizontal cut.
There are therefore two splitting scenarios.
\item 
For triangular partitions, we allow to bisect each triangle
from one of its vertex towards the mid-point of the opposite
edge. There are therefore three splitting scenarios.
\end{enumerate}
We display on Figure \ref{CManisopart} two examples
of anisotropic partitions respectively obtained by such
splitting techniques.
\begin{figure}
\centerline{
\includegraphics[width=4cm,height=4cm]{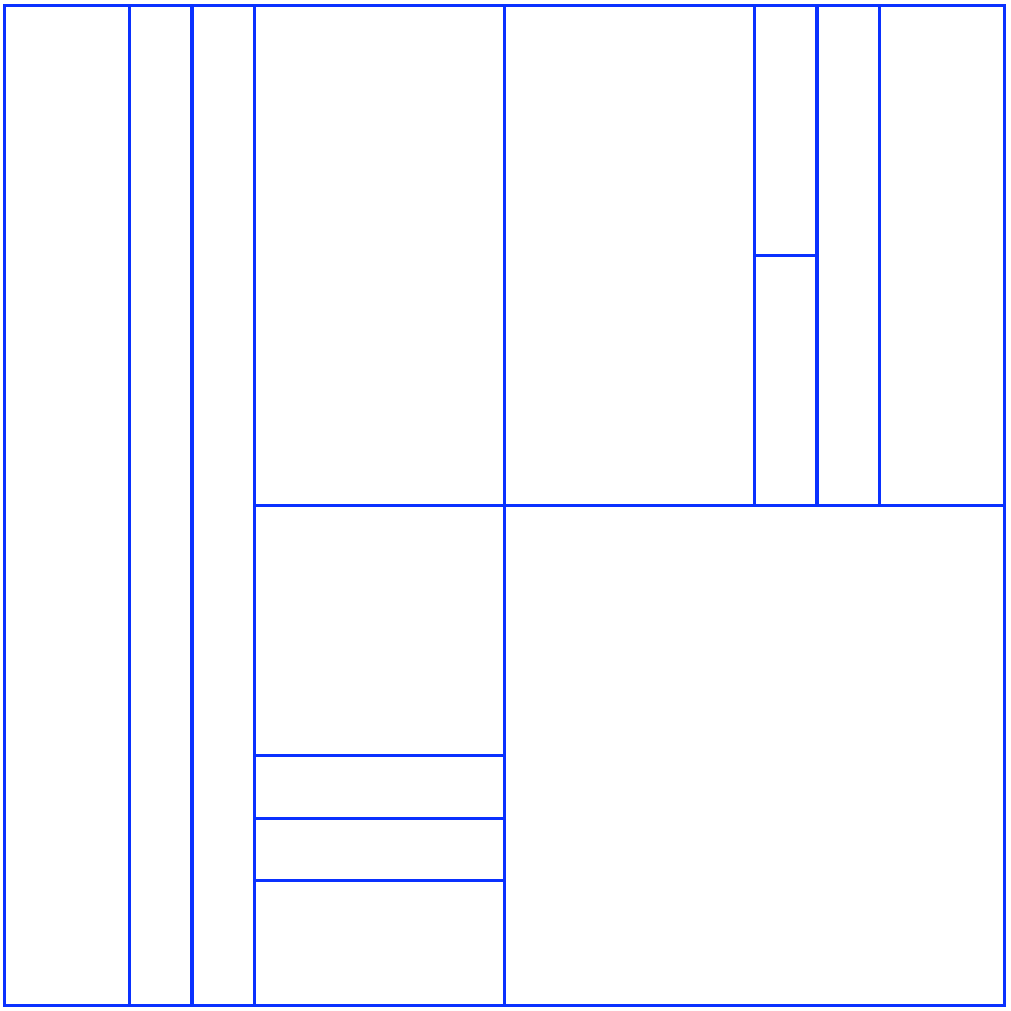}
\hspace{1cm}
\includegraphics[width=4cm,height=4cm]{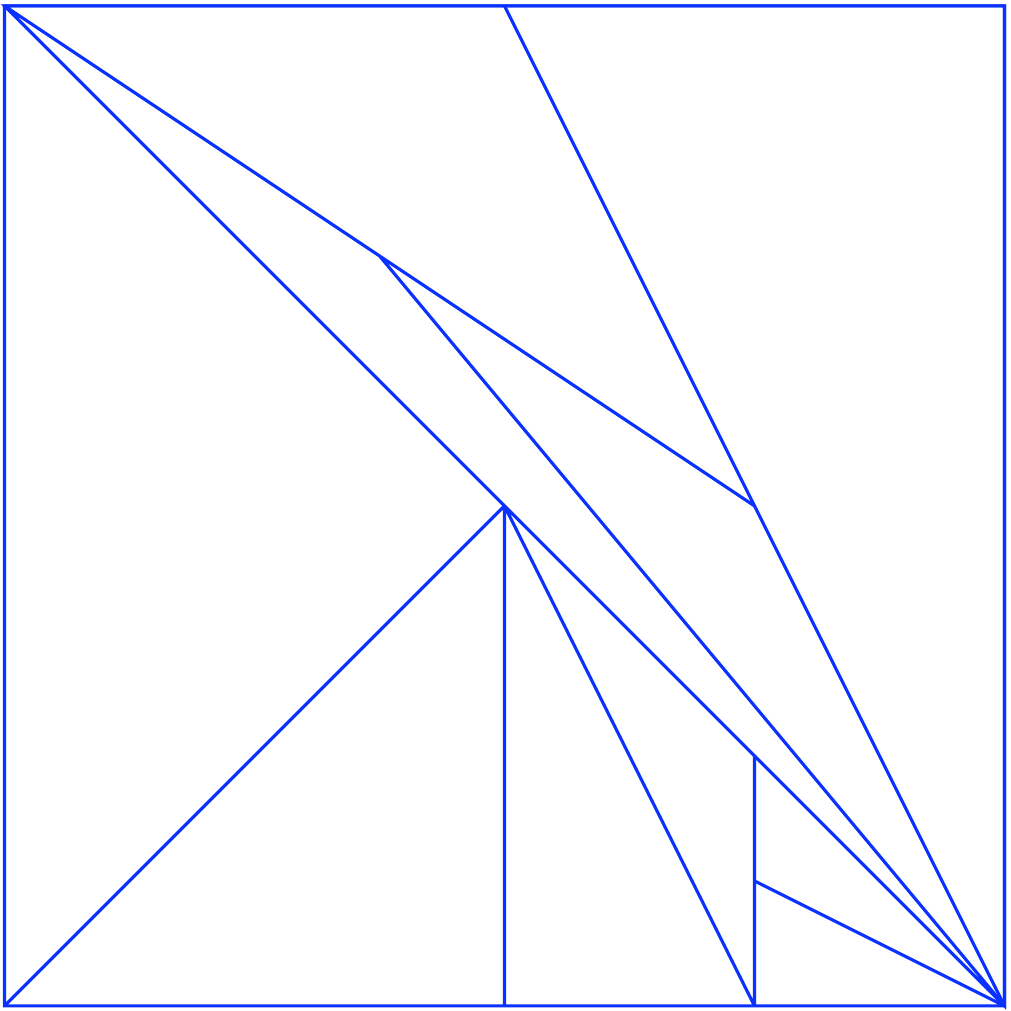}
}
\caption{Anisotropic partitions obtained by rectangle split (left)
and triangle bisection (right)}
\label{CManisopart}
\end{figure}
The choice between the different splitting scenarios is done by
a {\it decision rule} which depends on the function $f$.
A typical decision rule is to select the split which
best decreases the local error. The greedy refinement algorithm
therefore reads as follows:
\begin{enumerate}
\item
Initialization: $\cT_{N_0}=\cD_0$ with $N_0:=\#(\cD_0)$.\vspace{0.2cm}
\item
Given $\cT_N$ select $T\in\cT_N$ that maximizes $e_{m,T}(f)_T$.\vspace{0.2cm}
\item
Use the decision rule in order to select the type of split to be performed on $T$.\vspace{0.2cm}
\item
Split $T$ into $K$ elements to obtain $\cT_{N+K-1}$ and return to step 2.
\end{enumerate}
Intuitively, the error equidistribution is ensured by
selecting the element that maximizes the local
error, while the role of the decision rule is to optimize
the shape of the generated elements.

The problem is now to understand if the piecewise
polynomial approximations generated by such 
refinement algorithms satisfy similar convergence
properties as those which were established  
in \S 4 and \S 5 when using optimally adapted 
partitions. We first study the anisotropic refinement algorithm 
for the simple case of piecewise constant
on rectangles, and we give a complete proof of 
its optimal convergence properties. We then
present the anisotropic refinement algorithm for
piecewise polynomials on triangulations,
and give without proof the available results
on its optimal convergence properties.

\begin{remark}
Let us remark that in contrast to the refinement
algorithm discussed in \S 2.3 and 3.3,
the partition $\cT_N$ may not anymore be identified to a 
finite subtree within a fixed infinite master tree $\cM$. Instead, 
for each $f$, the decision rule defines an infinite master
tree $\cM(f)$ that {\it depends} on $f$. The refinement
algorithm corresponds to selecting a finite subtree
within $\cM(f)$. Due to the finite number of splitting 
possibilities for each element, this finite subtree
may again be encoded by a number of bits proportional to $N$.
Similar to the isotropic refinement
algorithm, one may use more sophisticated
techniques such as CART in order to select an optimal partition
of $N$ elements within $\cM(f)$. On the other hand the selection
of the optimal partition within {\it all} possible splitting scenarios
is generally of high combinatorial complexity.
\end{remark}

\begin{remark}
A closely related algorithm was introduced in
\cite{CMdl} and studied in \cite{CMddp}. In this algorithm
every element is a convex polygon which may be
split into two convex polygons by an arbitrary line cut, allowing therefore
an infinite number of splitting scenarios. The
selected split is again typically the one that 
decreases most the local error. Although
this approach gives access to more possibilities
of anisotropic partitions, the analysis of its convergence
rate is still an open problem.
\end{remark}

\subsection{The refinement algorithm for piecewise constants on rectangles}

As in \S 4, we work on the square domain $\Omega=[0,1]^2$
and we consider piecewise constant approximation on 
anisotropic rectangles. At a given stage of the refinement
algorithm, the rectangle $T=I\times J$ that maximizes $e_{1,T}(f)_p$
is split either vertically or horizontally,
which respectively corresponds to split one interval among $I$ and $J$ into two intervals
of equal size and leaving the other interval unchanged.
As already mentionned in the case of the refinement
algorithm discussed in \S 3.3, we may replace
$e_{1,T}(f)_p$ by the more computable quantity
$\|f-P_{1,T}f\|_p$ for selecting the rectangle $T$
of largest local error. Note that the $L^2(T)$-projection onto
constant functions is simply the average of $f$ on $T$:
$$
P_{1,T}f= \frac 1 {|T|}\int_T f.
$$
If $T$ is the rectangle that is selected for
being split, we denote by
$(T_d,T_u)$ the down and up rectangles
which are obtained by a horizontal split of $T$
and by $(T_l,T_r)$ the left and right rectangles
which are obtained by a vertical split of $T$.
The most natural decision rule for selecting the
type of split to be performed on $T$ is based on comparing
the two quantities
$$
e_{T,h}(f)_{p}:=\(e_{1,T_d}(f)_p^p+e_{1,T_u}(f)_p^p\)^{1/p}\;\;{\rm and}\;\;
e_{T,v}(f)_{p}:=\(e_{1,T_l}(f)_p^p+e_{1,T_r}(f)_p^p\)^{1/p},
$$
which represent the local approximation error after splitting
$T$ horizontally or vertically, with the standard
modification when $p=\infty$. The decision rule
based on the $L^p$ error is therefore :
\nl
\nl
{\it If $e_{T,h}(f)_p\leq e_{T,v}(f)_p$,
then $T$ is split horizontally, otherwise $T$ is split 
vertically.}
\nl
\nl
As already explained, the role of the decision rule is to optimize
the shape of the generated elements. We have seen in
\S 4.1 that in the case where $f$ is an
affine function 
$$
q(x,y)=q_0+q_x x+ q_y y,
$$
the shape of a rectangle $T=I\times J$ which is
optimally adapted to $q$ is given by the relation \iref{CMoptiaspectrect}.
This relation cannot be exactly fullfilled by the rectangles
generated by the refinement algorithm since
they are by construction dyadic type, and in particular
$$
\frac{|I|}{|J|}=2^{j},
$$
for some $j\in\ZZ$.
We can measure the adaptation of $T$
with respect to $q$ by the quantity
\be
a_q(T):=\left |\log_2\(\frac {|I|\, |q_x|}{|J|\, |q_y|}\) \right |,
\label{CMoptiaspectmeasrect}
\ee
which is equal to $0$ for optimally adapted rectangles
and is small for ``well adapted'' rectangles. Inspection 
of the arguments leading the heuristic error
estimate \iref{CMidealanirecterr} in \S 4.1 
or to the more rigourous estimate
\iref{CMlimsupsigmaNanisorect}
in Theorem \ref{CMupperecttheo} reveals that
these estimates also hold up to a fixed multiplicative
constant if we use rectangles which have well adapted
shape in the sense that $a_{q_T}(T)$ is uniformly bounded
where $q_T$ is the approximate value of $f$ on $T$.

We notice that for all $q$ such that $q_xq_y\neq 0$, 
there exists at least a dyadic rectangle $T$ such
that $a_T(q)\leq \frac 1 2$. We may therefore hope that
the refinement algorithm leads to optimal error
estimate of a similar form as \iref{CMlimsupsigmaNanisorect},
provided that the decision rule tends to generate
well adapted rectangles. The following result shows that this is 
indeed the case when $f$ is exactly
an affine function, and when using the decision rule
either based on the $L^2$ or $L^\infty$ error.
 
\begin{proposition}
\label{CMadaptrectprop}
Let $q\in \PP_1$ be an affine function and let $T$ be a 
rectangle. If $T$ is split according to the
decision rule either based on the $L^2$ or $L^\infty$
error for this function and if $T'$ a child of $T$ obtained from this splitting,
one then has
\be
a_{q}(T')\leq |a_q(T)-1|.
\label{CMreductratio}
\ee
As a consequence, all rectangles
obtained after sufficiently many refinements
satisfy $a_q(T)\leq 1$.
\end{proposition}

\begin{proof}
We first observe that if $T=I\times J$,
the local $L^{\infty}$ error is given by
$$
e_{1,T}(q)_\infty:=\frac 1 2\max\{|q_x|\, |I|, |q_y|\, |J|\},
$$
and the local $L^2$ error is given by
$$
e_{1,T}(q)_2:=\frac 1 {4\sqrt 3}(q_x^2 |I|^2+q_y^2|J|^2)^{1/2}.
$$
Assume that $T$ is such that $|I|\, |q_x| \geq |J|\, |q_y|$.
In such a case, we find that
$$
e_{T,v}(q)_\infty=\frac 1 2\max\{|q_x|\, |I|, |q_y|\, |J|/2\}=|q_x|\, |I|/2,
$$
and
$$
e_{T,h}(q)_\infty=\frac 1 2\max\{|q_x|\, |I|/2, |q_y|\, |J|\} \leq |q_x|\, |I|/2.
$$
Therefore $e_{T,h}(q)_\infty\leq e_{T,v}(q)_\infty$ which shows that
the horizontal cut is selected
by the decision rule based on the $L^\infty$ error.
We also find that
$$
e_{T,v}(q)_2:=\frac 1 {\sqrt 6}(q_x^2 |I|^2+q_y^2|J|^2/4)^{1/2},
$$
and
$$
e_{T,h}(q)_2:=\frac 1 {\sqrt 6}(q_x^2 |I|^2/4+q_y^2|J|^2)^{1/2},
$$
and therefore $e_{T,h}(q)_2\leq e_{T,v}(q)_2$ which shows that
the horizontal cut is selected
by the decision rule based on the $L^2$ error. Using
the fact that
$$
\log_2\(\frac {|I|\, |q_x|}{|J|\, |q_y|}\) \geq 0,
$$
we find that if $T'$ is any of the two rectangle generated
by both decision rules, we
have $a_{q}(T')=a_q(T)-1$ if $a_q(T)\geq 1$
and $a_{q}(T')=1-a_q(T)$ if $a_q(T)\leq 1$. In the case
where $|I|\, |q_x|< |J|\, |q_y|$, we reach a similar
conclusion observing that the vertical cut is selected
by both decision rules. This proves \iref{CMreductratio}
\hfill $\Box$
\end{proof}

\begin{remark}
We expect that the above result also holds
for the decision rules based on the $L^p$ error
for $p\notin\{2,\infty\}$ which therefore also lead 
to well adapted rectangles when $f$ is an affine. In this sense
all decision rules are equivalent, and it is reasonable
to use the simplest rules based on the $L^2$ or $L^\infty$
error in the refinement algorithm 
that selects the rectangle which maximizes
$e_{1,T}(f)_p$, even when $p$ differs from $2$ or $\infty$.
\end{remark}

\subsection{Convergence of the algorithm}

From an intuitive point of view,
we expect that when we apply the refinement
algorithm to an arbitrary function $f\in C^1(\Omega)$, 
the rectangles tend to adopt a locally
well adapted shape, provided that the 
algorithm reaches a stage where 
$f$ is sufficiently close to an affine
function on each rectangle. However this
may not necessarily happen due to
the fact that we are not ensured that the 
diameter of all the elements tend to $0$
as $N\to \infty$. Note that this is not
ensured either for greedy
refinement algorithms based on isotropic
elements. However, we have used
in the proof of Theorem \ref{CMupperisotheo}
the fact that for $N$ large enough, a fixed portion - say $N/2$ - 
of the elements have arbitrarily
small diameter, which is not anymore guaranteed 
in the anisotropic setting.  

We can actually
give a very simple example
of a smooth function $f$ for which
the approximation produced by
the anisotropic greedy refinement algorithm {\it fails to
converge} towards $f$ due to this problem.
Let $\vp$ be a smooth function of 
one variable which is compactly supported
on $]0,1[$ and positive. We then
define $f$ on $[0,1]^2$ by
$$
f(x,y):=\vp(4x)-\vp(4x-1).
$$
This function is supported in $[0,1/2]\times[0,1]$.
Due to its particular structure, we find that
if $T=[0,1]^2$, the best approximation
in $L^p(T)$ is achieved by the constant $c=0$
and one has
$$
e_{1,T}(f)_p=2^{1/p}\|\vp\|_{L^p}.
$$ 
We also find that $c=0$ is the best approximation
on the four subrectangles $T_d$, $T_u$, $T_l$ and $T_r$
and that $e_{T,h}(f)_p=e_{T_v}(f)_p=e_{1,T}(f)_p$
which means both horizontal and vertical split do not reduce the error.
According to the decision rule, the
horizontal split is selected. We are then
facing a similar situation on $T_d$ and $T_u$
which are again both split horizontally. Therefore,
after $N-1$ greedy refinement steps, the
partition $\cT_N$ consists of rectangles all of the form
$[0,1]\times J$ where $J$ are dyadic intervals,
and the best approximation remains $c=0$ on
each of these rectangles. This shows that
the approximation produced by
the algorithm fails to
converge towards $f$, and the global error remains
$$
e_{1,\cT_N}(f)_p=2^{1/p}\|\vp\|_{L^p},
$$
for all $N>0$.

The above example illustrates the fact that
the anisotropic greedy refinement algorithm
may be defeated by simple functions 
that exhibit an oscillatory behaviour.
One way to correct this defect is to impose
that the refinement of $T=I\times J$ reduces its largest side-length
the case where the refinement suggested by the original decision
rule does not sufficiently reduce the local error. This means
that we modify as follow the decision rule: 
\nl
\nl
{\it Case 1:} if $\min\{e_{T,h}(f)_p,e_{T,v}(f)_p\}\leq \rho e_{1,T}(f)_p$,
then $T$ is split horizontally
if $e_{T,h}(f)_p\leq e_{T,v}(f)_p$ or vertically if $e_{T,h}(f)_p> e_{T,v}(f)_p$.
We call this a {\it greedy split}.
\nl
\nl
{\it Case 2:} if $\min\{e_{T,h}(f)_p,e_{T,v}(f)_p\}> \rho e_{1,T}(f)_p$,
then $T$ is split horizontally
if $|I| \leq |J|$ or vertically if $|I|> |J|$.
We call this a {\it safety split}.
\nl
\nl
Here $\rho$ is a parameter chosen in $]0,1[$. It should not be chosen too
small in order to avoid that all splits are of safety type which would then
lead to isotropic partitions. Our next result shows that 
the approximation produced by the
modified algorithm does converge towards $f$.

\begin{theorem}
\label{CMtheogreedyconv}
For any $f\in L^p(\Omega)$ or in $C(\Omega)$
in the case $p=\infty$, the partitions
$\cT_N$ produced by the modified greedy refinement 
algorithm with parameter $\rho\in ]0,1[$
satisfy 
\be
\label{CMgreedyconvp}
\lim_{N\to +\infty} e_{1,\cT_N}(f)_p=0.
\ee
\end{theorem}

\begin{proof}
Similar to the original refinement procedure,
the modified one
defines a infinite master tree $\cM:=\cM(f)$ with
root $\Omega$ which contains all elements
that can be generated at some stage of the
algorithm applied to $f$. This tree 
depends on $f$, and the partition $\cT_N$
produced by
the modified greedy refinement algorithm may be identified to
a finite subtree within $\cM(f)$.
We denote by $\cD_j:=\cD_j(f)$ the partition consisting
of the rectangles of area $2^{-j}$ in $\cM$, 
which are thus obtained by $j$ refinements of $\Omega$.
This partition also depends on $f$.

We first prove that $e_{1,\cD_j}(f)_p\to 0$ as $j\to \infty$.
For this purpose we split $\cD_j$ into two sets $\cD_j^g$ and $\cD_j^s$.
The first set $\cD_j^g$ consists of the element $T$
for which more than half of the splits that led from $\Omega$ to $T$
were of greedy type. 
Due to the fact that such splits reduce the 
local approximation error by a factor $\rho$ and that this error
is not increased by a safety split, it is easily cheched by an induction argument
that
$$
e_{1,\cD_j^g}(f)_p=\(\sum_{T\in \cD_j^g}e_{1,T}(f)_p^p\)^{1/p} \leq \rho^{j/2}e_{1,\Omega}(f)_p\leq \rho^{j/2}\|f\|_{L^p},
$$
which goes to $0$ as $j\to +\infty$. This result also holds when $p=\infty$.
The second set $\cD_j^s$ consists of the elements $T$
for which at least half of the splits that led from $\Omega$ to $T$
were safety split.  Since two safety splits reduce at least by $2$ the diameter
of $T$, we thus have
$$
\max_{T\in \cD_j^s} h_T \leq  2^{1-j/4},
$$
which goes to $0$ as $j\to +\infty$.
From classical properties of density of piecewise constant functions in $L^p$ spaces
and in the space of continuous functions,
it follows that
$$
e_{1,\cD_j^s}(f)_p  \to 0 \;\;{\rm as}\;\; j\to +\infty.
$$
This proves that 
$$
e_{1,\cD_j}(f)_p=\(e_{1,\cD_j^g}(f)_p^p+e_{1,\cD_j^s}(f)_p^p\)^{1/p}\to 0\;\;{\rm as}\;\; j\to +\infty,
$$
with the standard modification if $p=\infty$.

In order to prove that $e_{1,\cT_N}(f)_p$ also converges to $0$, we first observe that 
since $e_{1,\cD_j}(f)_p\to 0$, it follows that for all $\e>0$, there exists 
only a finite number of $T\in\cM(f)$ such that $e_{1,T}(f)_p\geq \e$. In turn, we find
that
$$
\e(N):=\max_{T\in\cT_N} e_{1,T}(f)_p\to 0\;\; {\rm as}\;\; N\to +\infty.
$$
For some $j>0$, we split $\cT_N$ into two sets $\cT_N^{j+}$
and $\cT_N^{j-}$ which consist of those $T\in\cT_N$
which are in $\cD_l$ for $l\geq j$ and $l<j$ respectively.
We thus have
$$
e_{1,\cT_N}(f)_p = \(e_{1, \cT_N^{j+}}(f)_p^p+e_{1, \cT_N^{j+}}(f)_p^p\)^{1/p}\\
\leq \(e_{1,\cD_j}(f)_p^p+2^j\e(N)^p\)^{1/p}.
$$
Since $e_{1,\cD_j}(f)_p \to 0$ as $j\to +\infty$ and 
$\e(N)\to 0$ as $N\to \infty$, and since $j$ is arbitrary,
this concludes the proof , with the standard modification if $p=\infty$.  \hfill $\Box$
\end{proof}

\subsection{Optimal convergence}

We now prove that
using the specific value $\rho=\frac 1 {\sqrt 2}$
the modified greedy refinement algorithm has optimal convergence properties
similar to \iref{CMlimsupsigmaNanisorect}
in the case where we measure the error in the $L^\infty$ norm.
Similar results can be obtained when the error is measured in $L^p$ with $p<\infty$,
at the price of more technicalities.

\begin{theorem}
\label{CMtheogreedyrect}
There exists a constant $C>0$ such that for any $f\in C^1(\Omega)$, the partition
$\cT_N$ produced by the modified greedy refinement algorithm with parameter $\rho=\frac 1 {\sqrt 2}$
satisfy the asymptotic convergence estimate
\be
\label{CMeqAsympt}
\limsup_{N\to +\infty} \, N^{1/2} e_{1,\cT_N}(f)_\infty \leq C \left\|\sqrt{|\partial_x f\ \partial_y f|}\right\|_{L^2}
\ee
\end{theorem}

\noindent
The proof of this theorem requires a preliminary result.
Here and after, we use the $\ell^\infty$ norm on $\RR^2$ for
measuring the gradient: for $z=(x,y)\in\Omega$
$$
|\nabla f(z)|:=\max \{|\partial_x f(z)|,|\partial_y f(z)|\},
$$
and
$$
\|\nabla f\|_{L^\infty(T)}:= \sup_{z\in T}|\nabla f(z)|=\max\{\|\partial_x f\|_{L^\infty(T)},\|\partial_y f\|_{L^\infty(T)}\}.
$$
We recall that the local $L^\infty$-error on $T$ is given by
$$
e_{1,T}(f)_\infty= \frac 1 2 \left(\max_{z\in T} f(z) - \min_{z\in T} f(z) \right).
$$
For the sake of simplicity we define 
$$
e_T(f) := \max_{z\in T} f(z) - \min_{z\in T} f(z) = 2 e_{1,T}(f)_\infty,
$$
and 
$$
e_{T,h} f : =2 e_{T,h}(f)_\infty, \quad e_{T,v}(f) := 2 e_{T,v}(f)_\infty.
$$
We also recall from the proof of Theorem \ref{CMtheogreedyconv} that
$$
\e(N):=\max_{T\in\cT_N} e_T(f) \to 0 \;\; {\rm as}\;\; N\to +\infty.
$$
Finally we sometimes use the notation $x(z)$ and $y(z)$ to denote the coordinates 
of a point $z\in\RR^2$.

\begin{lemma}
\label{CMlemmaRobust}
Let $T_0=I_0\times J_0 \in\cT_M$ be a dyadic rectangle obtained 
at some stage $M$ of the refinement algorithm, and let $T=I\times J\in \cT_N$
be a dyadic rectangle obtained at some later stage $N>M$ and such
that $T\subset T_0$. We then have 
$$
|I|\geq \min\left \{|I_0|, \frac {\ve(N)} {4\|\nabla f \|_{L^\infty(T_0)}} \right \} \text{ and } 
|J|\geq \min\left \{|J_0|, \frac {\ve(N)} {4\|\nabla f \|_{L^\infty(T_0)}}\right \}.
$$
\end{lemma}
\proof
Since the coordinates $x$ and $y$ play symmetrical roles, it suffices to prove the first inequality.
We reason by contradiction. If the inequality does not hold, there exists 
a rectangle $T'=I'\times J'$ in the chain that led from $T_0$ to $T_1$
which is such that
$$
|I'| < \frac  {\ve(N)} {2\|\nabla f \|_{L^\infty(T_0)}},
$$
and such that $T'$ is split vertically by the algorithm. If this was a safety split, 
we would have that $|J'|\leq |I'|$
and therefore 
$$
e_{T'}(f) \leq (|I'|+|J'|) \|\nabla f \|_{L^\infty(T)} \leq 2|I'|\|\nabla f \|_{L^\infty(T)} < \ve(N),
$$ 
which is a contradiction, since all ancestors of $T$ should satisfy $e_{T'}(f)\geq \ve(N)$. 
Hence this split was necessarily a greedy split.

Let $z_m:={\rm Argmin}_{z\in T'} f(z)$ and $z_M:={\rm Argmax}_{z\in T'} f(z)$, and let $T''$ be the child of $T'$ (after the vertical split) containing $z_M$. Then $T''$ also contains a point $z'_m$ such that $|x(z'_m) - x(z_m)|\leq |I'|/2$ and $y(z'_m) = y(z_m)$. It follows that
$$
\begin{array}{ll}
e_{T',v}(f) &=e_{T''}(f) \\
& \geq f(z_M) - f(z'_m)\\
 & \geq f(z_M) - f(z_m)- \|\partial_x f\|_{L^\infty(T')} |I'|/2\\
 & \geq  e_{T'}(f) - \ve(N)/4\\ 
&  \geq \frac 3 4 e_{T'}(f) \\
& > \rho e_{T'}(f).
\end{array}
$$
The error was therefore insufficiently reduced which contradicts a greedy split.
\hfill $\Box$
\nl
\nl
{\bf Proof of Theorem \ref{CMtheogreedyrect}:}
We consider a small but fixed $\delta >0$, we define $h(\delta)$ as the maximal $h>0$ such that
$$
\forall z,z'\in \Omega,\ |z-z'|\leq 2h(\delta) \Rightarrow |\nabla f(z)-\nabla f(z')|\leq  \delta.
$$
For any rectangle $T=I\times J\subset \Omega$, we thus have 
\be
\label{CMeqhdelta}
\begin{array}{ll}
&e_{T}(f) \geq (\|\partial_x f\|_{L^\infty(T)}-\delta) \min\{h(\delta),|I|\}, \\
&e_{T}(f) \geq (\|\partial_y f\|_{L^\infty(T)}-\delta) \min\{h(\delta),|J|\}.
\end{array}
\ee
Let $\delta>0$ and $M=M(f,\delta)$ be the smallest value of $N$ such that $\e(N) < 9\delta h(\delta)$.
For all $N\geq M$, and therefore $\e(N) < 9\delta h(\delta)$, we consider the 
partition $\cT_N$ which is a refinement of $\cT_{M}$. For any rectangle $T_0=I_0\times J_0 \in \cT_M$,
we denote by $\cT_N(T_0)$ the set of rectangles of $\cT_N$ that are contained $T_0$.
We thus have
$$
\cT_N:=\cup_{T_0\in\cT_M}\cT_N(T_0),
$$
and $\cT_N(T_0)$ is a partition of $T_0$. We shall next bound by below
the side length of $T=I\times J$ contained in $\cT_N(T_0)$,
distinguishing different cases depending on the behaviour of $f$ on $T_0$.
\nl
\nl
{\it Case 1.} If $T_0\in\cT_M$ is such that $\|\nabla f\|_{L^\infty(T_0)} \leq 10 \delta$,
then a direct application of Lemma \ref{CMlemmaRobust} shows that for all $T=I\times J\in \cT_N(T_0)$ we have 
\be
\label{CMeqFlat}
|I|\geq \min\left \{|I_0|, \frac {\ve(N)} {40\delta}\right \} \text{ and } |J|\geq \min\left \{|J_0|, \frac {\ve(N)} {40\delta}\right \}
\ee
\nl
{\it Case 2.} If $T_0\in\cT_M$ is such that
$\|\partial_x f\|_{L^\infty(T_0)} \geq 10 \delta$ and $\|\partial_y f\|_{L^\infty(T_0)} \geq 10 \delta$,
we then claim that for all $T=I\times J\in \cT_N(T_0)$ we have
\be
\label{CMeqAffine}
|I| \geq \min\left\{|I_0|, \frac{\ve(N)} {20\|\partial_x f\|_{L^\infty(T_0)}}\right\} \text{ and } 
|J| \geq \min\left\{|J_0|, \frac{\ve(N)} {20\|\partial_x f\|_{L^\infty(T_0)}}\right\},
\ee
and that furthermore 
\be
|T_0| \ \|\partial_x f\|_{L^\infty(T_0)} \|\partial_y f\|_{L^\infty(T_0)} \leq \left (\frac {10}9\right )^2 \int_{R^*} |\partial_x f\ \partial_y f| dx dy.
\label{CMintegralstatement}
\ee
This last statement easily follows by the following observation: combining 
\iref{CMeqhdelta} with the fact that $\|\partial_x f\|_{L^\infty(T_0)} \geq 10 \delta$ and $\|\partial_y f\|_{L^\infty(T_0)} \geq 10 \delta$
and that $e_{T}(f)\leq \e(N) \leq 9\delta h(\delta)$, we find that for all $z\in T_0$
$$
|\partial_x f(z)|\geq \|\partial_x f\|_{L^\infty(T_0)} - \delta \geq \frac 9 {10}\|\partial_x f\|_{L^\infty(T_0)},
$$
and 
$$
|\partial_y f(z)|\geq \|\partial_y f\|_{L^\infty(T_0)} - \delta \geq \frac 9 {10}\|\partial_y f\|_{L^\infty(T_0)},
$$
Integrating over $T_0$ yields \iref{CMintegralstatement}.
Moreover for any rectangle $T\subset T_0$, we have
\be
\label{CMestimAffine}
\frac 9 {10} \leq \frac {e_{T}(f) }{\|\partial_x f\|_{L^\infty(T_0)}|I|+\|\partial_y f\|_{L^\infty(T_0)} |J|}\leq 1.
\ee
Clearly the two inequalities in \iref{CMeqAffine} are symmetrical, and it suffices to prove the first one. 
Similar to the proof of Lemma \ref{CMlemmaRobust}, we 
reason by contradiction, assuming that a rectangle $T'=I'\times J'$ with $|I'| \|\partial_x f\|_{L^\infty(T_0)}< \frac{\ve(N)}{10}$ 
was split vertically by the algorithm in the chain leading from $T_0$ to $T$. 
A simple computation using inequality \iref{CMestimAffine} shows that 
$$
\frac{e_{T',h}(f) }{e_{T'}(f)}\leq \frac{e_{T',h}(f)}{e_{T',v}(f)} \leq \frac 5 9 \times \frac{1+2\sigma}{1+\sigma/2} \ \text{ with }\ \sigma := \frac{\|\partial_x f\|_{L^\infty(T_0)} |I'|}{\|\partial_y f\|_{L^\infty(T_0)} |J'|}.
$$
In particular if $\sigma<0.2$ the algorithm performs a horizontal greedy split on $T'$, which contradicts our assumption.
Hence $\sigma\geq 0.2$, but this also leads to a contradiction since
$$
\ve(N) \leq  e_{T'}(f) \leq \|\partial_x f\|_{L^\infty(T_0)} |I'|+ \|\partial_y f\|_{L^\infty(T_0)} |J'| 
\leq (1+\sigma^{-1}) \frac {\e(N)} {10} < \ve(N)
$$
\nl
\nl
{\it Case 3.} If $T_0\in\cT_M$ be such that $\|\partial_x f\|_{L^\infty(T_0)} \leq 10 \delta$ and $\|\partial_y f\|_{L^\infty(T_0)} 
\geq 10 \delta$, we then claim that for all $T=I\times J\in \cT_N(T_0)$ we have
\be
\label{CMeq1d}
|I| \geq \min\left \{|I_0|, \frac{\ve(N)} {C\delta}\right \} \text{ and } |J| \geq \min\left \{|J_0|, \frac {\ve(N)} {4 \|\nabla f\|_{L^\infty}}\right \}, \text{ with } C=200,
\ee
with symmetrical result if $T_0$ is such that $\|\partial_x f\|_{L^\infty(T_0)} \geq 10 \delta$ and $\|\partial_y f\|_{L^\infty(T_0)} 
\leq 10 \delta$. The second part of \iref{CMeq1d} is a direct consequence of Lemma
\ref{CMlemmaRobust}, hence we focus on the first part.
Applyting the second inequality of \iref{CMeqhdelta} to $T=T_0$, we obtain
$$
9\delta h(\delta) > e_{T_0}(f) \geq (\|\partial_y f\|_{L^\infty}(T_0)-\delta)\min\{h(\delta),|J_0|\}
\geq 9\delta\min\{h(\delta),|J_0|\},
$$
from which we infer that $|J_0| \leq h(\delta)$. If $z_1,z_2\in T_0$ and $x(z_1) = x(z_2)$ we therefore have 
$|\partial_y f(z_1)| \geq |\partial_y f(z_2)|-\delta$. It follows that for any rectangle $T=I \times J\subset T_0$ we have
\be
\label{CMerror1D}
(\|\partial_y f\|_{L^\infty(T)}-\delta) |J| \leq e_{T}(f) \leq \|\partial_y f\|_{L^\infty(T)} |J|+ 10\delta |I|.
\ee
We then again reason by contradiction, assuming that a rectangle $T'=I'\times J'$ with $|I'| \leq \frac{2\ve(N)}{C\delta}$ was split vertically by the algorithm in the chain leading from $T_0$ to $T$. 
If $\|\partial_y f\|_{L^\infty(T')} \leq 10\delta $, then $\|\nabla f\|_{L^\infty(T')}\leq 10\delta$ and 
Lemma \ref{CMlemmaRobust} shows that $T'$ should not have been split vertically, which is a contradiction.
Otherwise $\|\partial_y f\|_{L^\infty(T')} -\delta \geq \frac 9{10} \|\partial_y f\|_{L^\infty(T')}$, and we obtain
\be
\label{CMestimDy}
(1-20/C)e_{T'}(f) \leq \|\partial_y f\|_{L^\infty(T')} |J'|\leq \frac {10} 9 e_{T'}(f).
\ee
We now consider the children $T'_v$ and $T'_h$ of $T'$
of maximal error after a horizontal and vertical split
respectively, and we inject \iref{CMestimDy} in \iref{CMerror1D}. 
It follows that
$$
\begin{array}{ll}
e_{T',h}(f) &=e_{T'_h}(f)\\
&  \leq  \|\partial_y f\|_{L^\infty(T')} |J'|/2+ 10\delta |I'| \\
&  \leq \frac 5 9 e_{T'}(f)+ 20\ve(N)/C \\
& \leq  (\frac 5 9+20/C)  e_{T'}(f)= \frac{59}{90} e_{T'}(f),
\end{array}
$$
and
$$
\begin{array}{ll}
e_{T',v}(f) & =e_{T'_h}(f)\\
 &\geq (\|\partial_y f\|_{L^\infty(T')}-\delta) |J| \\
&  \geq \frac 9 {10} \|\partial_y f\|_{L^\infty(T')} |J'|  \\
& \geq \frac 9 {10} (1-20/C) e_{T'}(f)= \frac {81}{100}e_{T'}(f).
\end{array}
$$
Therefore $e_{T',v}(f) > e_{T',h}(f)$
which is a contradiction, since our decision rule would then select a horizontal split.
\nl

We now choose $N$ large enough so that 
the minimum in \iref{CMeqFlat}, \iref{CMeqAffine} and \iref{CMeq1d} is
are always equal to the second term. For all  $T\in \cT_N(T_0)$, we respectively find that 
$$
\frac{\ve(N)^2} {|T|}\leq  C
\left\{\begin{array}{ccl}
\delta^2 & \text{ if } & \|\nabla f\|_{L^\infty(T_0)}\leq 10\delta\\
\frac 1 {|T_0|} \int_{T_0} |\partial_x f\ \partial_y f| & \text{ if } &   \|\partial_x f\|_{L^\infty(T_0)} \geq 10 \delta \text{ and } \|\partial_y f\|_{L^\infty(T_0)} \geq 10 \delta\\
 \delta \|\nabla f\|_{L^\infty} & \text{ if } & \|\partial_x f\|_{L^\infty(T_0)} \leq 10 \delta \text{ and } \|\partial_y f\|_{L^\infty(T_0)} \geq 10 \delta \text{ (or reversed).}
\end{array}\right.
$$
with $C= \max\{40^2,\ 20^2 (10/9)^2,\ 800\}=1600$. For $z\in \Omega$, we set
$\psi(z) := \frac 1 {|T|}$ where $T\in \cT_N$ such $z\in T$, and obtain
$$
N=\#(\cT_N)  = \int_{\Omega} \psi \leq C\ve(N)^{-2} \left(  \int_{\Omega} |\partial_x f\ \partial_y f| dx dy+ \delta \|\nabla f\|_{L^\infty}+\delta^2\right).
$$
Taking the limit as $\delta \to 0$, we obtain 
$$
\limsup_{N\to \infty} \, N^{\frac 1 2} \|f-f_N\|_{L^\infty}\leq 20\left \| \sqrt {|\partial_x f\partial_y f|}\right \|_{L^2},
$$
which concludes the proof. \hfill $\Box$

\begin{remark} 
The proof of the Theorem can be adapted 
to any choice of parameter $\rho\in ]\frac 1 2,1[$.
\end{remark}

\subsection{Refinement algorithms for piecewise polynomials on triangles}

As in \S 5, we work on a polygonal domain $\Omega\subset \RR^2$
and we consider piecewise polynomial approximation on 
anisotropic triangles. At a given stage of the refinement
algorithm, the triangle $T$ that maximizes $e_{m,T}(f)_p$
is split from one of its vertices $a_i\subset \{a_1,a_2,a_3\}$
towards the mid-point $b_i$ of the opposite edge $e_i$.
Here again, we may replace
$e_{m,T}(f)_p$ by the more computable quantity
$\|f-P_{m,T}f\|_p$ for selecting the triangle $T$
of largest local error. 

If $T$ is the triangle that is selected for
being split, we denote by
$(T_i',T_i'')$ the two children 
which are obtained when $T$ is split
from $a_i$ towards $b_i$.
The most natural decision rule is based on comparing
the three quantities
$$
e_{T,i}(f)_{p}:=\(e_{m,T_i'}(f)_p^p+e_{m,T_i''}(f)_p^p\)^{1/p},\;\; i=1,2,3.
$$
which represent the local approximation error on $T$
after the three splitting options, with the standard
modification when $p=\infty$. The decision rule
based on the $L^p$ error is therefore :
\nl
\nl
{\centerline {\it $T$ is split from $a_i$ towards $b_i$ for 
an $i$ that minimizes $e_{T,i}(f)_{p}$.}
}
\nl
A convergence analysis of this anisotropic greedy algorithm is
proposed in \cite{CMcm} in the case of piecewise affine functions
corresponding to $m=2$. Since it is by far more involved than
the convergence analysis presented in \S 6.1, \S 6.2 and \S 6.3 
for piecewise constants on rectangles, but possess several similar
features, we discuss without proofs the main available results
and we also illustrate their significance through numerical tests.

No convergence analysis is so far available for the
case of higher order piecewise polynomial $m>2$, beside
a general convergence theorem similar
to Theorem \ref{CMtheogreedyconv}. The algorithm
can be generalized to simplices in dimension
$d>2$. For instance, a $3$-d simplex
can be split into two simplices by a plane connecting 
one of its edges to the midpoint of the opposite edge, allowing
therefore between $6$ possibilities.  

As remarked in the end of \S 6.1, we may use a decision 
rule based on a local error measured in another norm than 
the $L^p$ norm for which we select the element $T$ of largest
local error. In \cite{CMcm}, we considered the ``$L^2$-projection'' decision rule
based on minimizing the quantity 
$$
e_{T,i}(f)_{2}:=\(\|f-P_{2,T_i'}(f)\|_{L^2(T_i')}^2+\|f-P_{2,T_i''}(f)\|_{L^2(T_i'')}^2\)^{1/2},
$$
as well as the ``$L^\infty$-interpolation'' decision rule
based on minimizing the quantity
$$
d_{T,i}(f)_{2}:=\|f-I_{2,T_i'}(f)\|_{L^\infty(T_i')}+\|f-I_{2,T_i''}(f)\|_{L^\infty(T_i'')},
$$
where $I_{2,T}$ denotes the local interpolation operator: $I_{2,T}(f)$ is the affine function
that is equal to $f$ at the vertices of $T$. Using either of these two decision
rules, it is possible to prove that the generated triangles
tend to adopt a well adapted shape.

In a similar way to the algorithm for piecewise constant approximation
on rectangles, we first discuss the behaviour of the algorithm
when $f$ is exactly a quadratic function $q$. Denoting by $\bq$ its
the homogeneous part of degree $2$, we have seen in \S 5.1
that when ${\rm det}(\bq)\neq 0$, the approximation error
on an optimally adapted triangle $T$ is given by
$$
e_{2,T}(q)_p=e_{2,T}(q)_p=|T|^{1/\tau}ÊK_{2,p}(\bq), \;\; \frac 1 \tau:=\frac 1 p+1.
$$
We can measure the adaptation of $T$ with respect to $\bq$ by the
quantity 
$$
\sigma_{\bq}(T)_p=\frac {e_{2,T}(\bq)_p}{|T|^{1/\tau}ÊK_{2,p}(\bq)},
$$
which is equal to $1$ for optimally adapted triangles and small
for ``well adapted'' triangles. It is easy to check that the functions
$(\bq,T)\mapsto \sigma_T(\bq)_p$ are equivalent for all $p$, similar to the 
shape functions $K_{2,p}$ as observed in \S 5.2.

The following theorem, which is a direct consequence
of the results in \cite{CMcm}, shows that
the decision rule tends to make ``most
triangles'' well adapted to $\bq$.

\begin{theorem}
\label{CMadaptriprop}
There exists constants $0<\theta,\mu <1$ and a constant $C_p$
that only depends on $p$ such that the following holds.
For any $\bq\in \HH_2$ such that ${\rm det}(\bq)\neq 0$ and any triangle $T$,
after $j$ refinement levels of $T$ according to the
decision rule, a proportion $1-\theta^j$ of the
$2^j$ generated triangles $T'$ satisfies 
\be
\sigma_{\bq}(T')_p\leq \min\{ \mu^{j}\sigma_{\bq}(T)_p,C_p \}.
\ee
As a consequence, for $j>j(\bq,T)=-\frac {\log C_p- \log(\sigma_{\bq}(T)_p)}{\log \mu}$
one has
\be
\sigma_{\bq}(T')_p \leq C_p,
\ee
for a proportion $1-\theta^j$ of the
$2^j$ generated triangles $T'$.
\end{theorem}

This result should be compared to Proposition
\ref{CMadaptrectprop} in the case of rectangles.
Here it is not possible to show that {\it all} triangles 
become well adapted to $q$, but a proportion 
that tends to $1$ does. It is quite remarkable that
with only three splitting options, the greedy algorithm
manages to drive most of the triangles to 
a near optimal shape. We illustrate this
fact on Figure \ref{CMquadformadaptanis},
in the case of the quadratic form $\bq(x,y):=x^2+100 y^2$,
and an initial triangle $T$ which is equilateral for the
euclidean metric and therefore not well adapted to $\bq$.
Triangles such that $\sigma_{\bq}(T')_2\leq C_2$ 
are displayed in white, 
others in grey. We observe the growth of the proportion
of well adapted triangles as the refinement level
increases.

\begin{figure}[htbp]
\centerline{
\includegraphics[width=3.7cm,height=3.7cm]{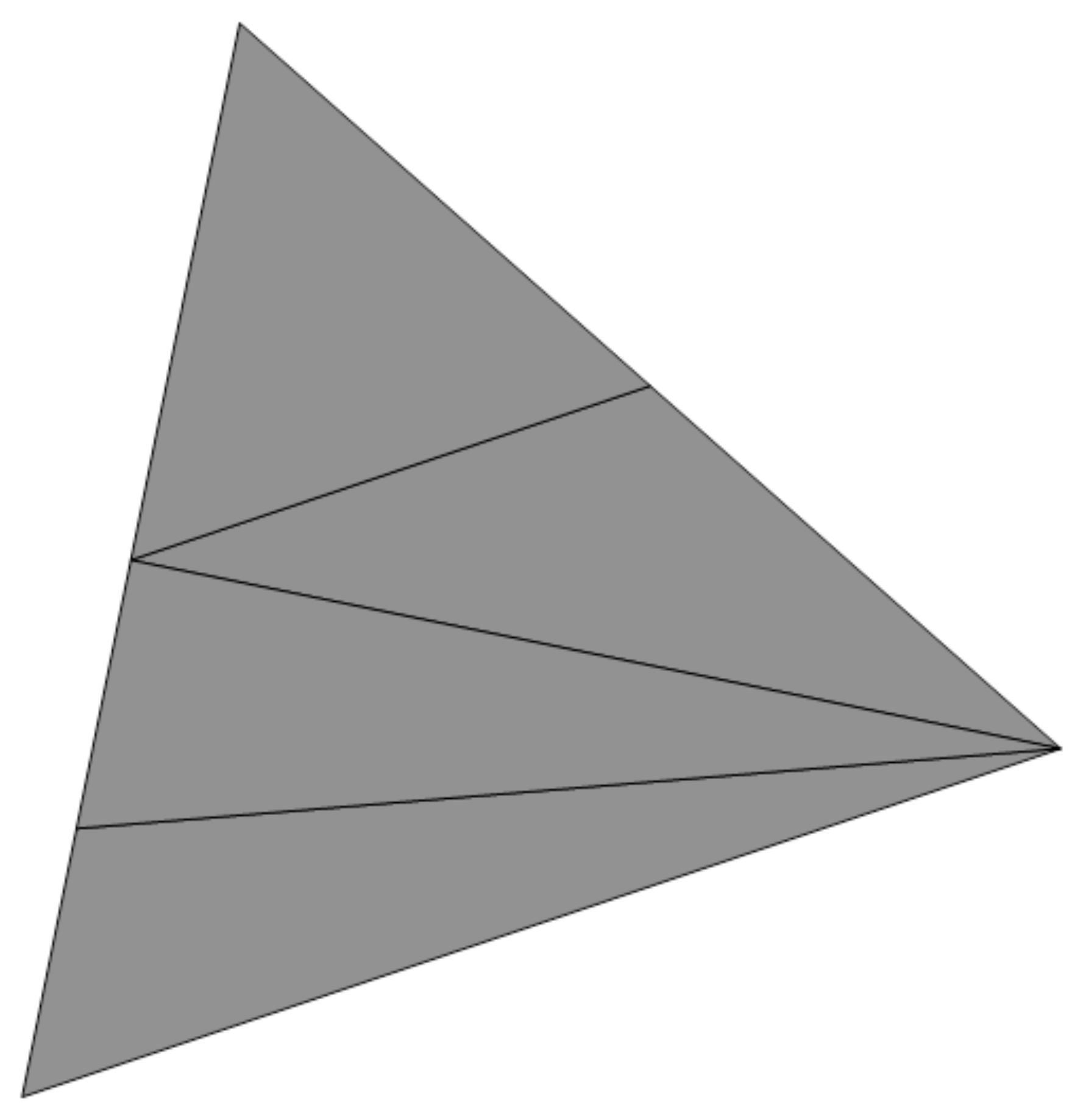}
\includegraphics[width=3.7cm,height=3.7cm]{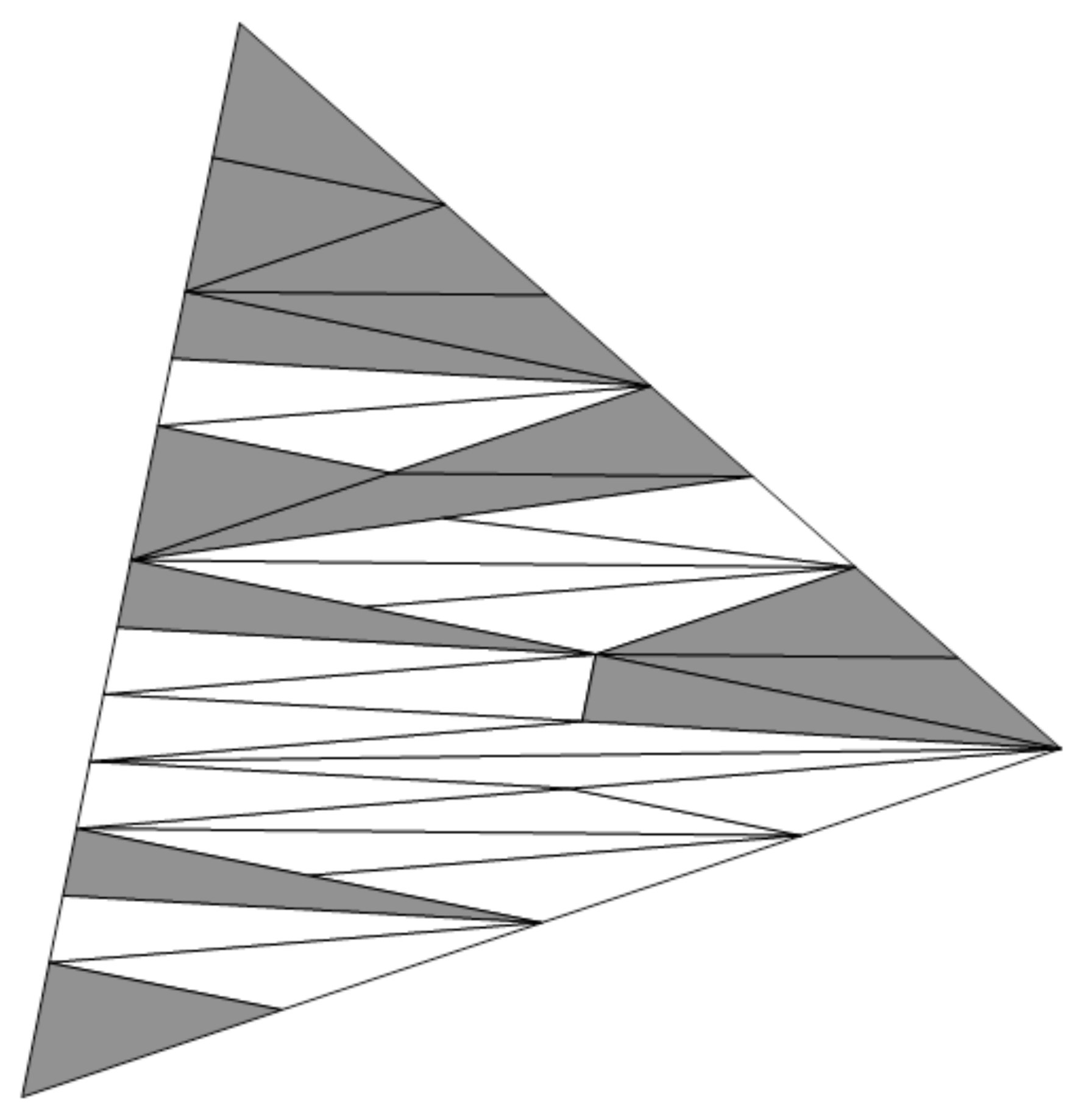}
\includegraphics[width=3.7cm,height=3.7cm]{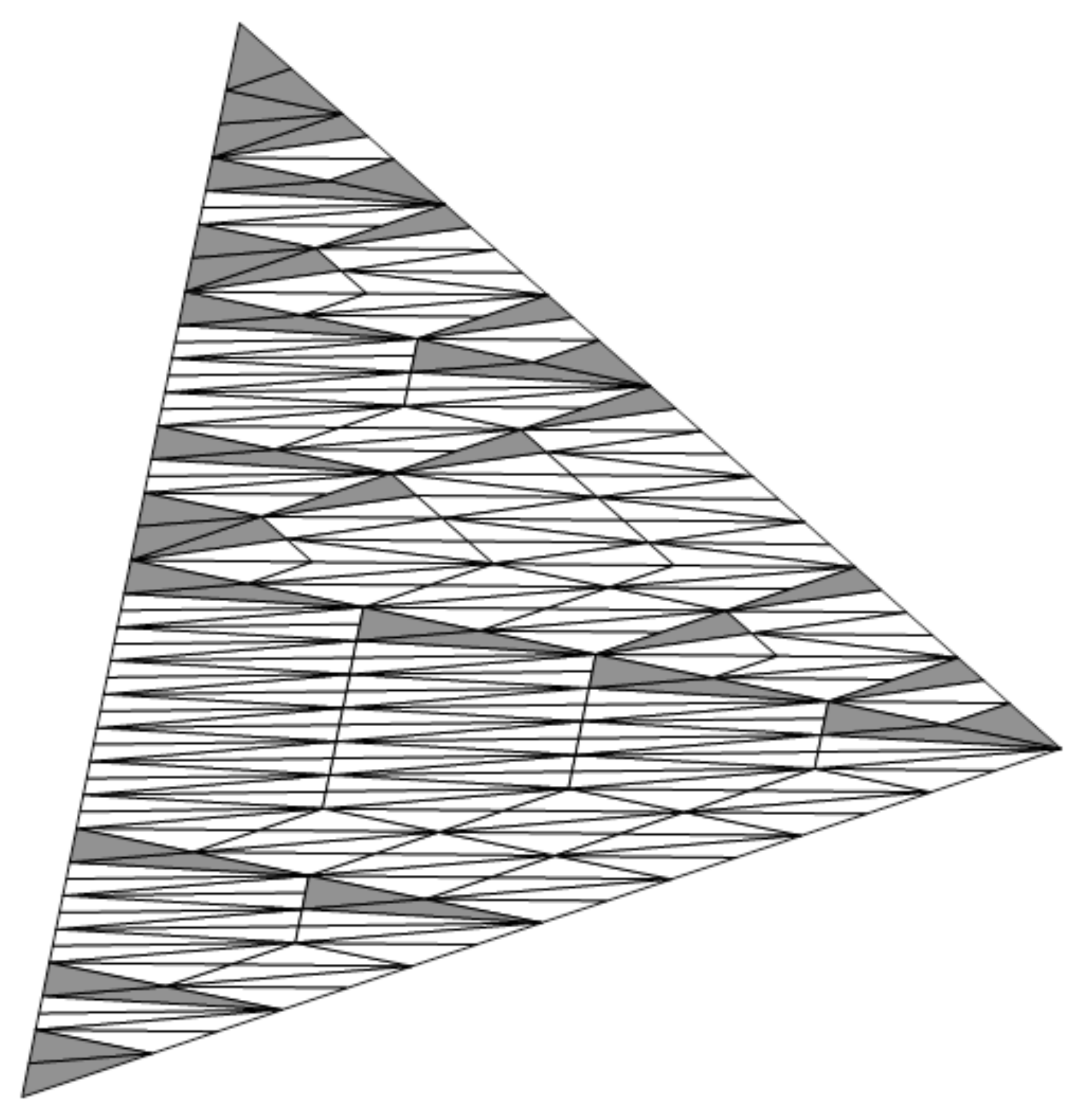}
}
\caption{Greedy refinement for $\bq(x,y):=x^2+100 y^2$: $j=2$ (left), $j=5$ (center), $j=8$ (right).}
\label{CMquadformadaptanis}
\end{figure}

From an intuitive point of view,
we expect that when we apply the anisotropic greedy refinement
algorithm to an arbitrary function $f\in C^2(\Omega)$, 
the triangles tend to adopt a locally
well adapted shape, provided that the 
algorithm reaches a stage where 
$f$ is sufficiently close to an quadratic
function on each triangle. As in the case
of the greedy refinement algorithm for rectangles,
this may not always be the case. It is however
possible to prove that this property holds
in the case of strictly convex or concave functions,
using the ``$L^\infty$-interpolation'' decision rule.
This allows to prove in such a case that the approximation
produced by the anisotropic
greedy algorithm satisfies an optimal convergence
estimate in accordance with Theorem \ref{CMuppersimplextheo}. 
These results from
\cite{CMcm} can be summarized as follows.

\begin{theorem}
\label{CMtheogreedydiam}
If $f$ is a $C^2$ function such that $d^2f(x)\geq \alpha I$ or $d^2f(x)\leq -\alpha I$, 
for all $x\in \Omega$ and some $\alpha>0$, then the triangulation
generated by the anisotropic greedy refinement algorithm (with
the $L^\infty$-interpolation decision rule) satisfies
\be
\lim_{N\to +\infty}\max_{T\in\cT_N} h_T =0.
\label{CMdiamtozero}
\ee
Moreover, there exists a constant $C>0$ such that for any such $f$, the approximation
produced by the anisotropic greedy refinement algorithm 
satisfies the asymptotic convergence estimate
\be
\label{CMgreedtrisympt}
\limsup_{N\to +\infty} \, N e_{2,\cT_N}(f)_p \leq C \left\|\sqrt{|{\rm det}(d^2f)|}\right\|_{L^\tau},\;\; \frac  1 \tau:=\frac 1 p+1.
\ee
\end{theorem}

For a non-convex function, we are not ensured that the 
diameter of the elements tends to $0$
as $N\to \infty$, and similar to the greedy algorithm 
for rectangles, it is possible to produce examples
of smooth functions $f$ for which
the approximation produced by
the anisotropic greedy refinement algorithm fails to
converge towards $f$. A natural way to modify
the algorithm in order to circumvent this problem
is to impose a type of splitting that tend to 
diminish the diameter, such as longest edge or newest vertex bisection, in
the case where the refinement suggested by the original decision
rule does not sufficiently reduce the local error. This means
that we modify as follow the decision rule: 
\nl
\nl
{\it Case 1:} if $\min\{e_{T,1}(f)_p,e_{T,2}(f)_p,e_{T,3}(f)_p\}\leq \rho e_{2,T}(f)_p$,
then split $T$  from $a_i$ towards $b_i$ for 
an $i$ that minimizes $e_{T,i}(f)_{p}$. We call this a {\it greedy split}.
\nl
\nl
{\it Case 2:} if $\min\{e_{T,1}(f)_p,e_{T,2}(f)_p,e_{T,3}(f)_p\}> \rho e_{2,T}(f)_p$,
then split $T$ from the most recently generated vertex or towards its longest
edge in the euclidean metric. We call this a {\it safety split}.
\nl

As in modified greedy algorithm for rectangles, $\rho$ is a parameter chosen in $]0,1[$
that should not be chosen too
small in order to avoid that all splits are of safety type which would then
lead to isotropic triangulations.
It was proved in \cite{CMcdhm} that
the approximation produced by this
modified algorithm does converge towards $f$
for any $f\in L^p(\Omega)$. The following result
also holds for the
generalization of this algorithm to higher degree 
piecewise polynomials.

\begin{theorem}
\label{CMtheogreedytriconv}
For any $f\in L^p(\Omega)$ or in $C(\Omega)$
in the case $p=\infty$, the approximations produced by the modified
anisotropic  greedy refinement 
algorithm with parameter $\rho\in ]0,1[$
satisfies
\be
\label{CMgreedyconvptri}
\limsup_{N\to +\infty} e_{2,\cT_N}(f)_p=0.
\ee
\end{theorem}

Similar to Theorem \ref{CMtheogreedyrect}, we may expect that the modified
anisotropic greedy refinement algorithm satisfies optimal convergence
estimates for all $C^2$ function, but this is an open question at the present
stage.
\nl
\nl
{\bf Conjecture.} {\it There exists a constant $C>0$ and $\rho^*\in ]0,1[$
such that for any $f\in C^2$, the approximation 
produced by the modified anisotropic greedy refinement algorithm 
with parameter $\rho\in ]\rho^*,1[$ satisfies the asymptotic convergence estimate
{\rm \iref{CMgreedtrisympt}}.}
\nl
\nl
We illustrate the performance of the anisotropic greedy refinement algorithm
algorithm for a function $f$ which has a sharp transition along a curved edge. Specifically
we consider
$$
f(x,y)=f_\delta(x,y) := g_\delta(\sqrt{x^2+y^2}),
$$ 
where $g_\delta$
is defined by
$g_\delta(r)= \frac{5-r^2} 4$ for $0\leq r\leq 1$, 
$g_\delta(1+\delta+r)=-\frac{5-(1-r)^2} 4$ for $r\geq 0$, $g_{\delta}$ is a polynomial of degree $5$ on $[1,1+\delta]$
which is determined by imposing that
$g_\delta$ is globally $C^2$.
The parameter $\delta$ therefore measures the sharpness of the transition. We
apply the anisotropic refinement algorithm based on splitting the
triangle that maximizes the local $L^2$-error
and we therefore measure the global error in $L^2$.

Figure \ref{CMfigtrianis} displays the triangulation $\cT_{10000}$ obtained after $10000$ steps of
the algorithm for $\delta=0.2$. In particular, triangles $T$ such that $\sigma_\bq(T)_2\leq C_2$ 
- where $\bq$ is the quadratic form associated with $d^2f$ measured at the barycenter of $T$ - 
are displayed in white, others in grey. As expected, most triangles are of the
first type therefore well adapted to $f$. We also display
on this figure the adaptive isotropic triangulation produced by the greedy tree algorithm
based on newest vertex bisection for the same number of triangles.

\begin{figure}[htbp]
\centerline{
\includegraphics[width=4cm,height=4cm]{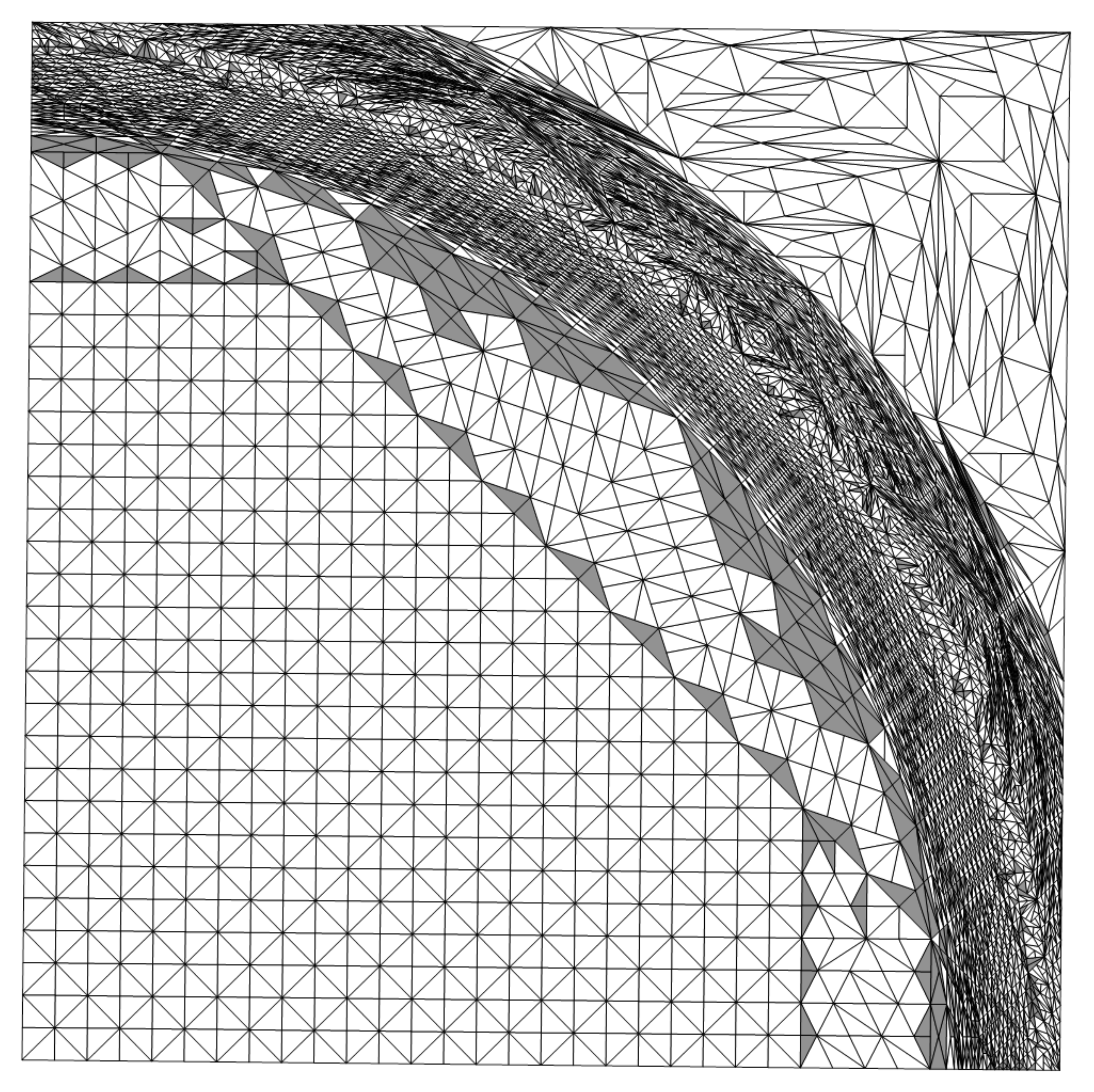}
\includegraphics[width=3.9cm,height=3.9cm]{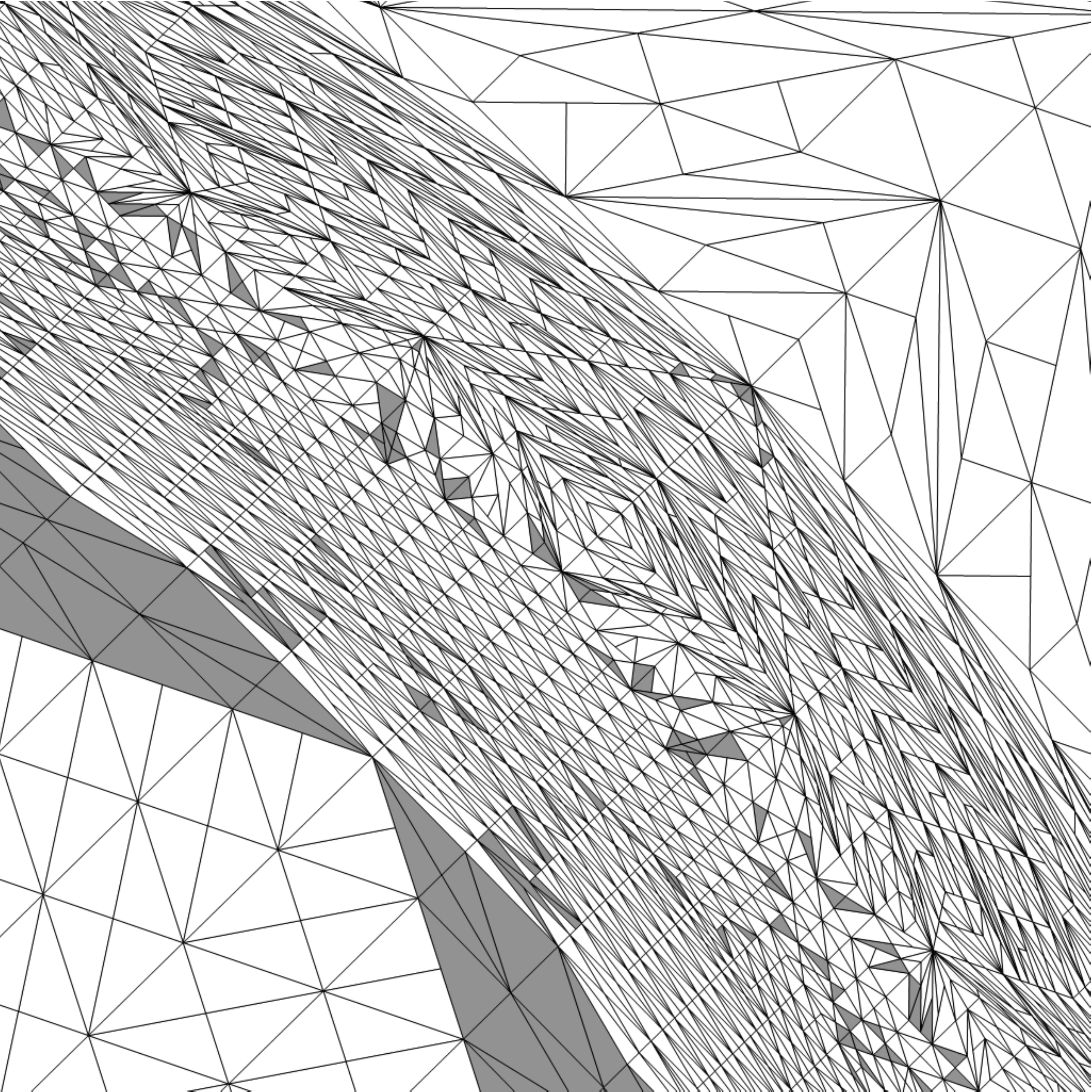}
\includegraphics[width=4cm,height=4cm]{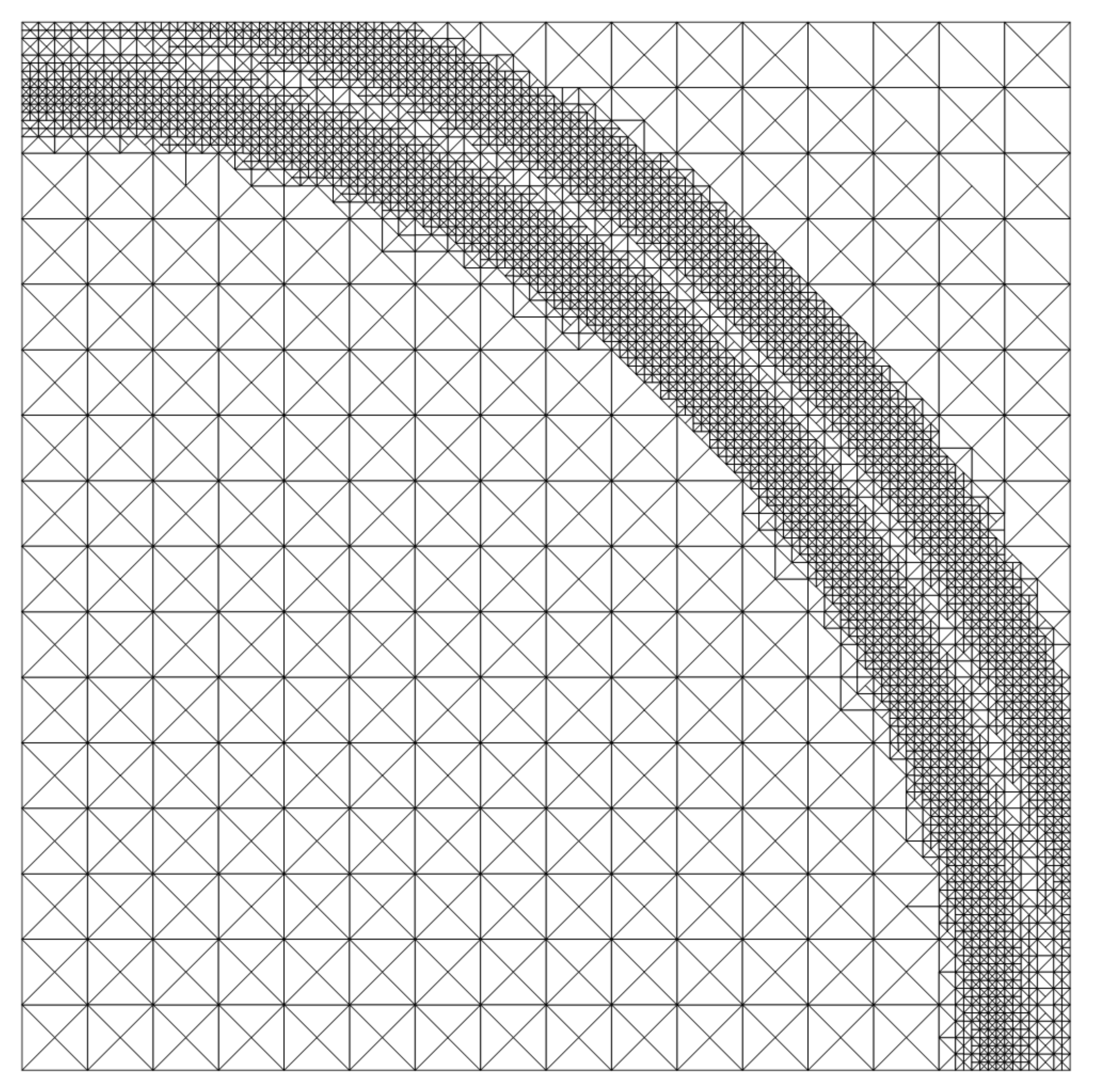}
}
\caption{The anisotropic triangulation $\cT_{10000}$ (left),  detail (center), isotropic triangulation (right).}
\label{CMfigtrianis}
\end{figure}

Since $f$ is a $C^2$ function, approximations by uniform, adaptive 
isotropic and adaptive anisotropic triangulations all yield the convergence rate $\cO(N^{-1})$.
However the constant 
$$
C:=\limsup_{N\to +\infty} Ne_{2,\cT_N}(f)_2,
$$ 
strongly differs depending on the algorithm
and on the sharpness of the transition.
We denote by $C_U$, $C_I$ and $C_A$ the empirical constants
(estimated by $N \|f-f_N\|_2$ for $N=8192$)
in the uniform, adaptive isotropic and adaptive anisotropic case respectively, and by
$U(f):=\|d^2f\|_{L^{2}}$,
$I(f):=\|d^2f\|_{L^{2/3}}$ and
$A(f):=\|\sqrt{|{\rm det}(d^2f)|}\|_{L^{2/3}}$
the theoretical constants suggested by the
convergence estimates. We observe on Figure \ref{CMfigtricomp}.
that $C_U$
and $C_I$ grow in a similar way as $U(f)$ and $I(f)$
as $\delta\to 0$ (a detailed computation shows that $U(f)\approx 10.37  \delta^{-3/2}$
and $I(f)\approx 14.01 \delta^{-1/2}$).
In contrast $C_A$ and $A(f)$ remain uniformly bounded, a fact
which is in accordance with Theorem \ref{CMthreg}
and reflects the superiority of anisotropic triangulations
as the layer becomes thinner and $f_\delta$ tends to 
a cartoon function.

\begin{figure}[htbp]
$$
\begin{array}{c|c|c|c|c|c|c|}
\delta & U(f)  & I(f) & A(f) & C_U & C_I   & C_A
\\
\hline
0.2	& 103  & 27 & 6.75 & 7.87	& 1.78 &  0.74\\
0.1 & 602 	& 60 & 8.50 & 23.7	& 2.98 & 0.92\\
0.05& 1705 & 82 & 8.48 & 65.5	&	4.13 & 0.92\\
0.02& 3670 & 105 & 8.47 & 200	&	6.60 & 0.92
\end{array}
$$
\caption{Comparison between theoretical and empirical convergence constants for uniform, adaptive isotropic
and anisotropic refinements, and for different values of $\delta$.}
\label{CMfigtricomp}
\end{figure}

We finally apply the anisotropic refinement algorithm 
to the numerical image of Figure \ref{CMfigimage}
based on the discretized $L^2$ error and using $N=2000$ triangles.
We observe on Figure \ref{CMfigimageanis} that 
the ringing artefacts produced by the 
isotropic greedy refinement algorithm near the
edges are strongly reduced. This is due to the
fact that the anisotropic greedy refinement algorithm
generates long and thin triangles aligned with the edges.
We also observe that the quality is slightly improved
when using the modified algorithm.
Let us mention that a different approach to the approximation
of image by adaptive anisotropic triangulations was proposed in \cite{CMddfi}.
This approach is based on a {\it thinning} algorithm, which starts
from a fine triangulation and iteratively coarsens it by point removal.
The use of adaptive adaptive anisotropic partitions has also strong
similarities with thresholding methods based on
representations which have more directional
selectivity than wavelet decompositions \cite{CMacddm,CMcd,CMd,CMlm}.
It is not known so far if these methods satisfy 
asymptotic error estimates of the same form as \iref{CMgreedtrisympt}.

\begin{figure}[htbp]
\centerline{
\includegraphics[width=5cm,height=5cm]{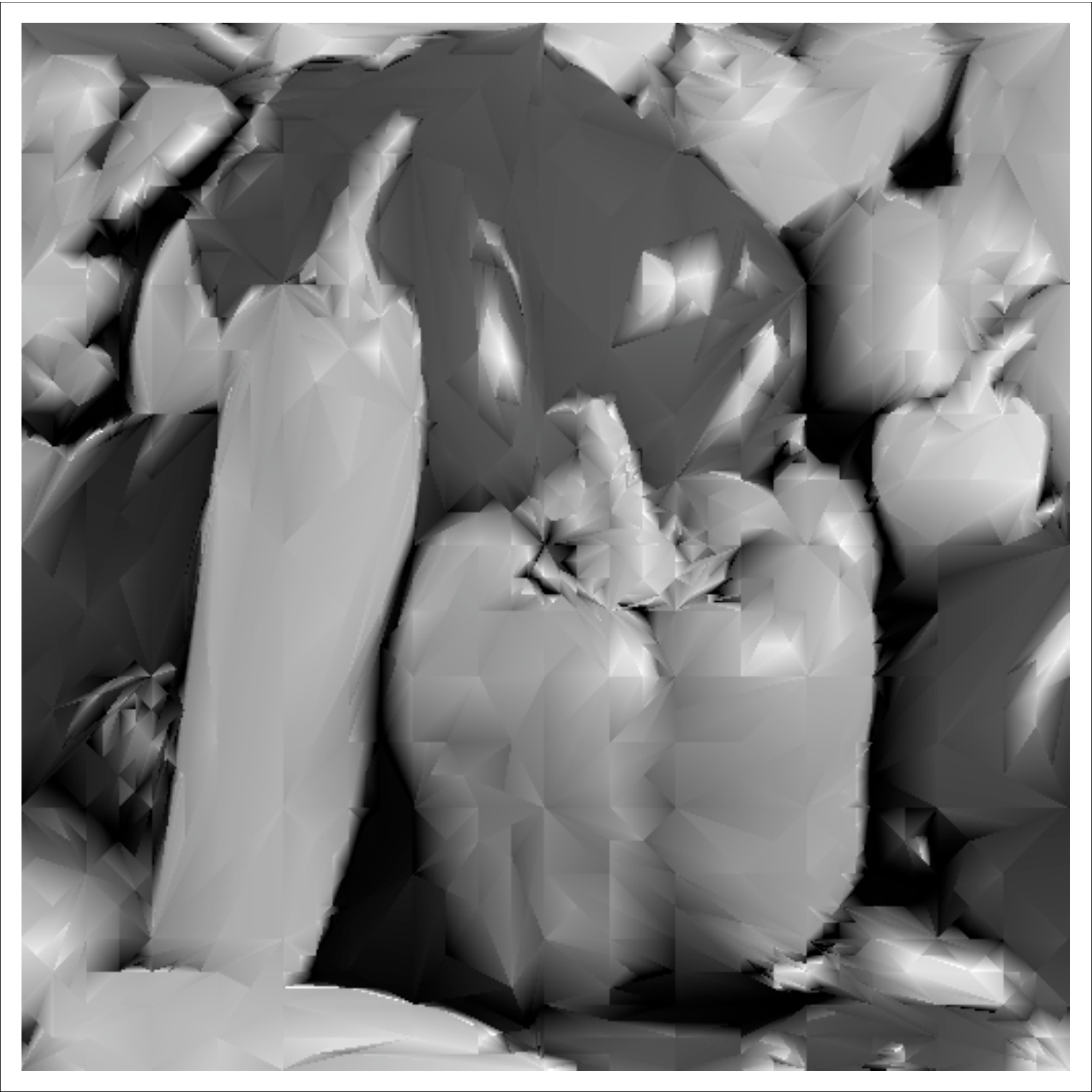}
\hspace{1cm}
\includegraphics[width=5cm,height=5cm]{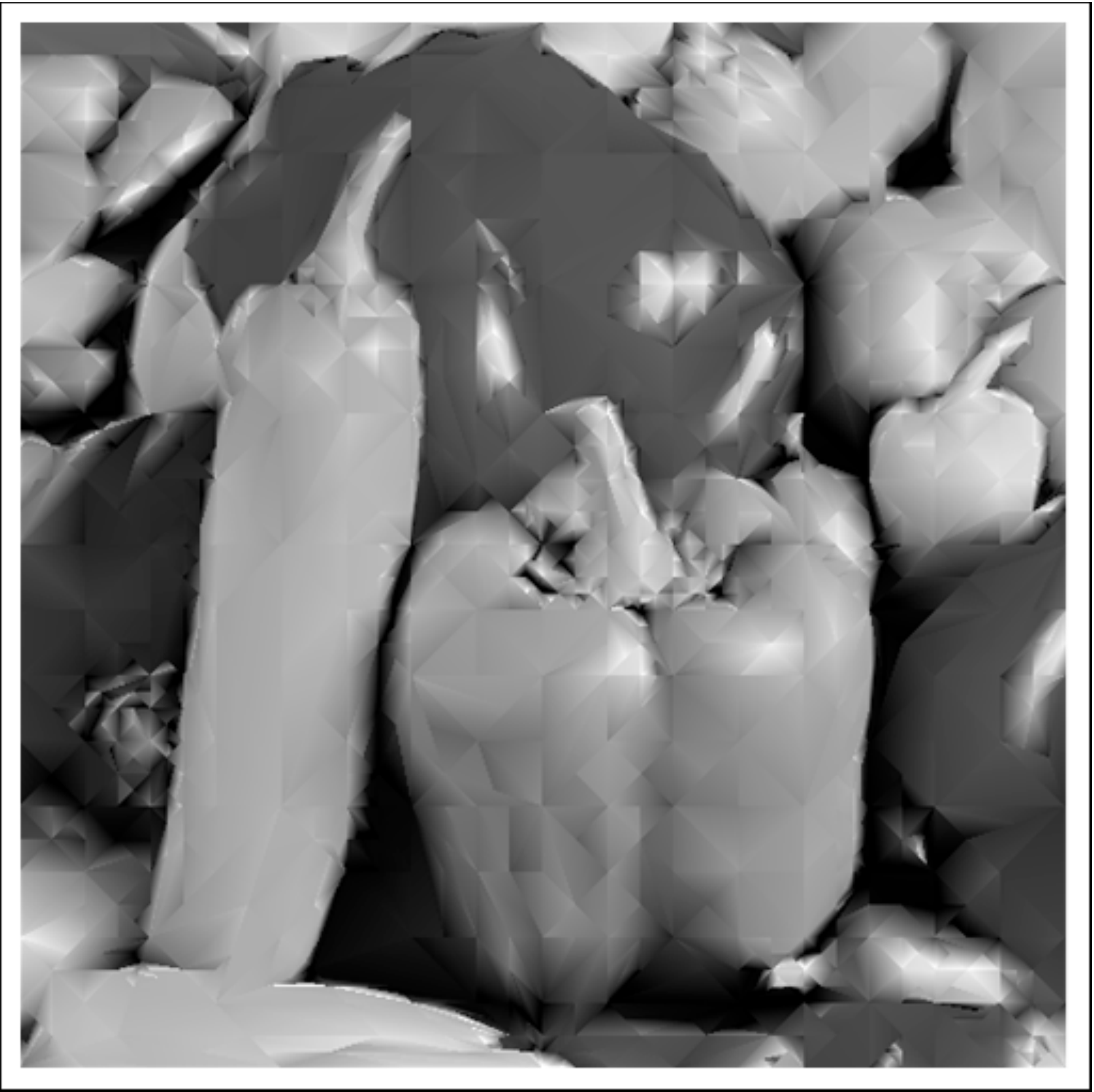}
}
\caption{Approximation by $2000$ anisotropic triangles obtained by the greedy (left) and modified (right) algorithm.}
\label{CMfigimageanis}
\end{figure}

\input{referenc}

\end{document}

%% file: referenc.tex
%
%
%